\documentclass[a4paper,11pt]{article}
\usepackage{amsmath,amsthm}
\usepackage{aligned-overset}
\usepackage{authblk}
\usepackage[style=numeric-comp,maxbibnames=99,bibencoding=utf8,firstinits=true]{biblatex}
\usepackage{booktabs}
\usepackage{cancel}
\usepackage[hmargin={26mm,26mm},vmargin={30mm,35mm}]{geometry}
\usepackage{nicefrac}
\usepackage{paralist}
\usepackage[colorlinks, allcolors=blue]{hyperref}
\usepackage{subcaption}
\usepackage[compact,small]{titlesec}
\usepackage{tikz}
\usepackage{tikz-cd}
\usepackage{pgfplotstable}
\usepackage{newtxmath}
\usepackage{newtxtext}
\usepackage{enumitem}

\usepackage{graphicx}
\graphicspath{{./}}

\bibliography{ddr-manifold.bib}

\newcommand{\email}[1]{\href{mailto:#1}{#1}}

\allowdisplaybreaks

\definecolor{darkbrown}{HTML}{996633}
\newcommand{\logLogSlopeTriangle}[5]
{

    \pgfplotsextra
    {
        \pgfkeysgetvalue{/pgfplots/xmin}{\xmin}
        \pgfkeysgetvalue{/pgfplots/xmax}{\xmax}
        \pgfkeysgetvalue{/pgfplots/ymin}{\ymin}
        \pgfkeysgetvalue{/pgfplots/ymax}{\ymax}

        \pgfmathsetmacro{\xArel}{#1}
        \pgfmathsetmacro{\yArel}{#3}
        \pgfmathsetmacro{\xBrel}{#1-#2}
        \pgfmathsetmacro{\yBrel}{\yArel}
        \pgfmathsetmacro{\xCrel}{\xArel}

        \pgfmathsetmacro{\lnxB}{\xmin*(1-(#1-#2))+\xmax*(#1-#2)} 
        \pgfmathsetmacro{\lnxA}{\xmin*(1-#1)+\xmax*#1} 
        \pgfmathsetmacro{\lnyA}{\ymin*(1-#3)+\ymax*#3} 
        \pgfmathsetmacro{\lnyC}{\lnyA+#4*(\lnxA-\lnxB)}
        \pgfmathsetmacro{\yCrel}{\lnyC-\ymin)/(\ymax-\ymin)} 

        \coordinate (A) at (rel axis cs:\xArel,\yArel);
        \coordinate (B) at (rel axis cs:\xBrel,\yBrel);
        \coordinate (C) at (rel axis cs:\xCrel,\yCrel);

        \draw[black]   (A)-- node[pos=0.5,anchor=north] {\scriptsize{1}}
                    (B)-- 
                    (C)-- node[pos=0.,anchor=west] {\scriptsize{\color{#5}#4}} 
                    (A);
    }
}

\newtheorem{theorem}{Theorem}
\newtheorem{proposition}[theorem]{Proposition}
\newtheorem{lemma}[theorem]{Lemma}

\theoremstyle{remark}
\newtheorem{remark}[theorem]{Remark}
\theoremstyle{definition}
\newtheorem{assumption}[theorem]{Assumption}
\newtheorem{definition}[theorem]{Definition}

\DeclareRobustCommand{\bvec}[1]{\boldsymbol{#1}}
\newcommand{\ul}[1]{\underline{{#1}}}

\newcommand{\uvec}[1]{\underline{\bvec{#1}}}

\newcommand{\norm}[2]{\Vert #2\Vert_{#1}}
\newcommand{\seminorm}[2]{\vert #2\vert_{#1}}

\newcommand{\Real}{\mathbb{R}}
\newcommand{\Natural}{\mathbb{N}}
\newcommand{\Zintegers}{\mathbb{Z}}

\newcommand{\st}{\,:\,}

\DeclareMathOperator{\tr}{tr}

\newcommand{\DIFF}{\mathrm{d}}
\newcommand{\KOSZUL}{\kappa}
\newcommand{\JAC}{\mathrm{D}}
\newcommand{\DET}{\mathrm{det}}
\newcommand{\INJ}{\mathfrak{I}}
\newcommand{\TRF}{\mathcal{T}}
\newcommand{\Lie}{\mathcal{L}}

\DeclareMathOperator{\GRAD}{\bf grad}
\DeclareMathOperator{\CURL}{\bf curl}
\DeclareMathOperator{\DIV}{div}
\DeclareMathOperator{\ROT}{rot}

\newcommand{\trimmed}{{-}}

\newcommand{\PL}[2]{\mathcal{P}_{#1}\Lambda^{#2}}
\newcommand{\HL}[2]{\mathcal{H}_{#1}\Lambda^{#2}}
\newcommand{\PLtrimmed}[2]{\mathcal{P}_{#1}^{\trimmed}\Lambda^{#2}}

\newcommand{\lproj}[2]{\pi^{#1}_{r,#2}}
\newcommand{\ltproj}[2]{\pi^{\trimmed,#1}_{r,#2}}

\newcommand{\Mh}{\mathcal{M}_h}
\newcommand{\FM}[1]{\Delta_{#1}}
\newcommand{\pf}{{\partial f}}

\newcommand{\uH}[2]{\uvec{X}_{r,#2}^{#1}}

\newcommand{\uI}[2]{\uvec{I}_{r,#2}^{#1}}

\newcommand{\vol}{\mathrm{vol}}
\newcommand{\spvol}{\widetilde{\vol}}
\newcommand{\spstar}{\widetilde{\star}}
\newcommand{\spdiff}{\widetilde{\DIFF}}
\newcommand{\spdelta}{\widetilde{\delta}}
\newcommand{\spsharp}{{\widetilde{\sharp}}}

\newcommand{\dt}{\DIFF t}
\newcommand{\dx}[1]{\DIFF x^{#1}}
\newcommand{\spE}{\widetilde{E}}
\newcommand{\spB}{\widetilde{B}}

\newcommand{\contr}[1]{\mathsf{i}_{#1}}

\begin{document}

\title{A polytopal discrete de Rham complex on manifolds, with application to the Maxwell equations}
\author[1,2]{J\'er\^ome Droniou\footnote{\email{jerome.droniou@umontpellier.fr}}}
\author[1,3]{Marien Hanot\footnote{\email{marien-lorenzo.hanot@univ-lille.fr}}}
\author[2]{Todd Oliynyk\footnote{\email{todd.oliynyk@monash.edu}}}
\affil[1]{IMAG, Univ. Montpellier, CNRS, Montpellier, France}
\affil[2]{School of Mathematics, Monash University, Australia}
\affil[3]{Univ. Lille, UMR 8524 - Laboratoire Paul Painlevé, CNRS, Inria, Lille, France}

\maketitle

\begin{abstract}
We design in this work a discrete de Rham complex on manifolds. This complex, written in the framework of exterior calculus, has the same cohomology as the continuous de Rham complex, is of arbitrary order of accuracy and, in principle, can be designed on meshes made of generic elements (that is, elements whose boundary is the union of an arbitrary number of curved edges/faces). Notions of local (full and trimmed) polynomial spaces are developed, with compatibility requirements between polynomials on mesh entities of various dimensions. We give explicit constructions of such polynomials in 2D, for some meshes made of curved triangles or quadrangles (such meshes are easy to design in many cases, starting from a few charts describing the manifold). The discrete de Rham complex is then used to set up a scheme for the Maxwell equations on a 2D manifold without boundary, and we show that a natural discrete version of the constraint linking the electric field and the electric charge density is satisfied. Numerical examples are provided on the sphere and the torus, based on bespoke analytical solutions and meshes on each of these manifolds.
  \medskip\\
  \textbf{Key words.} Discrete de Rham complex, manifolds, exterior calculus, Maxwell equations, numerical simulations
  \medskip\\
  \textbf{MSC2020.} 65M50, 65M60, 35Q61, 78M10.

\end{abstract}



\section{Introduction}

The goal of this paper is to design a discrete version of the de Rham complex on manifolds, which can be built on meshes made of generic elements and has the same cohomology as the continuous de Rham complex. This complex can be used to design schemes for partial differential equations (PDEs) which inherit some important stability properties from the continuous models (such as energy bounds or constraint preservation).

Let $\Omega$ be an $n$-dimensional compact Riemannian manifold with or without boundary. In the exterior calculus framework, the continuous de Rham complex on $\Omega$ is
\begin{equation}\label{eq:dR}
  \begin{tikzcd}
   \cdots  \arrow{r}{\DIFF}& H\Lambda^k(\Omega)\arrow{r}{\DIFF} & H\Lambda^{k+1}(\Omega) \arrow{r}{\DIFF}& \cdots,
  \end{tikzcd}
\end{equation}
where $\DIFF$ is the exterior derivative and $H\Lambda^k(\Omega)$ is the space of $L^2$-integrable $k$-forms whose exterior derivative is also $L^2$-integrable. For $\Omega$ a bounded domain of $\Real^n$ with $n=2,3$, through vector proxies this complex is equivalent to the usual $\GRAD$--$\ROT$ complex (in 2D) or $\GRAD$--$\CURL$--$\DIV$ complex (in 3D). The complex property, which reads $\DIFF\circ\DIFF=0$, has important consequences for PDE models based on these operators, such as the preservation of the divergence constraint in Maxwell equations or the stability of magnetostatics problems \cite[Section 2]{Di-Pietro.Droniou.ea:20}. Designing discrete versions of this complex is essential to ensure that the resulting scheme also satisfy these properties.

Several approaches have been employed to achieve such a discretisation. One of the most successful one is the Finite Element Exterior Calculus (FEEC), which provides a generic framework for constructing and analysing discrete complexes \cite{Arnold.Falk.ea:06,Arnold:18}. However, its design -- based on explicit locally polynomial and globally conforming functions -- makes it only applicable to certain meshes, mostly made of triangles/tetrahedra or quadrangles/hexahedra. We also note the broader framework of Generalised Finite Element Systems \cite{Christiansen.Hu:18}, but whose analysis tools are currently limited to algebraic properties of the discrete complexes, and does not provide explicit means to construct fully computable complexes on generic meshes. 

The application of FEEC to computational electromagnetism on domains in $\mathbb{R}^3$ was pioneered by A. Bossavit \cite{Bossavit:88,Bossavit:88a,Bossavit:88b}. Subsequent developments are discussed in \cite{Hiptmair:02,Hiptmair:15}, along with references therein. Despite significant progress, to our knowledge the application of FEEC techniques in electromagnetism has so far remained restricted to flat domains in $\mathbb{R}^3$.

A recent new trend is that of polytopal methods, that is, numerical methods that are applicable on meshes made of generic polygons/polyhedra (polytopes), and have an arbitrary degree of accuracy. The Virtual Element Method (VEM) \cite{Beirao-da-Veiga.Brezzi.ea:16,Beirao-da-Veiga.Brezzi.ea:18*2} and Discrete De Rham (DDR) \cite{Di-Pietro.Droniou.ea:20,Di-Pietro.Droniou:23} are two main examples of approaches to building \emph{polytopal discrete complexes} (these two approaches can actually be bridged, see \cite{Beirao.Dassi.ea:22}). Such complexes have benefits over FEEC ones: the flexibility of polytopal meshes leads to seamless local mesh refinement (to better capture, e.g., stiff solutions or complex geometries), and their high-level approach (which does not require globally conforming basis functions) may lead to leaner constructions using systematic processes in their design, such as static condensation and serendipity; see \cite{Dassi.Lovadina.ea:20,Beirao-da-Veiga.Mascotto:23,Botti.Di-Pietro.ea:23,Di-Pietro.Droniou:23*2,Di-Pietro.Droniou:21*1} for examples of the efficiency of discrete polytopal complexes.

The question of stability of discretisations of PDEs also applies to models set on manifolds, and leads to the natural question of designing discrete versions of the de Rham complex on such spaces. 
In the finite element setting, most studies have been done on embedded manifolds -- typically, a surface in $\Real^3$. An initial (extrinsic) approach is to approximate the surface by a piecewise linear surface, 
on which a triangulations and finite element spaces can be easily constructed, 
and to numerically approximate the PDE on that approximated surface \cite{Demlow.Dziuk:07,Demlow:09} (see also \cite{Arnold:18} for FEEC and \cite{Frittelli.Madzvamuse.ea:23} for VEM). 
The drawback of this simple -- although completely computable -- approach is its practical limitation to low-order methods because the error committed by approximating the surface limits the accuracy of the resulting scheme 
(unless we have an explicit knowledge -- which is rare -- of the distance function in a tubular neighbourhood of the manifold). 
To mitigate this issue, a second (intrinsic) approach consists in defining triangulations and finite element spaces directly on the manifold, 
and in therefore discretising the PDE without the geometrical error created by approximating the manifold, see \cite{Bachini.Farthing.ea:21} and reference therein. 
The challenge is then to define a notion of piecewise polynomial functions on the triangulation of the manifold. 
One option is through the usage of an explicit piecewise linear manifold, 
on which finite element triangulations and spaces can be trivially defined (the notion of polynomials being obvious), 
and to transport these triangulations and spaces onto the physical manifold. 
In \cite{Bachini.Farthing.ea:21}, it is required that this approximate manifold be contained on an open set close enough to the surface to ensure the
uniqueness, on that open set, of the orthogonal projection on the manifold; 
this is actually not a requirement, as the mere existence of an homeomorphism between the physical and piecewise linear manifolds is sufficient to transport all finite element spaces 
($H^1$-conforming, but also $\mathbf{H}(\CURL)$- or $\mathbf{H}(\DIV)$-conforming) through pull-back \cite{Licht:23}. 
Another, strongly related, approach to construct intrinsic triangulations and finite element spaces is to assume the existence of a global explicit system of coordinates on the manifold; 
this system can then be used to represent the mesh and local polynomial spaces (and amounts to the above-mentioned homeomorphism and pull-back approach).
This, however, has some limitations on the type of manifold that can be considered (a sphere, e.g., cannot be covered by a single chart).
The reference \cite{Bachini.Manzini.ea:21} describes this approach in the context of VEM, 
and highlights that the usage of a single chart on a sphere generates a loss of accuracy due to the singularity of this chart at a point ; it is numerically shown that using two charts can mitigate this issue, but the design of the method does not account for the usage of multiple charts. 

In all cases, the existence of a global object to represent the manifold (piecewise linear manifold approximating and/or homeomorphic to the physical manifold) can be established, 
but the explicit description of this object remains elusive except in special cases, 
which mainly restricts global approaches to theoretical analysis, or requires to construct computable spaces that introduce an additional geometrical error, the control of which requiring additional properties on the global object -- e.g., 
that the homeomorphism is close to an isometry. 
See the discussion in \cite{Licht:23}, which designs and analyses a FEEC on manifolds.

The reason for the need of a global approach mostly lies in the requirement for finite element spaces to be subspaces of the continuous spaces (conforming approximation). 
This is also visible in the abstract construction of \cite[Section 3-4]{Licht:23} where, although the design of the polynomial spaces is done locally through pullbacks, 
the continuity requirements of the \emph{smooth triangulation} ultimately makes it non-computable (except for trivial case).

Polytopal methods do not have such a requirement. Instead, their discrete spaces can be interpreted as spaces of vectors of polynomial functions on mesh entities of various dimensions 
(elements, but also faces, edges, vertices); 
these polynomial functions are completely unrelated with each other: for example,
polynomials on the faces do not have to be the traces of polynomials in the elements;
see \cite{Beirao.Dassi.ea:22,Bonaldi.Di-Pietro.ea:23} for DDR and VEM. 
As a consequence, the restrictions -- even on flat spaces -- imposed on the geometry of the mesh elements by finite element methods are lifted, 
and there is no issue in gluing local constructions to form a global space (no global property on the space needs to be imposed). 
The trade-off is that notions of polynomials on faces and edges must be available. 
This is not in itself an issue even on manifolds, as local coordinates can be used for this, 
but when considering \emph{complexes}, some compatibility conditions are necessary; 
in particular, traces on faces of polynomials in elements must be polynomial on faces 
(which imposes a compatibility of the local coordinates used to define polynomials). 
Moreover, \emph{trimmed} local polynomial spaces play an important role in discrete complexes \cite{Arnold:18,Bonaldi.Di-Pietro.ea:23,Beirao-da-Veiga.Brezzi.ea:18*2}; 
when going to polytopal complexes on manifolds, such notions of spaces must therefore also be defined, 
which raises questions on intrinsic Koszul operators in local coordinates and their compatibility conditions to ensure that traces of trimmed polynomial spaces remain trimmed polynomial spaces of the same degree.

At this point, it is worthwhile noticing other approaches to solving partial differential equations on manifolds that have been developed, primarily by the General Relativity community. The oldest approach is the Regge calculus \cite{Gentle:2002,Regge:1961,Williams.Tuckey:1992}, which is a lattice based method for numerically solving the Einstein equations. Related to the Regge calculus is the smooth lattice method developed in \cite{Brewin:1998}. This method offer significant improvements over the Regge calculus and has been used to solve the Einstein equations on spacetime manifolds with compact spacial slices \cite{Brewin:2015,Brewin:2017}.
In a different direction, an approach to solving partial differential equations, in particular the Einstein equations, on spacetime manifolds that is based on local discrete coordinate patches was developed in \cite{Schnetter_et_al:2006}.
More recently, the authors of the articles \cite{Lindblom_et_al:2022,Lindblom.Szilagyi:2013,Zhang.Lindblom:2022} developed a multi-cube decomposition method for solving partial differential on compact manifolds and employed this method to construct numerical solutions of the Einstein constraint equations on various compact three-manifolds. Finally, we mention the article \cite{Frauendiener:2006}, which is most relevant to our results, where a discrete exterior calculus is employed to design a numerical scheme for solving the Einstein equations. This approach has been used to find solutions to the Einstein equations on spacetimes with compact spatial slices in \cite{Richter.Frauendiener:2010}. All these methods, however, are limited to meshes with elements having a specific shape (tetrahedral or cubic).

\medskip

In this work, we tackle the question of designing a discrete complex on generic (not necessarily embedded) manifolds, that supports meshes having generic element shapes. 
We define suitable notions of local (complete and trimmed) polynomial spaces on a generic mesh of $\Omega$, tackling the issue of compatibility of traces of polynomials. Like the abstract setting of \cite{Licht:23}, our construction relies on mappings between subsets of $\Real^d$ and the manifold.
However, our compatibility requirement between the mappings is much milder. This added flexibility enables the construction of computable mappings.
Our construction works even when several charts are used to represent the manifold, and does not require a global homeomorphism between the manifold and a piecewise linear manifold; 
it is also purely local, and we provide practical spaces on various polygonal elements. 
We then use these local polynomial spaces to design, 
following the construction in \cite{Bonaldi.Di-Pietro.ea:23}, a discrete version of the de Rham complex of differential forms on the manifold. Some key properties of this complex are stated, 
including the isomorphism of its cohomology with that of the de Rham complex, as well as its primal consistency (Theorem \ref{thm:primal.consistency}), which is one of the main novelties of this work as it does not follow as in the regular (flat) DDR complex -- in particular, it is not a consequence of an exact consistency on polynomial forms.
We derive a 2+1 formulation of the Maxwell equations on 2D manifolds, and use the DDR complex to design a scheme for these equations; 
the preservation of constraint follows from the properties of the discrete complex. 
The behaviour of the scheme is numerically illustrated on manufactured solutions on the sphere and the torus; to do so, we
design suitable meshes (based on two charts of each manifold) on which local polynomial spaces with compatibility conditions can be constructed.

\medskip 

The paper is organised as follows. In Section \ref{sec:setting}, we introduce the abstract setting required to define the DDR complex on a manifold; in particular, we identify a suitable notion of
polynomial spaces on a mesh of the manifold, that allows us to construct trimmed polynomial spaces and such that restrictions of polynomials on lower-dimensional mesh entities are also polynomials on these mesh entities. This abstract setting is used in Section \ref{sec:DDR} to build the discrete spaces and operators of the discrete de Rham complex of differential
forms on the manifold, following the design in \cite{Bonaldi.Di-Pietro.ea:23}. We state the main properties (consistency, commutation, isomorphism of cohomologies) of this DDR complex. In Section \ref{sec:construction.polynomial} shows how to explicitly construct suitable polynomial spaces,
satisfying the setting of Section \ref{sec:setting}, provided that local charts with a compatibility condition (between charts of a face and charts of its sub-faces) can be found.
Section \ref{sec:application} then tackles an application of the discrete de Rham complex: the electromagnetism equations on a manifold. We recall the 2+1 formulation of the Maxwell equations in differential forms, and build a scheme for this model. The properties of the DDR complex show that this scheme preserves a discrete version of the constraint on the electric field as well as the energy of the model. We finally perform numerical tests on the 2-dimensional sphere and torus, illustrating the practical behaviour and accuracy of the scheme. Three appendices conclude the paper: Appendix \ref{sec:example.construction} details a practical construction of polynomial spaces on a 2D manifold mesh made of curved 3- or 4-edge elements, Appendix \ref{sec:tensor.calculus} recalls important tensor calculus constructions and derives the 2+1 formulation of the elecromagnetism equations, and Appendix \ref{sec:exact.solutions} presents the analytical solutions used in the numerical simulations.


\section{Setting}\label{sec:setting}
\subsection{Mesh}

We let $\Mh$ be a mesh made of a finite collection of submanifolds of dimensions $d=0,\ldots,n$ that partition $\Omega$. For such a $d$, the set $\Delta_d(\Mh)$ collects all the submanifolds in $\Mh$ of dimension $d$, which we also call $d$-cells, or cells when the dimension is not useful; ``elements'' refer to $n$-cells.
 
We assume the following properties (the first two explicitly re-state that $\Mh$ partitions $\Omega$):
\begin{itemize}
\item The submanifolds cover $\Omega$: $\bigcup_{d\in [0,n]}\bigcup_{f\in\FM{d}(\Mh)} f = \Omega$.
\item The submanifolds are pairwise disjoint: 
  $\forall f, f' \in \Mh, f \neq f' \implies f \cap f' = \emptyset$.
\item Boundaries of the submanifolds are submanifolds:
  $\forall d \in [1,n]$, $\forall f \in \FM{d}(\Mh)$, $\forall d' \in [0,d-1]$, $\forall f' \in \FM{d'}(\Mh)$, $f' \cap \overline{f} \neq \emptyset \implies f' \subset \pf$.
\end{itemize}
The last property allows us to define the collection of cells on the boundary of $f$:
$\FM{d'}(f) := \lbrace f' \in \FM{d'}(\Mh) \st f' \subset \pf \rbrace$.

Algebraic properties of the discrete de Rham complex designed in Section \ref{sec:DDR} do not require any other assumption on the mesh. 
However, analytical properties (such as the consistency of the complex, see Section \ref{sec:consistency}) do depend on the regularity of the mesh, a concept that we introduce now. 
In what follows, the notation $a\lesssim b$ stands for $a\le Cb$ with a constant $C$ that does not depend on 
the parameter $h$, the dimension $d\le n$, on the subcell $f\in\Delta_d(\Mh)$ or on the forms considered on $f$. 
The notation $a\approx b$ stands for $a\lesssim b$ and $b\lesssim a$.

\begin{definition}[Equivalence of meshes] \label{def:equiv.meshes}
We say that a (flat) polytopal mesh $M_h$ (as defined in \cite[Definition 1.4]{Di-Pietro.Droniou:20})
  is equivalent to $\Mh$ if:
  \begin{itemize}
    \item Their same-dimensional cells can be put in bijection: 
      $\forall d\in [0, n]$, there is a bijection $\Phi_d\st\FM{d}(\Mh)\to\FM{d}(M_h)$.
    \item They are topologically equivalent: 
      $\forall d\in [1,n]$, $\forall f\in\FM{d}(\Mh)$, $\Phi_{d-1}(\FM{d-1}(f)) = \FM{d-1}(\Phi_{d}(f))$.
    \item Their geometries are equivalent: 
      $\forall d\in [1, n]$, $\forall f\in\FM{d}(\Mh)$, there is a diffeomorphism 
      $\phi_f \st f\to\Phi_d(f)$ satisfying:
      \begin{equation}\label{eq:shape.reg}
        \norm{\infty}{\nabla \phi_f} \approx \norm{\infty}{\det(\nabla\phi_f)}^{\frac{1}{d}}
        \approx \norm{\infty}{\det(\nabla\phi_f^{-1})}^{-\frac{1}{d}} \approx\norm{\infty}{\nabla\phi_f^{-1}}^{-1}.
      \end{equation}
    \item The geometry is regular on boundaries: 
      $\forall d\in[2 , n]$, $\forall f\in\FM{d}(\Mh)$, $\forall f'\in\FM{d-1}(f)$,
      \begin{equation}\label{eq:boundary.reg}
        \norm{\infty}{\nabla\phi_f} \approx \norm{\infty}{\nabla\phi_{f'}}.
      \end{equation}
  \end{itemize}
  \end{definition}
  
  Above, $\norm{\infty}{F}$ denotes the $L^\infty$-norm of the function $F$ on its domain. Note that imposing, for example, that $\norm{\infty}{\nabla \phi_f}$ is comparable to $\norm{\infty}{\det(\nabla\phi_f)}^{\frac{1}{d}}$ amounts to enforcing a level of ``isotropy'' to $f$, and \eqref{eq:shape.reg} is therefore an assumption on the shape regularity of $f$.
    
  \begin{assumption}[Regular mesh sequence]\label{assum:reg.seq}
  We assume that the mesh sequence $(\Mh)_{h}$ is regular in the sense that there exists an equivalent (in the sense of Definition \ref{def:equiv.meshes}) polytopal mesh sequence $(M_h)_h$ which is regular in the sense of \cite[Definition 1.9]{Di-Pietro.Droniou:20}.
  \end{assumption}

  For $d\ge 1$, we define the characteristic size of a $d$-cell $f$ as 
  $h_f := \vert f\vert^{\frac{1}{d}}$. 
  Noticing that
  \begin{equation}\label{eq:hf.equiv}
    h_f \approx \norm{\infty}{\nabla\phi_f}^{-1} h_{\Phi_d(f)},
  \end{equation}
  we can translate the geometric bounds on regular polytopal mesh sequences into our setting. 
  We recall the main properties:
  \begin{itemize}
    \item The size of neighbouring cells is comparable: 
      $\forall d\in [2,n]$, $\forall f\in\FM{d}(\Mh)$, $\forall f'\in\FM{d-1}(f)$, $h_f \approx h_{f'}$. 
    \item The number of sub-cells in an cell boundary is bounded: 
      $\forall d\in [1,n]$, $\forall f\in\FM{d}(\Mh)$, $\mathrm{Card}(\FM{d-1}(f)) \lesssim 1$. 
  \end{itemize}
  
\subsection{Discrete polynomial spaces} \label{sec:discrete_spaces}

In the rest of the paper, we use a range of concepts related to differential forms and exterior calculus. We recall in Appendix \ref{sec:tensor.calculus} some of these concepts, and we refer the interested reader to \cite[Appendix A]{Bonaldi.Di-Pietro.ea:23} for a longer presentation in the case of Euclidean spaces. Wherever relevant, we will highlight specific considerations for manifolds.

To build a DDR complex on $\Mh$, we need a notion of ``local exterior calculus polynomial space'' on the submanifolds in $\Mh$. 
This is not an easy concept to define as these spaces must satisfy a range of compatibility properties with the exterior derivative, the trace, etc. 
In this section,  we define the key axioms on these spaces, and infer further notions such as trimmed polynomial spaces.
The construction below mimics well-known properties of polynomial spaces, 
exterior derivative and Koszul operator on flat spaces; 
we however have to check carefully that the minimal properties we require 
(the complex properties and Assumption \ref{assumption:homogeneous}) 
are indeed sufficient to deduce, in particular, 
the decomposition of polynomial spaces which justifies the definition of trimmed spaces.

Let $0 \leq d \leq n$ and $f \in \FM{d}(\Mh)$. We consider spaces $(\PL{r}{l}(f))_{r\in \Zintegers,l\in\Zintegers}$ such that $\PL{r}{l}(f)\subset C^1\Lambda^l(f)$, and that form a complex for the exterior derivative $\DIFF$ in the following way:
\begin{equation}\label{eq:polynomial.complex.diff}
  \begin{tikzcd}
   \cdots  \arrow{r}{\DIFF}& \PL{r}{l}(f) \arrow{r}{\DIFF} & \PL{r-1}{l+1}(f) \arrow{r}{\DIFF}& \cdots.
  \end{tikzcd}
\end{equation}
We assume the existence of a graded map $\KOSZUL_f$ (removing the index $f$ when no confusion can arise), 
playing the role of a Koszul operator in our setting, 
such that the following ``reverse'' sequence forms a complex
\begin{equation}\label{eq:polynomial.complex.koszul}
  \begin{tikzcd}
   \cdots  \arrow{r}{\KOSZUL}& \PL{r}{l}(f) \arrow{r}{\KOSZUL} & \PL{r+1}{l-1}(f) \arrow{r}{\KOSZUL}& \cdots.
  \end{tikzcd}
\end{equation}

\begin{assumption}[Homogeneous polynomials] \label{assumption:homogeneous}
We assume the following properties, for all $0\le d\le n$ and $f\in\Delta_d(\Mh)$:
\begin{enumerate}[label=\textbf{(\ref{assumption:homogeneous}.{\alph*})}]
  \item\label{A:homogeneous.polynomials} There exists spaces $(\HL{s}{l}(f))_{s\in\Zintegers,l\in\Zintegers}$ that form a graded decomposition $\PL{r}{l}(f) = \bigoplus_{s \leq r} \HL{s}{l}(f)$ 
    such that, on each $\HL{s}{l}(f)$, $\DIFF \KOSZUL + \KOSZUL \DIFF$ acts as an homothetie:
    $\forall s,l\in\Zintegers$, $\exists \lambda_{s,l}\in\Real$ such that $(\DIFF\KOSZUL + \KOSZUL\DIFF)p = \lambda_{s,l} p$ for all $p \in \HL{s}{l}(f)$.
    For $s<0$, $\HL{s}{l}(f)=\{0\}$ and we fix $\lambda_{s,l}=\lambda_{s+l,0}$, with $\lambda_{s+l,0}=0$ if $s+l<0$.
  \item\label{A:eigenvalues} The eigenvalues of $\DIFF\KOSZUL+\KOSZUL\DIFF$ identify the polynomial degree: 
    $\forall d > 0$, $\forall l\in [0,d]$, $\forall s,s'\ge 0$, $\lambda_{s,l} = \lambda_{s',l} \implies s = s'$.
    Moreover, if $s\ge 0$ then $\lambda_{s,l} = 0$ if and only if $(s,l) = (0,0)$.
  \item\label{A:d.non.zero} $\DIFF \HL{s}{l}(f) \neq \lbrace 0 \rbrace$ when $s > 0$ and $l < d$.
\end{enumerate}
\end{assumption}

\begin{lemma} \label{lemma:homogeneous.mapping}
  Under Assumption \ref{assumption:homogeneous}, for all $0 \le l \leq d$, $s\in\Zintegers$ and $f\in\Delta_d(\Mh)$, the eigenvalues are related by the relation $\lambda_{s,l} = \lambda_{s+l,0}$.
  Moreover $\DIFF\HL{s+1}{l}(f)\subset \HL{s}{l+1}(f)$ and $\KOSZUL\HL{s}{l+1}(f)\subset \HL{s+1}{l}(f)$.
\end{lemma}
\begin{proof}
  The case $s<0$ being obvious by the choice in Assumption \ref{A:homogeneous.polynomials}, we only need to consider $s\ge 0$.
  Let us first take $l=0$ and let us prove by induction on $s\in \Natural$ that 
  $\DIFF\HL{s+1}{0}(f)\subset \HL{s}{1}(f)$ and $\lambda_{s,1}=\lambda_{s+1,0}$.
  Starting with $s=0$, Assumption \ref{A:homogeneous.polynomials} gives $\DIFF\HL{1}{0}(f) \subset \DIFF\PL{1}{0}(f) \subset \PL{0}{1}(f) = \HL{0}{1}(f)$, which proves the first property for the base case. For the second property,
  by Assumption \ref{A:d.non.zero} there is $p \in \HL{1}{0}(f)$ such that $\DIFF p \neq 0$.
  Moreover $\KOSZUL\DIFF p = \lambda_{1,0} p - \DIFF\KOSZUL p$.
  Thus,
  \[
    (\DIFF\KOSZUL+\KOSZUL\DIFF)\DIFF p 
    = \DIFF \KOSZUL \DIFF p
    = \lambda_{1,0} \DIFF p - \cancel{\DIFF \DIFF \KOSZUL \DIFF p},
  \]
  and $\DIFF p$ is an eigenvector of $\DIFF\KOSZUL+\KOSZUL\DIFF$ associated with the eigenvalue $\lambda_{1,0}$.
  We note that $\DIFF p \in \DIFF\HL{1}{0}(f)\subset \HL{0}{1}(f)$.
  Therefore, $\lambda_{0,1} = \lambda_{1,0}$ and the case $s=0$ is proved.
  
  Let now $s\ge 1$ and assume that, for all $0 \leq k < s$, $\DIFF \HL{k+1}{0}(f)\subset \HL{k}{1}(f)$ and $\lambda_{k,1} = \lambda_{k+1,0}$. We need to prove that $\lambda_{s,1}=\lambda_{s+1,0}$. 
  Let $p\in\HL{s+1}{0}(f)$. If $\DIFF p=0$ then clearly $\DIFF p\in\HL{s}{1}(f)$. Otherwise, using the same arguments as for $s=0$
  we note that $(\DIFF\KOSZUL+\KOSZUL\DIFF)\DIFF p = \lambda_{s+1,0} \DIFF p$, 
  and $\DIFF p \in \PL{s}{1}(f) = \bigoplus_{k \leq s} \HL{k}{1}(f)$.
  Since $\DIFF p$ is an eigenvector of $\DIFF\KOSZUL+\KOSZUL\DIFF$ (as it is not equal to $0$), by Assumption \ref{A:eigenvalues} it belongs to an $\HL{k}{1}(f)$ for some $k$ such that $\lambda_{k,1}=\lambda_{s+1,0}$.
  By induction hypothesis, if $k<s$ we have $\lambda_{k,1}=\lambda_{k+1,0}\not=\lambda_{s+1,0}$ (the non-equality coming from Assumption \ref{A:eigenvalues}). Hence, $k=s$ and $\lambda_{s,1}=\lambda_{s+1,0}$. In passing, we have also proved that $\DIFF p\in\HL{s}{1}(f)$, so that $\DIFF\HL{s+1}{0}(f)\subset\HL{s}{1}(f)$, and the induction is therefore complete.
  
  Another induction on $l$ concludes the proof that $\lambda_{s,l} = \lambda_{s+l,0}$ for all $s,l$.

  The fact that $\KOSZUL$ maps eigenspaces into eigenspaces also follows from Assumption \ref{A:homogeneous.polynomials} and similar arguments. The details are left to the reader.
\end{proof}

\begin{lemma} \label{lemma:homogeneous.iso}
  Under Assumption \ref{assumption:homogeneous}, for all $0 \le l \leq d$, $s\in\Zintegers$ and $f\in\Delta_d(\Mh)$, the following mappings are one-to-one:
  \begin{align*}
    \DIFF \st \KOSZUL\HL{s}{l}(f) \to&\, \DIFF\HL{s+1}{l-1}(f) \subset \HL{s}{l}(f)\\  
    \KOSZUL \st \DIFF\HL{s}{l}(f) \to&\, \KOSZUL\HL{s-1}{l+1}(f) \subset \HL{s}{l}(f)\,.
  \end{align*}
\end{lemma}
\begin{proof}
  We only need to prove that $\DIFF$ is injective on $\KOSZUL\HL{s}{l}(f)$, the other case being similar. 
  The case $l=0$ is trivial since 
  \begin{equation}\label{eq:kHL}
  \KOSZUL\HL{s}{0}(f)= \KOSZUL\PL{s}{0}(f)\subset \PL{s+1}{-1}(f)\subset C^1\Lambda^{-1}(f) = \lbrace 0\rbrace .
  \end{equation}
   We therefore assume that $l\ge 1$ and take
  $q\in \KOSZUL\HL{s}{l}(f)$ such that $\DIFF q=0$. We can write $q=\KOSZUL p$ for some $p\in \HL{s}{l}(f)$.
  By Assumption \ref{A:homogeneous.polynomials}, we have $(\DIFF\KOSZUL+\KOSZUL\DIFF)p = \lambda_{s,l} p$.
  As $\DIFF \KOSZUL p = \DIFF q=0$, we infer $\KOSZUL\DIFF p = \lambda_{s,l} p$ and $\lambda_{s,l} \KOSZUL p = \KOSZUL\KOSZUL\DIFF p = 0$.
  By Assumption \ref{A:eigenvalues}, $\lambda_{s,l} \neq 0$ (since $(s,l)\not=(0,0)$), and thus $q=\KOSZUL p= 0$.
  
  To conclude the proof, we note that the property $\DIFF\KOSZUL\HL{s}{l}(f)\subset \DIFF\HL{s+1}{l-1}(f)$ follows from Lemma \ref{lemma:homogeneous.mapping} if $l\ge 1$, and from \eqref{eq:kHL} if $l=0$.

\end{proof}

\begin{lemma}[Decomposition of polynomial spaces] \label{lemma:poly.dec}
  Under Assumption \ref{assumption:homogeneous}, for all $f\in\Delta_d(\Mh)$ and all $r\ge 0$ the following direct decompositions hold:
  \begin{align*}
    \PL{r}{0}(f) ={}& \PL{0}{0}(f) \oplus \KOSZUL \PL{r-1}{1}(f)\\
    \PL{r}{l}(f) ={}& \DIFF\PL{r+1}{l-1}(f) \oplus \KOSZUL \PL{r-1}{l+1}(f)\quad \mbox{if $l\ge 1$.}
  \end{align*}
\end{lemma}
\begin{proof}
  We first prove that
  \begin{equation}\label{eq:homogeneous.direct.sum}
    \HL{s}{l}(f) = \DIFF\HL{s+1}{l-1}(f) \oplus \KOSZUL\HL{s-1}{l+1}(f)\quad\mbox{ if $(s,l)\not=(0,0)$.}
  \end{equation}
  For $p \in \HL{s}{l}(f)$, Assumption \ref{A:homogeneous.polynomials} yields $(\DIFF\KOSZUL+\KOSZUL\DIFF) p = \lambda_{s,l} p$ with $\lambda_{s,l} \neq 0$ since $(s,l) \neq (0,0)$. Hence,
  \[
    p=\DIFF(\lambda_{s,l}^{-1}\KOSZUL p) + \KOSZUL(\lambda_{s,l}^{-1}\DIFF p).
  \]
  Noticing that $\KOSZUL p \in \HL{s+1}{l-1}(f)$, $\DIFF p \in \HL{s-1}{l+1}(f)$, 
  $\DIFF \HL{s+1}{l-1}(f) \subset \HL{s}{l}(f)$ and $\KOSZUL \HL{s-1}{l+1}(f) \subset \HL{s}{l}(f)$ by Lemma \ref{lemma:homogeneous.mapping}, 
  this proves that $\HL{s}{l}(f) = \DIFF\HL{s+1}{l-1}(f) + \KOSZUL\HL{s-1}{l+1}(f)$. 
  It remains to prove that this sum is direct. If $z\in \DIFF \HL{s-1}{l+1}(f)\cap\KOSZUL \HL{s+1}{l-1}(f)$ 
  then there exists $p \in \HL{s-1}{l+1}(f)$ and $q \in \HL{s+1}{l-1}(f)$ such that $z=\KOSZUL p = \DIFF q $, 
  then $\DIFF \KOSZUL p = \DIFF\DIFF q = 0$.
  Lemma \ref{lemma:homogeneous.iso} then implies $z=\KOSZUL p = 0$, which concludes the proof of \eqref{eq:homogeneous.direct.sum}.

  We now turn to the result in the lemma. If $l\not=0$, we can take the direct sums of \eqref{eq:homogeneous.direct.sum} to get
  \begin{align*}
    \PL{r}{l}(f) 
    ={}& \bigoplus_{s \leq r} \HL{s}{l}(f)\\
    ={}& \bigoplus_{s \leq r} \DIFF\HL{s+1}{l-1}(f) \oplus \KOSZUL\HL{s-1}{l+1}(f)\\
    ={}& \DIFF \bigoplus_{s \leq r} \HL{s+1}{l-1}(f) \oplus \KOSZUL\bigoplus_{s \leq r}\HL{s-1}{l+1}(f)\\
    ={}& \DIFF\PL{r+1}{l-1}(f) \oplus \KOSZUL \PL{r-1}{l+1}(f).
  \end{align*}
  If $l=0$, we simply write 
  $$
  \PL{r}{0}(f) = \HL{0}{0}(f)\oplus \bigoplus_{1\le s \leq r} \HL{s}{0}(f)=\PL{0}{0}(f)\oplus \bigoplus_{1\le s \leq r} \HL{s}{0}(f)
  $$ 
  and apply \eqref{eq:homogeneous.direct.sum} with $s\in [1,r]$ and $l=0$, noting that $\DIFF\HL{s+1}{-1}(f)=\{0\}$.
\end{proof}

\begin{definition}[Trimmed polynomial spaces] \label{def:poly.trimmed}
Under Assumption \ref{assumption:homogeneous}, for all $f\in\Delta_d(\Mh)$ and all $r\ge 0$,
the trimmed polynomial spaces are defined as
\begin{align*}
  \PLtrimmed{r}{0}(f) :={}& \PL{r}{0}(f)\\
  \PLtrimmed{r}{l}(f) :={}& \DIFF \PL{r}{l-1}(f) \oplus \KOSZUL \PL{r-1}{l+1}(f)\quad\mbox{ if $l\ge 1$}.
\end{align*}
\end{definition}

\begin{remark} 
  It follows directly from the direct decomposition of Lemma \ref{lemma:poly.dec} that,
  for all $r, l, d \in \Zintegers$ and $f \in \FM{d}(\Mh)$,
  \[
    \PL{r}{l}(f) \subset \PLtrimmed{r+1}{l}(f) \subset \PL{r+1}{l}(f).
  \]
\end{remark}

\begin{assumption}[Traces of trimmed polynomials] \label{assumption:tracetrimmed}
We assume that traces of trimmed polynomials are trimmed polynomials, that is: for all $l\in\Zintegers$, $f\in\FM{d}(\Mh)$, $d'\in[0,d]$ and $r\ge 0$,
\[
  \tr_{f'} \PLtrimmed{r}{l}(f) \subset \PLtrimmed{r}{l}(f')\qquad \forall f' \in \FM{d'}(f),
\]
where $\tr_{f'}:C^0\Lambda^l(f)\to C^0\Lambda^l(f')$ is the standard trace operator of differential forms, with $\tr_{f'}={\rm Id}$ if $f'=f$.
\end{assumption}

We note that, when $r=0$, local polynomial spaces satisfying Assumptions \ref{assumption:homogeneous} and \ref{assumption:tracetrimmed} can  be trivially constructed by considering constant forms on each $f\in\Delta_d(\Mh)$; in that case, the discrete de Rham complex built in Section \ref{sec:DDR} is akin to a Compatible Discrete Operator sequence on the (curved) manifold mesh.
In the general setting $r\ge 1$, we show in Section \ref{sec:construction.polynomial} that the existence of particular local charts allows for explicit constructions of local polynomial spaces satisfying Assumptions \ref{assumption:homogeneous} and \ref{assumption:tracetrimmed}.

\section{Discrete de Rham complex}\label{sec:DDR}

\subsection{Construction}

The construction of the DDR complex is done following the approach of \cite{Bonaldi.Di-Pietro.ea:23} in the flat case, with one major difference due to the fact that we are working on manifolds instead of Euclidean spaces: the Hodge star operator $\star$ on $\Omega$ (or on its submanifolds forming the mesh) depends on the metric of this manifold (or the induced metric on the submanifolds) -- see \eqref{eq:Hodge.star.coords} in Appendix \ref{sec:tensor.calculus} for its definition. As a consequence, it does not preserve polynomial spaces, since the coefficients of the metric, in the local coordinates chosen to define the notion of polynomial functions, are not necessarily constant. This leads us to consider a slightly different definition of the discrete space \eqref{eq:def.Xkh} and related interpolator \eqref{eq:def.I}, as well as a slightly different local discrete exterior derivative and potential (Definition \ref{def:dk.Pk}).

We recall that, on a given cell $f\in\Delta_d(\Mh)$ and for any $k\in\{0,\ldots,d\}$, the Hodge star operator $\star$ is a mapping that associates to each $k$-differential form $\omega$ on $f$ a $(d-k)$-differential form $\star\omega$ on $f$ such that, for any $x\in f$ and any $(d-k)$-alternating form $\mu$ on the tangent space $T_xf$,
\[
\langle (\star\omega)_x,\mu\rangle_{g_x}\vol_x = \omega_x\wedge\mu,
\]
where $\langle\cdot,\cdot\rangle_{g_x}$ is the metric-dependent inner product on $(d-k)$-alternating forms on $T_xf$, and $\vol_x$ is the volume form on this tangent space. We also refer to \eqref{eq:Hodge.star.coords} for an expression in coordinates on Lorentzian manifolds -- the only difference with Riemannian manifolds being in the removal of the minus sign in front of $\det g$.

Given a polynomial degree $r\ge 0$, the discrete counterpart $\uH{k}{h}$ of the space $H\Lambda^k(\Omega)$, $0 \leq k \leq n$, is defined as
\begin{equation}\label{eq:def.Xkh}
  \uH{k}{h} := \bigtimes_{d = k}^{n} \bigtimes_{f \in \FM{d}(\Mh)} \star^{-1}\PLtrimmed{r}{d-k}(f).
\end{equation}
A generic $\ul\omega_h\in\uH{k}{h}$ is denoted by $\ul\omega_h=(\omega_f)_{f\in\Delta_d(\Mh),\,d\in[k,n]}$ with $\omega_f\in \star^{-1}\PLtrimmed{r}{d-k}(f)$
for all $f\in\Mh$. As commonly done in DDR constructions, we denote by $\ul\omega_f=(\omega_{f'})_{f'\in\Delta_d'(f),\,d'\in[k,d]}$ the restriction of $\ul\omega_h$ to $f\in\Delta_d(\Mh)$. The same convention (replacing an index $h$ with $f$) will also be used to denote restrictions of operators. 

The interpolator $\uI{k}{h}:C^0\Lambda^{k}(\overline{\Omega})\to\uH{k}{h}$ is defined by projecting the traces on the polynomial spaces: for all $\omega\in C^0\Lambda^{k}(\overline{\Omega})$,
  \begin{equation}\label{eq:def.I}
    \uI{k}{h} \omega := (\star^{-1}\ltproj{d-k}{f}\star \tr_f \omega)_{f \in \FM{d}(\Mh), d \in [k,n]}
  \end{equation}
  where $\ltproj{d-k}{f}:L^2\Lambda^{d-k}(f)\to\PLtrimmed{r}{d-k}(f)$ is the $L^2$-orthogonal projector on the trimmed space.


\begin{definition}[Local discrete exterior derivative and discrete potential]\label{def:dk.Pk}
  Let $f \in \FM{d}(\Mh)$, $k\in [0,d]$ and $r\ge 0$. The local discrete exterior derivative
  $\DIFF^k_{r,f} \st \uH{k}{f} \rightarrow \star^{-1}\PL{r}{d-k-1}(f)$ and discrete potential
  $P^k_{r,f} \st \uH{k}{f} \rightarrow \star^{-1} \PL{r}{d-k}(f)$ are defined inductively on the dimension of $f$ as follows.
  For all $\ul{\omega}_f\in \uH{k}{f}$:
  \begin{itemize}[leftmargin=1em]
    \item If $d = k$:
      \begin{equation} \label{eq:def.P.d=k}
        P^k_{r,f} \ul{\omega}_f :=\, \omega_f \in \star^{-1}\PL{r}{0}(f).
      \end{equation}
    \item If $d \geq k+1$:
      \begin{itemize}[label=$\diamondsuit$,leftmargin=1em]
        \item 
          $\forall \mu_f \in \PL{r}{d-k-1}(f)$,
          \begin{equation} \label{eq:def.d}
            \int_f \DIFF^k_{r,f} \ul{\omega}_f \wedge \mu_f
              := (-1)^{k+1} \int_f \omega_f \wedge \DIFF \mu_f + \int_\pf P^k_{r,\pf} \ul{\omega}_f \wedge \tr_\pf \mu_f,
          \end{equation}
          where we have denoted by $P^k_{r,\pf}\ul\omega_f$ the piecewise function on $\pf$ obtained patching the functions $P^k_{r,f'}\ul\omega_{f'}$ for $f'\in\Delta_{d-1}(f)$.
        \item 
          $\forall \mu_f \in \KOSZUL \PL{r}{d-k}(f)$, $\forall\nu_f \in \KOSZUL \PL{r-1}{d-k+1}(f)$,
          \begin{equation} \label{eq:def.P}
            \begin{aligned}
              (-1)^{k+1} \int_f P^k_{r,f} \ul{\omega}_f \wedge (\DIFF \mu_f + \nu_f)
              :={}& \int_f \DIFF^k_{r,f} \ul{\omega}_f \wedge \mu_f - \int_\pf P^k_{r,\pf} \ul{\omega}_\pf \wedge \tr_\pf \mu_f \\
              &+ (-1)^{k+1} \int_f \omega_f \wedge v_f.
            \end{aligned}
          \end{equation}
      \end{itemize}
      In these formulas, the orientation on the boundary $\partial f$ is the one induced by $f$.
  \end{itemize}
\end{definition}

The global discrete exterior derivative is then obtained projecting the local discrete exterior derivatives on the spaces forming the component of the global discrete space \eqref{eq:def.Xkh}: $\ul\DIFF^k_{r,h}:\uH{k}{h}\to\uH{k+1}{h}$ is given by
  \[
    \ul{\DIFF}^k_{r,h} \ul{\omega}_h := (\star^{-1}\ltproj{d-k-1}{f}\star \DIFF^k_{r,f} \ul{\omega}_f)_{f \in \FM{d}(\Mh), d \in [k+1,n]}.
  \]

The discrete DDR sequence on $\Mh$ is then
\begin{equation}\label{eq:DDR.complex}
  \begin{tikzcd}
   \cdots  \arrow{r}{\ul{\DIFF}^{k-1}_{r,h}}& \uH{k}{h} \arrow{r}{\ul\DIFF^k_{r,h}} & \uH{k+1}{h} \arrow{r}{\ul\DIFF^{k+1}_{r,h}}& \cdots.
  \end{tikzcd}
\end{equation}

To design numerical schemes based on this discrete complex, we need to define discrete $L^2$-like inner products.

\begin{definition}[Discrete $L^2$-like inner product] \label{def:inner.product}
    For $k\in [0,n]$ and $r\ge 0$, the discrete $L^2$-like inner product on $\uH{k}{h}$ is defined by: for all $\ul{\omega}_h, \ul{\mu}_h \in \uH{k}{h}$,
    \begin{align*}
        \langle \ul{\omega}_h, \ul{\mu}_h \rangle =
        {}&\sum_{f \in \FM{n}(\Mh)} \Bigg(
        \int_f P^k_{r,f}\ul{\omega}_f \wedge \star P^k_{r,f}\ul{\mu}_f \\
        {}&\qquad+ \sum_{d=k}^{n-1} h_f^{n-d} \sum_{f'\in\FM{d}(f)} 
        \int_{f'} (P^k_{r,f'}\ul{\omega}_{f'} - \tr_{f'} P^k_{r,f} \ul{\omega}_{f}) 
        \wedge \star (P^k_{r,f'}\ul{\mu}_{f'} - \tr_{f'} P^k_{r,f} \ul{\mu}_{f})
        \Bigg).
    \end{align*}  
\end{definition}

\subsection{Algebraic properties}

The following results are covered in the flat case in \cite{Bonaldi.Di-Pietro.ea:23}. Their proofs there only rely on the relations between the Hodge star operator and integrals of differential forms (which are also satisfied in the case of manifolds), the properties of full and trimmed polynomial spaces (which are mimicked by our constructions and assumptions in Section \ref{sec:discrete_spaces}), and the formulas defining the local discrete exterior derivative and potential reconstruction (same formulas as in Definition \ref{def:dk.Pk}, up to the bijective transformation of degrees of freedom through the introduction of $\star^{-1}$ in \eqref{eq:def.Xkh}). As a consequence, the proofs of \cite{Bonaldi.Di-Pietro.ea:23} directly adapt to our setting, and are therefore omitted.

\begin{lemma}[Link between discrete exterior derivatives on submanifolds]\label{lem:link.sub}
  Under Assumptions \ref{assumption:homogeneous} and \ref{assumption:tracetrimmed}, for all $d \geq 2$, $f \in \FM{d}(\Mh)$, $k\in [0,d]$, $r\ge 0$ and $\ul{\omega}_f \in \uH{k}{f}$, it holds
  \begin{equation*}
    \int_f \DIFF^k_{r,f} \ul{\omega}_f \wedge \DIFF \alpha_f = (-1)^{k+1} \int_{\pf} \DIFF^k_{r,\pf} \ul{\omega}_{\pf} \wedge \tr_{\pf} \alpha_f,
    \quad \forall \alpha_f \in \PLtrimmed{r+1}{d-k-2}(f).
  \end{equation*}
\end{lemma}
\begin{proof}
  See \cite[Lemma~22]{Bonaldi.Di-Pietro.ea:23}.
\end{proof}
  Lemma \ref{lem:link.sub} is a discrete counterpart of the formula
\[
 \int_f \DIFF\omega\wedge\DIFF\alpha =(-1)^{k+1}\int_f \DIFF (\DIFF \omega \wedge \alpha) = 
(-1)^{k+1}  \int_\pf  \tr_{\pf} \DIFF\omega\wedge\DIFF\alpha,
\]
where the first equality is obtained using $\DIFF (\DIFF \omega \wedge \alpha)=(\cancel{\DIFF\DIFF \omega}) \wedge \alpha+
(-1)^{k+1} \DIFF \omega \wedge \DIFF\alpha$.
While mostly useful at a technical level, it provides some insights on the compatibility between the 
continuous and the discrete exterior derivative, 
as well as on the relationship between the trace of the discrete derivative and the derivative on the boundary submanifold.
\begin{lemma}[Complex property]\label{lem:disc.complex}
  Under Assumptions \ref{assumption:homogeneous} and \ref{assumption:tracetrimmed}, for all $k \geq 1$, $d \geq k$, $f \in \FM{d}(\Mh)$, $r\ge 0$ and $\ul{\omega}_f \in \uH{k}{f}$, it holds
  \begin{align}
      P^k_{r,f} \ul{\DIFF}^{k-1}_{r,f} \ul{\omega}_f ={}& \DIFF^{k-1}_{r,f} \ul{\omega}_f, \label{eq:link.P.d}\\
      \ul{\DIFF}^k_{r,f} \ul{\DIFF}^{k-1}_{r,f} \ul{\omega}_f ={}& 0 . \nonumber
  \end{align}
\end{lemma}
\begin{proof}
  See \cite[Theorem~23]{Bonaldi.Di-Pietro.ea:23}.
\end{proof}
  The complex property is primordial to the whole construction of the DDR sequence. 
  This relation is essentially combinatoric: 
  composing two discrete derivatives on a $d$-cell leave out integral terms on the boundary $(d-2)$-cells, each one appearing twice (once for each boundary $(d-1)$-cell it belongs to) with opposite signs, thus cancelling out.
  While it would look straightforward from \eqref{eq:def.d}, 
  we must be careful since the final operator $\ul{\DIFF}^k_{r,f}$ is a projection of the local discrete exterior derivative $\DIFF^k_{r,f}$.
  This is why the complex property result is intertwined with the link 
  $P^k_{r,f}\ul{\DIFF}^{k-1}_{r,f} = \DIFF^{k-1}_{r,f}$ between potential reconstruction and discrete exterior derivative, 
  showing that no information is lost during the projection step. As a matter of fact, the proofs of both properties has to be done simultaneously.
\begin{lemma}\label{lem:proj.pot}
  Under Assumptions \ref{assumption:homogeneous} and \ref{assumption:tracetrimmed}, for all $d \geq 0$, $k\in [0,d]$, $f \in \FM{d}(\Mh)$, $r\ge 0$ and $\ul{\omega}_f \in \uH{k}{f}$, it holds
  \begin{equation*}
    \star^{-1}\ltproj{d-k}{f}\star P^k_{r,f} \ul{\omega}_f = \omega_f
  \end{equation*}
\end{lemma}
\begin{proof}
  See \cite[Lemma~24]{Bonaldi.Di-Pietro.ea:23}.
\end{proof}
  Each cell component encodes the moment against a trimmed polynomial space.
  The role of the potential reconstruction operator is to use the data on the boundary 
  to enrich the information in the cell from the trimmed to the full polynomial space. 
  Lemma \ref{lem:proj.pot} shows that this process does not alter the data already available on the cell, 
  and that the potential reconstruction is actually a higher-order enhancement of that data (see \cite[Lemma~24]{Bonaldi.Di-Pietro.ea:23} for a more detailed formula regarding this enhancement).
\begin{theorem}[Commutation property]\label{thm:commute}
  Under Assumptions \ref{assumption:homogeneous} and \ref{assumption:tracetrimmed}, for all $k \geq 0$, $d > k$, $f\in\FM{d}(\Mh)$ and $r\ge 0$, it holds
  \begin{equation}\label{eq:commutation}
    \ul{\DIFF}^k_{r,f} (\uI{k}{f} \omega) = \uI{k+1}{f} (\DIFF \omega) \quad \forall\omega\in C^1\Lambda^k(\overline{f}).
  \end{equation}
\end{theorem}
\begin{proof}
  See \cite[Theorem~25]{Bonaldi.Di-Pietro.ea:23}.
\end{proof}
  The complex \eqref{eq:DDR.complex} being fully discrete, 
  the only direct link with the continuous de Rham complex is through the interpolator. 
  Theorem \ref{thm:commute} ensures that 
  the local algebraic structure of the discrete complex is consistent with the continuous one. 
  This property is key in the construction of structure-preserving and robust discretizations; see, e.g., \cite{Di-Pietro.Droniou:21*1,Beirao.Dassi.ea:22,DiPietro.ea:24}.
\begin{theorem}[Isomorphism of cohomologies]\label{thm:iso.coho}
  Under Assumptions \ref{assumption:homogeneous} and \ref{assumption:tracetrimmed}, for all $r\ge 0$ the DDR sequence \eqref{eq:DDR.complex} is a complex that has the same cohomology as the continuous de Rham complex
  \eqref{eq:dR}.
\end{theorem}
\begin{proof}
  See \cite[Theorem~14]{Bonaldi.Di-Pietro.ea:23}.
\end{proof}
  The cohomology characterises the intersection of the kernel of the exterior derivative and of its adjoint, 
  or equivalently the kernel of the Hodge-Laplacian.
  Therefore, preserving the global structure is important for the stability and the well-posedness of the systems built upon the complex.
  Theorem \ref{thm:iso.coho} implies that the kernel of the discrete Hodge-Laplacian is determined by the topology of the domain, 
  in a correct way that is compatible with the continuous one. 
  Notice that the result holds for all $r \geq 0$, and thus for $r = 0$. In that case, 
  the only non trivial components of a $k$-form are the spaces $\mathcal{H}_0\Lambda^0$ lying on $k$-cells. 
  This shows that the topological information is fully carried by the lowest-order component on the $k$-skeleton, 
  and that higher-order contributions are essentially local. See \cite[Lemma 27]{Bonaldi.Di-Pietro.ea:23} for more on this aspect.

\subsection{Consistency}\label{sec:consistency}

In the usual DDR construction, the initial consistency result for local discrete potentials and exterior derivatives is an exact polynomial consistency \cite[Theorem 15]{Bonaldi.Di-Pietro.ea:23}.
However, on manifolds, Hodge stars of polynomials are not necessarily polynomials; for this reason, it does not seem possible to obtain an exact
polynomial consistency of these local operators (see Remark \ref{rem:polynomial.consistency} for more details).
Instead, we directly establish a consistency result in terms of approximation
properties of these operators when applied to interpolates of smooth differential forms -- a result which is usually seen as a consequence of a polynomial consistency, but independently established here.
 
As the DDR construction relies on the decompositions in Definition \ref{def:poly.trimmed},
we must ensure that the direct sum does not degenerate as $h$ goes to zero. 
The setting in Section \ref{sec:construction.polynomial} gives a construction of local polynomial spaces that satisfy the following assumption.
Below, $\norm{X}{\mu}$ is a shorthand for the $L^2$-norm of $\mu\in L^2\Lambda^k(X)$.

\begin{assumption}[Local Poincar\'e inequality on Koszul complement, topological decomposition and discrete trace] \label{assumption:local.spaces}
  We assume the following bounds on the discrete spaces:
      For all $k\ge 0$, $d\ge k$, $f\in\FM{d}(\Mh)$ and $r\ge 0$ we have 
  \begin{alignat}{3}
    \norm{f}{\mu} &\lesssim h_f \norm{f}{\DIFF\mu},&&\quad \forall \mu\in\KOSZUL\PL{r-1}{k}(f),\label{eq:bound.d}\\
    \norm{f}{\alpha} + \norm{f}{\beta} &\lesssim \norm{f}{\alpha+\beta},&&\quad \forall \alpha\in\KOSZUL\PL{r-1}{k+1}(f), \quad\forall \beta\in\DIFF\PL{r+1}{k-1}(f),\label{eq:top.decomp}\\
    \norm{\pf}{\tr_\pf\mu} &\lesssim h_f^{-\frac12}\norm{f}{\mu}, 
    &&\quad \forall\mu\in\PL{r}{k}(f).\label{eq:tr.poly}
  \end{alignat}
\end{assumption}
We also assume that the discrete spaces have similar approximation properties as those of polynomials in Euclidean space (see \cite[Theorem 1.45]{Di-Pietro.Droniou:20}).
\begin{assumption}[Approximation properties]\label{assumption:local.approx}
  For all $0 \leq k \leq n$, $d\ge \min(1,k)$, $f\in\FM{d}(\Mh)$, $r\ge 0$ and $t\in [0,r]$
  \begin{equation}\label{eq:approx}
    \norm{f}{\star^{-1}\lproj{d-k}{f}\star\omega - \omega} \lesssim h_f^{t+1} \seminorm{H^{t+1}\Lambda^k(f)}{\omega}
    \quad \forall \omega \in H^{t+1}\Lambda^k(f),
  \end{equation}
  where $\lproj{d-k}{f}$ is the $L^2\Lambda^{d-k}(f)$-projection on $\PL{k}{d-k}(f)$. Note that, if $d=k=0$, we always have $\star^{-1}\lproj{0}{f}\star\omega = \omega$ since $\omega$ is then just a value at the vertex $f$.
\end{assumption}

To state the consistency result, we need a scaled norm that takes into account the regularity of differential forms and their traces on sub-dimensional cells. For $r\in\Natural$, we define the space
\[
H^{r+1}\Lambda^k(f;\Delta)\coloneq \left\{\omega\in H^{r+1}\Lambda^k(f)\,:\,\tr_{f'}\omega\in H^{r+1}\Lambda^k(f')\quad\forall f'\in\FM{d'}(f)\,,\forall d'\in [k,d-1]\right\}
\]
(note that the assumed regularity on $f$ ensures the existence of traces on $f'\in\FM{d-1}(f)$, and the assumed regularity of these traces in turn ensures the existence of traces on $f''\in\FM{d-2}(f)$, etc.), and endow it with the semi-norm
\begin{equation}\label{eq:def.scaled.norm}
  \seminorm{r+1,f,\Delta}{\omega} := \sum_{d'=k}^d h_f^{\frac{d-d'}{2}} \seminorm{H^{r+1}\Lambda^k(f')}{\tr_{f'}\omega}\quad\forall\omega\in H^{r+1}\Lambda^k(f;\Delta).
\end{equation}
\begin{theorem}[Primal consistency]\label{thm:primal.consistency}
  Under Assumptions \ref{assum:reg.seq}, \ref{assumption:homogeneous}, \ref{assumption:tracetrimmed}, \ref{assumption:local.spaces} and \ref{assumption:local.approx}, for all $k \geq 0$, $d\geq k$, $f\in\FM{d}(\Mh)$ and $r\ge 0$, it holds
    \begin{align}
      \norm{f}{P^k_{r,f} \uI{k}{f} \omega - \omega}&\lesssim h_f^{r+1} \seminorm{r+1,f,\Delta}{\omega} \quad\forall\omega \in H^{r+1}\Lambda^k(f;\Delta),\label{eq:consistency.P} \\
      \norm{f}{\DIFF^k_{r,f} \uI{k}{f} \omega - \DIFF\omega}&\lesssim h_f^{r+1} \seminorm{r+1,f,\Delta}{\DIFF\omega}\quad\forall\omega\in C^1\Lambda^k(\overline{f})\mbox{ s.t. } \DIFF\omega \in H^{r+1}\Lambda^{k+1}(f;\Delta)\label{eq:consistency.d}.
    \end{align}
\end{theorem}

\begin{proof}
  The proof if done by induction on $d$.  
  If $d=k$, then both discrete and continuous exterior derivatives vanish, so \eqref{eq:consistency.d} is trivially satisfied.
  The definitions \eqref{eq:def.I} and \eqref{eq:def.P.d=k} of $\uI{k}{h}$ and $P^k_{r,f}$ combined with the assumed approximation property \eqref{eq:approx} of the $L^2$-projector on polynomial spaces readily gives \eqref{eq:consistency.P}.
  
  Let us now take $d\ge k+1$ and assume that the result holds for $d-1$. Let $f\in\FM{d}(\Mh)$. We first prove a weaker (sub-optimal) version of \eqref{eq:consistency.d}, from which we deduce \eqref{eq:consistency.P} which, finally, gives the optimal estimate \eqref{eq:consistency.d} itself.
  
  The definitions \eqref{eq:def.d} of the discrete exterior derivative and \eqref{eq:def.I} of the interpolator yield,
  for all $\mu_f\in\PL{r}{d-k-1}(f)$, 
  \begin{align}
    \int_f \DIFF^k_{r,f} \uI{k}{f} \omega \wedge \mu_f 
    &= (-1)^{k+1} \int_f \star^{-1}\ltproj{d-k}{f}\star\omega \wedge\DIFF\mu_f
    + \int_\pf P^k_{r,\pf} \uI{k}{\pf}\omega\wedge\tr_\pf\mu_f\nonumber\\
      &=
      (-1)^{k+1}\int_f \omega\wedge\DIFF\mu_f + \int_\pf\tr_\pf\omega\wedge\tr_\pf\mu_f \nonumber\\
      &\quad + \int_\pf(P^k_{r,\pf} \uI{k}{\pf}\omega-\tr_\pf\omega)\wedge\tr_\pf\mu_f \nonumber\\
      &= \int_f\DIFF\omega\wedge\mu_f 
      + \int_\pf(P^k_{r,\pf} \uI{k}{\pf}\omega-\tr_\pf\omega)\wedge\tr_\pf\mu_f
      \label{eq:cons.d.sub1}
  \end{align}
  where, in the second equality, we have used the fact that $\DIFF\mu_f\in\PLtrimmed{r}{d-k}(f)$ together with the definition of the orthogonal projector $\ltproj{d-k}{f}$ to remove $\star^{-1}\ltproj{d-k}{f}\star$ (see \cite[Lemma 1]{Bonaldi.Di-Pietro.ea:23} for details), and we have used the Stokes formula in the conclusion.
  Subtracting 
  \[
    \int_f\DIFF\omega\wedge\mu_f = \int_f \star^{-1}\lproj{d-k-1}{f}\star\DIFF\omega \wedge\mu_f
 \]
 (the equality following from the definition of $\lproj{d-k-1}{f}$ and the fact that $\mu_f\in \PL{r}{d-k-1}(f)$)
 to both sides of \eqref{eq:cons.d.sub1} and using a Cauchy-Schwarz inequality then gives
  \begin{align}
      \int_f (\DIFF^k_{r,f} \uI{k}{f} \omega - \star^{-1}\lproj{d-k-1}{f}\star\DIFF\omega)\wedge \mu_f 
      &\leq \norm{\pf}{P^k_{r,\pf} \uI{k}{\pf}\omega-\tr_\pf\omega}\norm{\pf}{\tr_\pf\mu_f}\nonumber \\
      &\lesssim 
      h_f^{r+1}\seminorm{r+1,\pf,\Delta}{\tr_\pf\omega}\norm{\pf}{\tr_\pf\mu_f}\nonumber\\
      &\lesssim 
      h_f^{r}\seminorm{r+1,f,\Delta}{\omega}\norm{f}{\mu_f}, \label{eq:pc.P.0}
  \end{align}
  where we have used the induction hypothesis \eqref{eq:consistency.P} on all $f'\in\pf$ (the notation $\seminorm{r+1,\pf,\Delta}{{\cdot}}$ stands for $\sum_{f'\in\FM{d-1}(f)}\seminorm{r+1,f',\Delta}{{\cdot}}$) and the conclusion follows from the discrete trace inequality \eqref{eq:tr.poly} together with the definition \eqref{eq:def.scaled.norm} of the scaled norm, which gives
  \begin{equation}\label{eq:hierarchichal.scaled} 
    \seminorm{r+1,\pf,\Delta}{\tr_{\pf}\omega}\lesssim h_f^{-1/2}\seminorm{r+1,f,\Delta}{\omega}.
  \end{equation}
  With $\mu_f = \star\DIFF^k_{r,f} \uI{k}{f} \omega - \lproj{d-k-1}{f}\star\DIFF\omega\in\PL{r}{d-k-1}(f)$, the left-hand side of \eqref{eq:pc.P.0} is equal to $\norm{f}{\DIFF^k_{r,f} \uI{k}{f} \omega - \star^{-1}\lproj{d-k-1}{f}\star\DIFF\omega}^2$. Simplifying, we therefore obtain
  \[
    \norm{f}{\DIFF^k_{r,f} \uI{k}{f} \omega - \star^{-1}\lproj{d-k-1}{f}\star\DIFF\omega}
      \lesssim
      h_f^{r}\seminorm{r+1,f,\Delta}{\omega}.
  \]
  Introducing $\pm \DIFF\omega$ and using a triangle inequality together with the approximation property \eqref{eq:approx} (with $\DIFF\omega$ instead of $\omega$ and $t=r-1$ if $r\ge 1$; if $r=0$, we simply use the $L^2(f)$-boundedness of $\lproj{d-k-1}{f}$), we infer the following sub-optimal version of \eqref{eq:consistency.d}:
  \begin{equation}\label{eq:consistency.d.sub}
    \norm{f}{\DIFF^k_{r,f} \uI{k}{f} \omega - \DIFF\omega}
      \lesssim
      h_f^{r}\seminorm{r+1,f,\Delta}{\omega}.
  \end{equation}
  We can now prove \eqref{eq:consistency.P}. 
  For all $\mu_f\in\KOSZUL\PL{r}{d-k}(f)$ and $\nu_f\in\KOSZUL\PL{r-1}{d-k+1}(f)$, 
  the definitions \eqref{eq:def.P} of $P^k_{r,f}$ and \eqref{eq:def.I} of $\uI{k}{f}$ yield
  \begin{equation*}
    \begin{aligned}
      (-1)^{k+1}&\int_f P^k_{r,f}\uI{k}{f}\omega\wedge(\DIFF\mu_f+\nu_f)\\
      &= \int_f\DIFF^k_{r,f}\uI{k}{f}\omega\wedge\mu_f-\int_\pf P^k_{r,\pf}\uI{k}{\pf}\omega\wedge\tr_\pf\mu_f
      +(-1)^{k+1}\int_f \star^{-1}\ltproj{d-k}{f}\star\omega\wedge\nu_f \\
      \overset{\nu_f\in\PLtrimmed{r}{d-k}(f)}
      &= \int_f\DIFF\omega\wedge\mu_f -\int_\pf\tr_\pf\omega\wedge\tr_\pf\mu_f 
      + (-1)^{k+1}\int_f\omega\wedge\nu_f\\
      &\quad- \int_\pf(P^k_{r,\pf}\uI{k}{\pf}\omega-\tr_\pf\omega)\wedge\tr_\pf\mu_f 
      + \int_f(\DIFF^k_{r,f}\uI{k}{f}\omega-\DIFF\omega)\wedge\mu_f\\
      &= (-1)^{k+1}\int_f\omega\wedge(\DIFF\mu_f+\nu_f)
      - \int_\pf(P^k_{r,\pf}\uI{k}{\pf}\omega-\tr_\pf\omega)\wedge\tr_\pf\mu_f \\
      &\quad+ \int_f(\DIFF^k_{r,f}\uI{k}{f}\omega-\DIFF\omega)\wedge\mu_f,
    \end{aligned}
  \end{equation*}
  the conclusion following from the Stokes formula.
  A Cauchy-Schwarz inequality then gives
  \begin{align}
      (-1)^{k+1}&\int_f (P^k_{r,f}\uI{k}{f}\omega-\omega)\wedge(\DIFF\mu_f+\nu_f)\nonumber\\
      &\leq \norm{\pf}{P^k_{r,\pf}\uI{k}{\pf}\omega-\tr_\pf\omega}\norm{\pf}{\tr_\pf\mu_f}
      + \norm{f}{\DIFF^k_{r,f}\uI{k}{f}\omega-\DIFF\omega}\norm{f}{\mu_f}\nonumber\\
      &\lesssim
      h_f^{r+1}\seminorm{r+1,\pf,\Delta}{\tr_\pf\omega}\norm{\pf}{\tr_\pf\mu_f}
      + h_f^r \seminorm{r+1,f,\Delta}{\omega}\norm{f}{\mu_f} \nonumber\\
      &\lesssim 
      h_f^r \seminorm{r+1,f,\Delta}{\omega}\norm{f}{\mu_f},
      \label{eq:pc.P.2}
  \end{align}
  where we have used \eqref{eq:consistency.P} for $d-1$ (induction hypothesis) together with \eqref{eq:consistency.d.sub} in the second inequality, and the discrete trace inequality \eqref{eq:tr.poly} together with \eqref{eq:hierarchichal.scaled} in the conclusion.
  By Lemma \ref{lemma:poly.dec}, we can take $\mu_f$ and $\nu_f$ such that $\DIFF\mu_f + \nu_f = (-1)^{k+1} (\star P^k_{r,f}\uI{k}{f}\omega-\lproj{d-k}{f}\star\omega)$, and \eqref{eq:bound.d} and \eqref{eq:top.decomp} ensure that
  \begin{equation}\label{eq:pc.P.3}
    \norm{f}{\mu_f} \lesssim h_f\norm{f}{\DIFF\mu_f}
    \lesssim h_f\norm{f}{\star P^k_{r,f}\uI{k}{f}\omega-\lproj{d-k}{f}\star\omega}=h_f\norm{f}{P^k_{r,f}\uI{k}{f}\omega-\star^{-1}\lproj{d-k}{f}\star\omega}.
  \end{equation}
  We then write
  \begin{equation*}
    \begin{aligned}
      \norm{f}{P^k_{r,f}&\uI{k}{f}\omega-\star^{-1}\lproj{d-k}{f}\star\omega}^2\\
     &= \int_f(P^k_{r,f}\uI{k}{f}\omega-\star^{-1}\lproj{d-k}{f}\star\omega)\wedge(\DIFF\mu_f+\nu_f)\\
     &= \int_f(P^k_{r,f}\uI{k}{f}\omega-\omega)\wedge(\DIFF\mu_f+\nu_f)
      + \int_f(\omega-\star^{-1}\lproj{d-k}{f}\star\omega)\wedge(\DIFF\mu_f+\nu_f)\\
     \overset{\eqref{eq:pc.P.2}}&\lesssim
        h_f^r\seminorm{r+1,f,\Delta}{\omega}\norm{f}{\mu_f}
        + \norm{f}{\omega-\star^{-1}\lproj{d-k}{f}\star\omega}\norm{f}{\DIFF\mu_f+\nu_f} \\
      \overset{\eqref{eq:pc.P.3},\eqref{eq:approx}}&\lesssim
      h_f^{r+1}\seminorm{r+1,f,\Delta}{\omega}\norm{f}{P^k_{r,f}\uI{k}{f}\omega-\star^{-1}\lproj{d-k}{f}\star\omega}.
    \end{aligned}
  \end{equation*}
  Simplifying gives $\norm{f}{P^k_{r,f}\uI{k}{f}\omega-\star^{-1}\lproj{d-k}{f}\star\omega}\lesssim h_f^{r+1}\seminorm{r+1,f,\Delta}{\omega}$ and the bound \eqref{eq:consistency.P} then follows by invoking \eqref{eq:approx} with $t=r$ and the triangle inequality
  \[
    \norm{f}{P^k_{r,f}\uI{k}{f}\omega-\omega} \leq 
    \norm{f}{P^k_{r,f}\uI{k}{f}\omega-\star^{-1}\lproj{d-k}{f}\star\omega}
    + \norm{f}{\star^{-1}\lproj{d-k}{f}\star\omega-\omega}.
  \]
  
  Finally, \eqref{eq:consistency.d} follows by applying \eqref{eq:consistency.P} to $\DIFF\omega$ and $k+1$ instead of $\omega$ and $k$, and by using the commutation property \eqref{eq:commutation} and the link \eqref{eq:link.P.d} (also with $k+1$ instead of $k$) between potential reconstruction and discrete exterior derivative.
\end{proof}

\begin{remark}[Polynomial consistency] \label{rem:polynomial.consistency}
Polynomial consistency is the property: $P^k_{r,f} \uI{k}{f} \omega=\omega$ for all $\omega\in\PL{r}{k}(f)$.
As shown in \cite[Section 3.5]{Bonaldi.Di-Pietro.ea:23}, the initial step for proving this property for the DDR complex on the Euclidean space consists in considering $d=k$ and in writing $\star^{-1}\lproj{d-k}{f}\star\omega=\star^{-1}\star\omega=\omega$; the removal of the $L^2$-projector $\lproj{d-k}{f}$ is justified since, when the metric is constant, $\star\omega$ is a polynomial form of the same degree as $\omega$. However, in the context of DDR on a manifold, $\star\omega$ may no longer be polynomial of the same degree as $\omega$ since the $\star$ operator involves the coefficients of the (non-constant) metric.
Hence, polynomial consistency for the DDR on a manifold does not seem readily accessible. This does not prevent the method, however, from having optimal primal consistency on smooth forms, as demonstrated by Theorem \ref{thm:primal.consistency}.
\end{remark}

\begin{remark}[About the regularity requirement on $\omega$]
The space $H^{r+1}\Lambda^k(f;\Delta)$, with its assumed regularity of traces of functions on lower-dimensional cells, is the natural one to state the consistency estimates in Theorem \ref{thm:primal.consistency}. Classical spaces of differential forms can easily be embedded into $H^{r+1}\Lambda^k(f;\Delta)$.

For example, letting $C^{r+1}\Lambda^k(\overline{f})$ be the space of $k$-forms that are continuous over $\overline{f}$ along with all their derivatives up to order $r+1$, and considering the semi-norm
$$
\seminorm{C^{r+1}\Lambda^k(f)}{\omega}=\max_{\alpha\in\Natural^d,\,|\alpha|=r+1}\norm{L^\infty\Lambda^k(f)}{\partial^\alpha\omega},
$$
we have $C^{r+1}\Lambda^k(\overline{f})\subset H^{r+1}\Lambda^k(f;\Delta)$ and
\[
\seminorm{r+1,f,\Delta}{\omega}\lesssim |f|^{1/2}\seminorm{C^{r+1}\Lambda^k(f)}{\omega}\quad\forall\omega\in C^{r+1}\Lambda^k(\overline{f}).
\]
Likewise, repeated uses of the continuous trace inequality
\[
\norm{\pf'}{\tr_{\pf'}\mu} \lesssim h_{f'}^{-\frac12}\norm{f'}{\mu} + h_{f'}^\frac12 \seminorm{{H^{1}\Lambda^k(f')}}{\mu}, 
    \quad \forall\mu\in H^1\Lambda^k(f')\,,\quad\forall f'\in\Delta_{d'}(\Mh)\,,\quad\forall d'\in [k,d]
\]
  show that $H^{r+1+d-k}\Lambda^k(f)\subset H^{r+1}\Lambda^k(f;\Delta)$ with
\[
\seminorm{r+1,f,\Delta}{\omega}\lesssim \sum_{i=0}^{d-k} h_f^{i} \seminorm{H^{r+1+i}\Lambda^k(f)}{\omega}\quad\forall\omega\in H^{r+1+d-k}\Lambda^k(f).
\]
\end{remark}

\section{Construction of local polynomial spaces on manifolds}\label{sec:construction.polynomial}

The notion of polynomial depends on the choice of coordinates.
In general, a manifold cannot be covered with a single chart.
Therefore, if polynomial spaces are to be designed on a manifold, their definition needs to be coherent with (some) changes of coordinates.
We only need our polynomial spaces to be locally defined on $d$-cells, 
but we also need to ensure that restrictions of polynomials to the boundary of $d$-cells are polynomials on $(d-1)$-cells 
(of the same or lower degree), for all $d=1,\ldots,n$.

We prove in this section that, if Assumption \ref{assumption:straighten} below holds, then suitable local polynomial spaces can be constructed.
This assumption essentially states that, when read in the coordinates chosen on a $d$-cell $f$, the subcells of $f$ are linearly embedded in $f$.
Appendix \ref{sec:example.construction} shows how to practically construct mappings satisfying this assumption in dimension $n=2$ and for various cell shapes.
Notice that, if a cell is a polytope in all the charts of the chosen atlas on the manifold, 
then Assumption \ref{assumption:straighten} is trivially satisfied.
Therefore, we can automatically build a basis on cells that are flat shaped in all charts (even if the metric is not trivial). 
Only cells that are not flat polytopes in one of the atlas' charts require a special attention; these cells usually lie at the transition between two charts.

\begin{assumption}[Affine compatible local coordinates]\label{assumption:straighten}
For all $0 \leq d \leq n$ and $f \in \FM{d}(\Mh)$, there is a $C^2$-diffeomorphism $I_f:U_f\subset\Real^d\to \overline{f}$. Moreover,
for all $f' \in \FM{d-1}(f)$, setting $J_f := (I_f)^{-1}$ and denoting by $\INJ_{f,f'}\st f' \to \overline{f}$ the inclusion of $f'$ into $\overline{f}$,
the mapping $\TRF_{f,f'} := J_f \circ \INJ_{f,f'} \circ I_{f'}:U_{f'}\to \Real^d$ is affine.
\end{assumption}
  This assumption essentially states that each cell has a suitable coordinate system (represented by the diffeomorphisms $I_f$) such that, for
  each $(d-1)$-cell $f'$ of $f$, the coordinate system of $f'$ is affine in the coordinates of $f$. This key property ensures that, when defining
  polynomial functions through these coordinate systems (as in Definition \ref{def:poly.spaces.koszul}), the restriction to $f'$ of polynomials on $f$
  are polynomial on $f'$ (see Lemmas \ref{lemma:polymanifold.trPL} and \ref{lem:trace.trimmed}), which is one of the essential properties required on local polynomial spaces to build the DDR sequence (and potentially other polytopal schemes).

  Appendix \ref{sec:example.construction} gives a possible practical approach to construct local coordinates (that is, parametrisations) that satisfy Assumption \ref{assumption:straighten}, by a hierarchical construction: starting from parametrisations of the $1$-cells (edges), 
  parametrisations of $2$-cells are constructed using the coordinates originating from two of its edges, in an explicit way and taking care of respecting the compatibility with the coordinates on the other edges. Note that this construction builds at the same time a flat polytopal mesh
  that is equivalent to the manifold mesh in the sense of Definition \ref{def:equiv.meshes}
  (provided the diffeomorphisms $(J_f)_{f\in\Delta_d(\Mh),\,d\in [0,n]}$ satisfy \eqref{eq:shape.reg}); hence, following the process in Appendix \ref{sec:example.construction}, we not only satisfy Assumptions \ref{assumption:homogeneous} and \ref{assumption:tracetrimmed} (directly implied by Assumption \ref{assumption:straighten} as shown in Lemmas  \ref{lemma:polymanifold.trPL} and \ref{lem:trace.trimmed}), but also Assumption \ref{assumption:local.spaces} (see Lemma \ref{lem:local.bounds}).
  Only the approximation properties of Assumption \ref{assumption:local.approx} must be checked independently.

  In the following, we detail the construction of the polynomial spaces and the implications of Assumption \ref{assumption:straighten}.
We denote by $\PL{r}{l}(\Real^d)$, the set of $l$-forms with coefficients that are polynomial on $\Real^d$ of degree at most $r$,
and $\HL{r}{l}(\Real^d)$ the corresponding space of homogeneous polynomial forms of degree $r$.
We also denote by $X_d$ the identity vector field on $\Real^d$, 
that is, $X_d(\bvec{x}) = \bvec{x}$, for all $\bvec{x} \in \Real^d$.
We will often take the pullback of these object with functions that are not surjective on $\Real^d$ (which is not an issue).
Finally, recall that the interior product of a vector $X$ with an $\ell$-alternating form $\mu$ is the $(\ell-1)$-alternating form given by 
\begin{equation}\label{eq:contraction}
\contr{X}\mu(v_1,\ldots,v_{\ell-1})=\mu(X,v_1,\ldots,v_{\ell-1}).
\end{equation}

\begin{definition}[Polynomial spaces and Koszul operator via local charts]\label{def:poly.spaces.koszul}
  Under Assumption \ref{assumption:straighten}, we define
  the polynomial space on $f$ by:
  \begin{align}
    \HL{r}{l}(f) :={}& J_f^*\HL{r}{l}(\Real^d) \label{eq:def.HLf},\\
    \PL{r}{l}(f) :={}& \bigoplus_{s \leq r} \HL{s}{l}(f)\label{eq:def.PLf},\\
    \KOSZUL_f :={}& \contr{X_{f}}\mbox{}, \text{ with } X_{f} := {(I_f)}_*X_d .\label{eq:def.kf}
  \end{align}
  In particular, we also have $\PL{r}{l}(f)=J_f^*\PL{r}{l}(\Real^d)$.
\end{definition}

Let us check that this construction satisfies Assumption \ref{assumption:homogeneous} and \ref{assumption:tracetrimmed}.
Moreover, if we assume that the diffeomorphisms $J_f$ are regular, 
then this construction also satisfies Assumption \ref{assumption:local.spaces}.

\begin{lemma}[Affine compatible local coordinates build homogeneous polynomials]
  Assumption \ref{assumption:straighten} implies Assumption \ref{assumption:homogeneous}.
\end{lemma}
\begin{proof}
  We first prove that, for the polynomial spaces in Definition \ref{def:poly.spaces.koszul}, \eqref{eq:polynomial.complex.diff} and \eqref{eq:polynomial.complex.koszul} are well-defined complexes.
  Since the exterior derivative $\DIFF$ commutes with pullbacks, 
  \eqref{eq:def.HLf} and the standard inclusion $\DIFF \HL{s}{l}(\Real^d) \subset \HL{s-1}{l+1}(\Real^d)$ give, for any $(s,l)$,
  \begin{equation*} 
    \DIFF\HL{s}{l}(f)=\DIFF J_f^*\HL{s}{l}(\Real^d) = J_f^* \DIFF \HL{s}{l}(\Real^d) \subset J_f^*\HL{s-1}{l+1}(\Real^d)
    = \HL{s-1}{l+1}(f).
  \end{equation*}
  By taking the sum over $s\le r$ on both sides and recalling \eqref{eq:def.PLf},
  we infer that the spaces $(\PL{r}{l}(f))_{r\in\Zintegers,l\in\Zintegers}$ 
  form a complex for $\DIFF$.

  If $g$ is a diffeomorphism and $v$ a vector field, then it holds that 
  \begin{equation}\label{eq:commute.contraction.pull}
    \contr{(g^{-1})_* v} g^* = g^* \contr{v}.
  \end{equation}
  Applying this result to a generic $J_f^* \alpha \in \PL{r}{l}(f)$, with $\alpha\in\PL{r}{l}(\Real^d)$, and recalling the definition \eqref{eq:def.kf} of $\KOSZUL_f$ gives
  \begin{equation}\label{eq:pullback.interior}
    \KOSZUL_f J_f^* \alpha = \contr{(I_f)_*X_d}(J_f^*\alpha) = J_f^*(\contr{X_d} \alpha)
    \in J_f^* \PL{r+1}{l-1}(\Real^d) = \PL{r+1}{l-1}(f).
  \end{equation}
  This shows that the spaces $(\PL{r}{l}(f))_{r\in\Zintegers,l\in\Zintegers}$ 
  form a complex for $\KOSZUL_f$.
  
  Let us show that the graded decomposition 
  $\PL{r}{l}(f) = \bigoplus_{s \leq r} \HL{r}{l}(f)$
  satisfies Assumption \ref{A:homogeneous.polynomials}. 
  For any $p \in \HL{s}{l}(f)$, there is $\alpha \in \HL{s}{l}(\Real^d)$ such that
  $J_f^*\alpha = p$, and we have
  \[
    (\DIFF\KOSZUL_f + \KOSZUL_f\DIFF)p
    = (\DIFF\KOSZUL_f + \KOSZUL_f\DIFF)J_f^*\alpha
    = J_f^* (\DIFF \contr{X_d} + \contr{X_d}\DIFF) \alpha
    = J_f^* (\lambda_{s,l}\, \alpha) 
    = \lambda_{s,l}p,
  \]
  where we have used \eqref{eq:pullback.interior} (with $\alpha$ and $\DIFF\alpha$) together with the commutation of pull-back and exterior derivative to write the second equality, 
  and the fact that $\HL{s}{l}(\Real^d)$ is the eigenspace associated with the eigenvalue $\lambda_{s,l}:=s+l$
  for $\DIFF \contr{X_d} + \contr{X_d}\DIFF$ for the third equality.
  Hence, $\HL{s}{l}(f)$ is an eigenspace of $\DIFF\KOSZUL_f + \KOSZUL_f\DIFF$ for this eigenvalue $\lambda_{s,l}$.

  As we just saw, the eigenvalues $(\lambda_{s,l})_{s,l}$
  are the same as the eigenvalues on the flat space $\Real^d$, for which we know that Assumption \ref{A:eigenvalues} holds.

  Finally, let $(s,l)$ with $s > 0$ and $l < d$.
  There is $\alpha \in \HL{s}{l}(\Real^d)$ such that $\DIFF\alpha \neq 0$.
  Moreover, $\DIFF J_f^*\alpha = J_f^* \DIFF \alpha$, and because $J_f$ is an isomorphism, $J_f^* \DIFF \alpha \neq 0$.
  Thus, Assumption \ref{A:d.non.zero} holds.
\end{proof}
\begin{lemma}[Affine compatible local coordinates preserve traces of polynomials] \label{lemma:polymanifold.trPL}
  If Assumption \ref{assumption:straighten} holds then, for all $f\in\Delta_d(\Mh)$, $l\le d$ and $r\ge 0$, we have $\tr_{f'} \PL{r}{l}(f) \subset \PL{r}{l}(f')$ for all $f' \in \FM{d-1}(f)$.
\end{lemma}
\begin{proof}
  A polynomial form $\eta\in\PL{r}{l}(f)$ is written as $\eta=J_f^*\alpha$ for some $\alpha \in \PL{r}{l}(\Real^d)$.
  Then, noting that 
  \begin{equation}\label{eq:compo.J.I}
    J_f\circ\INJ_{f,f'}=J_f\circ\INJ_{f,f'}\circ I_{f'}\circ J_{f'}=\TRF_{f,f'}\circ J_{f'},
  \end{equation}
  we have $\tr_{f'}\eta=\INJ_{f,f'}^*\eta  = \INJ_{f,f'}^*J_f^*\alpha=J_{f'}^*(\TRF_{f,f'}^*\alpha)$.  By assumption $\TRF_{f,f'} $ is affine, so 
  $\JAC_{x_{f'}} \TRF_{f,f'}$ is constant.
  Therefore $\TRF_{f,f'}^*\alpha$ is polynomial  of degree at most $r$, 
  and belongs in $\PL{r}{l}(\Real^{d-1})$.
  Hence,
  $\tr_{f'}\eta \in J_{f'}^*\PL{r}{l}(\Real^{d-1})=\PL{r}{l}(f')$ and the proof is complete.
\end{proof}
\begin{lemma}[Affine compatible local coordinates preserve traces of trimmed polynomials]\label{lem:trace.trimmed}
  Assumption \ref{assumption:straighten} implies Assumption \ref{assumption:tracetrimmed}.
\end{lemma}
\begin{proof}
  Let $f' \in \FM{d-1}(f)$.
  We need to prove that $\tr_{f'} \PLtrimmed{r}{l}(f) \subset \PLtrimmed{r}{l}(f')$.
  Since $\PLtrimmed{r}{l}(f) = \PL{r-1}{l}(f) + \KOSZUL_f \PL{r-1}{l+1}(f)$ by Lemma \ref{lemma:poly.dec} and Definition \ref{def:poly.trimmed}, 
  and $\tr_{f'} \PL{r-1}{l}(f) \subset \PL{r-1}{l}(f')$ by Lemma \ref{lemma:polymanifold.trPL}, 
  we only need to show that
  $\tr_{f'} \KOSZUL_f \PL{r-1}{l+1}(f) \subset \PL{r-1}{l}(f') + \KOSZUL_{f'} \PL{r-1}{l+1}(f')$.

  By assumption $\TRF_{f,f'} $ is affine, 
  so there is a constant $C \in \Real^{d}$, such that for all $\bvec{x} \in J_{f'}(f') \subset \Real^{d-1}$,
  \begin{align*}
    \TRF_{f,f'} (\bvec{x}) ={}& \JAC_{\bvec{x}} (\TRF_{f,f'}) \bvec{x} + C,\\
    X_d (\TRF_{f,f'} (\bvec{x})) ={}& \JAC_{\bvec{x}} (\TRF_{f,f'}) X_{d-1}(\bvec{x}) + C.
  \end{align*}
  For any $\alpha \in \PL{r-1}{l+1}(\Real^d)$, 
  any point $x_{f'} \in f'$, 
  and any $l$-uplet of vectors $(v,\dots) \in (T_{x_{f'}} f')^l$, 
  by \eqref{eq:commute.contraction.pull}, \eqref{eq:pullback.interior} and \eqref{eq:compo.J.I}, we have
  \begin{align*}
    (\INJ_{f,f'}^* \contr{X_{f}} (J_f^* \alpha))_{x_{f'}} (v,\dots) 
    ={}& (\underbrace{\INJ_{f,f'}^* J_{f}^*}_{=J_{f'}^*\TRF_{f,f'}^*}(\contr{X_d} \alpha))_{x_{f'}} (v,\dots)\\
    ={}& \alpha_{\TRF_{f,f'}(J_{f'}(x_{f'}))} ( X_d(\TRF_{f,f'} (J_{f'}(x_{f'}))), 
      \JAC_{J_{f'}(x_{f'})} (\TRF_{f,f'}) \JAC_{x_{f'}} J_{f'}v, \dots)\\
    ={}& \alpha_{\TRF_{f,f'}(J_{f'}(x_{f'}))} (\JAC_{x_{f'}} (\TRF_{f,f'}) X_{d-1}(J_{f'} (x_{f'})), 
       \JAC_{J_{f'}(x_{f'})} (\TRF_{f,f'}) \JAC_{x_{f'}} J_{f'}v, \dots)\\
    &\quad +  \alpha_{\TRF_{f,f'}(J_{f'}(x_{f'}))} (C, \JAC_{J_{f'}(x_{f'})} (\TRF_{f,f'}) \JAC_{x_{f'}} J_{f'}v, \dots)\\
    ={}& (\TRF_{f,f'}^* \alpha)_{J_{f'}(x_{f'})} (X_{d-1}(J_{f'} (x_{f'})),\JAC_{x_{f'}} J_{f'}v, \dots)
     +  J_{f'}^* \TRF_{f,f'}^* \contr{C} \alpha_{x_{f'}} (v, \dots)\\
    ={}& J_{f'}^* \contr{X_{d-1}} I_{f'}^* \INJ_{f,f'}^* J_f^* \alpha_{x_{f'}} (v, \dots)
       +   \INJ_{f,f'}^* J_f^* \contr{C} \alpha_{x_{f'}} (v, \dots)\\
    ={}& \underbrace{\contr{X_{f'}}}_{=\KOSZUL_{f'}} \INJ_{f,f'}^* J_f^* \alpha_{x_{f'}} (v, \dots)
       + \INJ_{f,f'}^* J_f^* \contr{C} \alpha_{x_{f'}} (v, \dots).
  \end{align*}
  Hence,
  \[
    \tr_{f'} J_f^* \alpha = \KOSZUL_{f'} \tr_{f'} J_f^* \alpha + \tr_{f'} J_f^* \contr{C} \alpha.
  \]
  Since $C$ is a constant, we have $\contr{C} \alpha \in \PL{r-1}{l}(\Real^{d})$,
  and Lemma \ref{lemma:polymanifold.trPL} yields
  $\tr_{f'} J_f^* \contr{C} \alpha \in \PL{r-1}{l}(f')$ and 
  $\tr_{f'} J_f^* \alpha \in \PL{r-1}{l+1}(f')$.
\end{proof}

\begin{lemma}[Local Poincar\'e and trace inequalities under affine compatible local coordinates]\label{lem:local.bounds}
  If Assumptions \ref{assumption:straighten} and \ref{assum:reg.seq} hold, and the diffeomorphisms $(J_f)_{f\in\Delta_d(\Mh),\,d\in [0,n]}$ satisfy \eqref{eq:shape.reg}, 
  then Assumption \ref{assumption:local.spaces} holds.
\end{lemma}
\begin{proof}
  The main idea is to use the general bound 
  $\norm{f}{\phi^* u} \leq \norm{\infty}{\nabla\phi}^k\norm{\infty}{\det(\nabla\phi^{-1})}^\frac12\norm{f}{u}$ 
  for $u\in L^2\Lambda^k(f)$. 
  We infer from the regularity assumption \eqref{eq:shape.reg} that, for all $1\leq d\leq n$, all $f\in\FM{d}(\Mh)$, and all $p\in\PL{r}{k}(\Real^d)$, 
  we have
  \begin{equation}\label{eq:pullback.equiv.norm}
    \norm{f}{J_f^*p} \approx \norm{\infty}{\nabla J_f}^{k-\frac{d}{2}}\norm{J_f(f)}{p}.
  \end{equation}
  Moreover, \eqref{eq:pullback.interior} and the fact that the pullback commutes with the exterior derivative give, for all $t,\ell\in\Natural$
  \begin{equation}\label{eq:pullback.decomp}
    \KOSZUL\PL{t}{\ell}(f) = J_f^*\KOSZUL\PL{t}{\ell}(\Real^d), \quad
    \DIFF\PL{t}{\ell}(f) = J_f^*\DIFF\PL{t}{\ell}(\Real^d).
  \end{equation}
  Let us now prove \eqref{eq:bound.d}. We first note that this inequality is valid in $J_f(f)\subset \Real^d$; this was established (in the case of vector proxies and $d=3$) in \cite[Lemma 9]{Di-Pietro.Droniou:23}, using the transport technique of \cite[Lemma 1.25]{Di-Pietro.Droniou:20} that only requires $J_f(f)$ to contain a ball of radius $\gtrsim h_{J_f(f)}$ -- which is the case by \eqref{eq:shape.reg} and the mesh regularity assumption.  
  For all $\mu\in \KOSZUL\PL{r-1}{k}(f)$, the relation \eqref{eq:pullback.decomp} shows that $\mu=J_f^*p$ and $\DIFF\mu=J_f^*\DIFF p$ for some $p\in\KOSZUL\PL{r-1}{k}(\Real^d)$. Hence,
    \begin{align*}
      \norm{f}{\mu} 
      \overset{\eqref{eq:pullback.equiv.norm}}&\approx\norm{\infty}{\nabla J_f}^{k-\frac{d}{2}}\norm{J_f(f)}{p}\\
      \overset{\text{\eqref{eq:bound.d} for $J_f(f)$}} &\lesssim\norm{\infty}{\nabla J_f}^{k-\frac{d}{2}}h_{J_f(f)}\norm{J_f(f)}{\DIFF p}\\
      \overset{\eqref{eq:hf.equiv}}&\approx\norm{\infty}{\nabla J_f}^{k-\frac{d}{2}+1}h_f\norm{J_f(f)}{\DIFF p}\\
      \overset{\eqref{eq:pullback.equiv.norm}}&\approx h_f\norm{f}{\DIFF \mu}.
    \end{align*}

  We now turn to \eqref{eq:top.decomp}.
  For all $\mu = J_f^*p\in\KOSZUL\PL{r-1}{k+1}(f)$ and $\nu = J_f^*q\in\DIFF\PL{r}{k-1}(f)$ with $p,q$ in the corresponding spaces on $\Real^d$ instead of $f$, 
  we have
  \begin{equation*}
    \begin{aligned}
      \norm{f}{\mu}+\norm{f}{\nu} 
      \overset{\eqref{eq:pullback.equiv.norm}}&\approx\norm{\infty}{\nabla J_f}^{k-\frac{d}{2}}(\norm{J_f(f)}{p} + \norm{J_f(f)}{q}) 
      \overset{}\approx\norm{\infty}{\nabla J_f}^{k-\frac{d}{2}}\norm{J_f(f)}{p+q}
      \overset{\eqref{eq:pullback.equiv.norm}}\approx\norm{f}{\mu+\nu},
    \end{aligned}
  \end{equation*}
  where the second equality follows from the fact that \eqref{eq:top.decomp} is valid on $J_f(f)\subset\Real^d$, see \cite[Lemma 2]{Di-Pietro.Droniou:23}.

  Finally, we have to prove \eqref{eq:tr.poly}, which is done in a similar way.
  For all $\mu = J_f^*p\in\PL{r}{k}(f)$ and all $f'\in\pf$, we have, using the discrete trace inequality in $\Real^d$,
  \begin{equation*}
    \begin{aligned}
      \norm{f'}{\tr_{f'}\mu} 
      \overset{\eqref{eq:pullback.equiv.norm}}&\approx 
      \norm{\infty}{\nabla J_{f'}}^{k-\frac{d-1}{2}}\norm{J_{f'}(f')}{\tr_{f'}p}\\
      &\lesssim \norm{\infty}{\nabla J_{f'}}^{k-\frac{d-1}{2}}h_{J_{f'}(f')}^{-\frac12}\norm{J_f(f)}{p}\\
      \overset{\eqref{eq:hf.equiv}}&\lesssim
      \norm{\infty}{\nabla J_{f'}}^{k-\frac{d-1}{2}-\frac12}h_{f'}^{-\frac12}\norm{J_f(f)}{p}\\
      \overset{\eqref{eq:boundary.reg},\eqref{eq:pullback.equiv.norm}}&
      \lesssim h_{f}^{-\frac12}\norm{f}{\mu}.\qedhere
    \end{aligned}
  \end{equation*}
\end{proof}

\section{Application}\label{sec:application}

We present here a 2+1 model for the Maxwell equations on a manifold written in the language of differential forms, and use the DDR complex to design a scheme for this model.
We refer the reader to the notations recalled in Appendix \ref{sec:tensor.calculus} and used throughout this section.

\subsection{Electromagnetism in $2+1$ dimensions}
Following \cite{Gourgoulhon:10,Gourgoulhon:12}, 
we foliate a $3$-dimensional space-time manifold $M$ by level sets of a time function $t$. Letting\footnote{Throughout this section, we use lower case Greek letters ($\mu,\nu,\gamma$ etc.) to label space-time coordinate indices that run over $0,1,2$ while lower case Latin letters ($i,j,k$ etc.) will label spatial coordinate indices that run over $1,2$. } $g=g_{\mu\nu}dx^\mu \otimes dx^\nu$ denote a Lorentzian metric on $M$, we perform a $2+1$ decomposition on $g$ via
\begin{equation}\label{eq:g.2metric}
(g_{\mu \nu}) := \begin{pmatrix} -N^2 + \vert\beta\vert_{\gamma}^2 & \beta_j \\ \beta_i & \gamma_{i j} \end{pmatrix},
\end{equation}
where $\gamma=\gamma_{i j}dx^i\otimes dx^j$ is the induced metric on the $t=\text{constant}$ spatial surfaces, $N$ is the lapse, $\beta=\beta_i dx^i$ is the shift, and we use $(x^i)$ to denote spatial coordinates on these surfaces.  
We define a future pointing unit normal to the spatial surfaces by
\begin{equation} \label{eq:defn}
n := (-N\dt)^\sharp,
\end{equation}
and, in the following, we distinguish geometric objects associated with the spatial surface with a tilde; for example, 
$\spvol$ and
$\spsharp$ are the volume form and $\sharp$ operator associated with the metric $\gamma$ on the spatial surfaces. Further, we introduce the vector field $\partial_t$ via
\begin{equation*}
\partial_t = N n + \beta^{\spsharp},
\end{equation*}
and note that for any adapted coordinate system $(t,x^i)$ (i.e. which satisfies $\Lie_{\partial_t}x^i=0$), the vector field $\partial_t$ will coincide with the coordinate vector field $\partial_t$ associated to the coordinate system $(t,x^i)$. 

Letting $F= \frac{1}{2}F_{\mu\nu}dx^\mu\wedge dx^\nu$ denote the electromagnetic field tensor, the Maxwell equations can be expressed in terms of $F$ as \cite[Section 18.2]{Gourgoulhon:10}
\begin{subequations}\label{eq:Maxwell}
\begin{align}
\DIFF F ={}& 0, \label{eq:Maxwell.1} \\
\DIFF \star F ={}& \epsilon_0^{-1} \star \ul{j}, \label{eq:Maxwell.2}
\end{align}
\end{subequations}
where $\ul{j}$ is the electric $3$-current and $\epsilon_0$ is the permittivity of the medium, supposed constant here.
Following the usual convention in $3+1$ dimensions, 
we define the electric field $E$ and the magnetic field $B$ by
\begin{align}
E :={}& -\contr{n} F,\label{eq:defE} \\
B :={}& \contr{n}(\star F), \label{eq:defB}
\end{align}
where we refer to \eqref{eq:contraction} for the definition of the interior product $\contr{n}$.
Notice that, in $2+1$ dimensions, the magnetic field $B$ is a scalar field.
Below, we will use $\spE$ to denote the restriction 
of $E$ to the spatial surfaces, that is, $\spE(\tau) = \iota_{\tau}^* E$ where $\iota_{\tau}$ is the inclusion map of the spatial surface $t=\tau$. If $E$ is expressed in terms of the adapted coordinate $(t,x^i)$ as
$E= E_0 dt + E_i dx^i$,
then $\spE = E_i dx^i$ and $E=E_0 dt + \spE$.

We denote the spatial codifferential on $k$-forms by $\spdelta := (-1)^k \spstar^{-1} \spdiff \spstar$ and assume that the shift vanishes, that is, $\beta=0$.
It is proved in Appendix \ref{sec:tensor.calculus} the Maxwell equations \eqref{eq:Maxwell} can then be re-written in $2+1$ formulation as:
\begin{equation}
\label{eq:Maxwell.general}
\begin{aligned}
  \spdiff (N\spE) ={}& - \partial_t (B \spvol), \\
  -\spdelta \spE ={}& \frac{\rho}{\epsilon_0}, \\
   \spdelta (N B \spvol) ={}& \epsilon_0^{-1} \ul{J} + \spstar^{-1} \partial_t (\spstar \spE),
\end{aligned}
\end{equation}
where $\rho := - \ul{j}(n)$ is the electric charge density and $J := N(\ul{j} - n^\flat \rho)$ is the electric current density. 

For the numerical tests, we consider a metric that is independent of time with constant lapse $N\equiv c$. The system \eqref{eq:Maxwell.general} then reduces to the more familiar form
\begin{subequations}
\label{eq:Maxwell.Stationary}
\begin{align}
\spdiff \spE ={}& - \partial_t B',\\
-\spdelta \spE ={}& \frac{\rho}{\epsilon_0},  \label{eq:Maxwell.Stationary.constraint}\\
\spdelta B' ={}& \mu_0 \widetilde{J} + \frac{1}{c^2} \partial_t \spE,
\end{align}
\end{subequations}
where $B' := \frac{1}{c} B \spvol$ and $\mu_0 := \frac{1}{c^2 \epsilon_0}$ is the vacuum permeability. To simplify, we will also work in geometric units and thus take $c=\epsilon_0=1$.

We note that the following compatibility condition (from hereon assumed) on the source terms, following from the property $\spdelta^2=0$:
\begin{equation}\label{eq:compatibility}
\spdelta\, \widetilde{J} = - \partial_t \spdelta\, \spE = \partial_t \rho .
\end{equation}

We assume that the manifold $\Omega$ has no boundary. Denoting by $\langle\cdot,\cdot\rangle$ the $L^2$-inner product on spatial $k$-forms (for any $k$), recalling that $\spdelta$ is the adjoint of $\spdiff$ for this inner product, and assuming that the constraint \eqref{eq:Maxwell.Stationary.constraint} holds at time $t=0$, a weak formulation of \eqref{eq:Maxwell.Stationary} is: find $(\spE,\spB)\in C^1([0,T];\Lambda^1(\Omega))\times C^1([0,T];\Lambda^2(\Omega))$ such that, for all $(v^1,v^2)\in\Lambda^1(\Omega)\times\Lambda^2(\Omega)$,
\begin{subequations}
\label{eq:Maxwell.Weak}
\begin{align}
\langle \spdiff\spE, v^2\rangle ={}& - \langle\partial_t B',v^2\rangle,\\
\langle B',\spdiff v^1\rangle ={}& \langle \widetilde{J},v^1\rangle + \langle\partial_t \spE,v^1\rangle.
\label{eq:Maxwell.weak.E}
\end{align}
\end{subequations}
We note that the constraint \eqref{eq:Maxwell.Stationary.constraint} has been dropped from this formulation as it can be recovered
by selecting a generic (time-independent) $v^0\in\Lambda^0(\Omega)$ and setting $v^1=\spdiff v^0$ in \eqref{eq:Maxwell.weak.E}, to see that
\[
\langle \partial_t\spdelta \spE,v^0\rangle=\langle \partial_t\spE,\spdiff v^0\rangle
=\langle B',\cancel{\spdiff^2 v^0}\rangle-\langle \widetilde{J},\spdiff v^0\rangle
=-\langle \spdelta\widetilde{J}, v^0\rangle=-\langle\partial_t\rho,v^0\rangle,
\]
showing that $\partial_t\langle \spdelta \spE+\rho,v^0\rangle=0$ and thus, since we assumed that $\spdelta \spE+\rho=0$ at $t=0$, that 
$\langle \spdelta \spE+\rho,v^0\rangle=0$ at all time.

\subsection{Semi-discrete scheme}

We describe here a numerical scheme based on the discrete de Rham complex for 
the system \eqref{eq:Maxwell.Stationary}. This scheme is of arbitrary order of accuracy, applicable to polygonal meshes on the chosen 2D manifold, and naturally preserves the constraint \eqref{eq:Maxwell.Stationary.constraint} at the discrete level.

Recalling the definition of the discrete $L^2$-like inner product (Definition \ref{def:inner.product}), the (semi-discrete) scheme is built on the product space $X_h := \uH{1}{h}\times \uH{2}{h}$ and reads:
find $(\uvec{E}_h,\uvec{B}'_h) \in C^1([0,T];X_h)$ such that, for all $(\uvec{v}_h^1,\ul{v}_h^2) \in X_h$ and all $t\in (0,T)$,
\begin{subequations}\label{eq:Maxwell.discrete.semicont}
\begin{align}
\label{eq:Maxwell.discrete.semicont.B}
  \langle \ul{\DIFF}^1_{r,h} \uvec{E}_h(t), \ul{v}_h^2 \rangle 
  ={}& - \langle \partial_t \uvec{B}'_h(t), \ul{v}_h^2 \rangle ,\\
\label{eq:Maxwell.discrete.semicont.E}
  \langle \uvec{B}'_h(t), \ul{\DIFF}^1_{r,h} \uvec{v}_h^1 \rangle 
  ={}& 
  \langle \uI{1}{h} \widetilde{J}(t), \uvec{v}_h^1 \rangle+
  \langle \partial_t \uvec{E}_h(t), \uvec{v}_h^1 \rangle.
\end{align}
\end{subequations}
We note that, in the context of the DDR complex, the discrete \emph{spatial} exterior derivatives $\ul{\DIFF}^k_{r,h}$ are denoted without a tilde.
Owing to the properties of this complex, this scheme preserves a discrete version of the constraint.

\begin{proposition}[Discrete constraint preservation]\label{eq:prop.constraint}
Let $\rho_h\in C^1([0,T];\uH{0}{h})$ be the discrete electric charge density defined by
\begin{equation}\label{eq:def.rhoh}
\langle\rho_h(t),\ul{v}_h^0\rangle= -\langle\uvec{E}_h(0),\ul{\DIFF}^0_{r,h}\ul{v}_h^0\rangle + \int_0^t \langle \uI{1}{h} \widetilde{J}(s),\ul{\DIFF}_{r,h}^0\ul{v}_h^0\rangle\,ds\qquad\forall t\in [0,T]\,,\forall\ul{v}_h^0\in\uH{0}{h}.
\end{equation}
Under Assumptions \ref{assumption:homogeneous} and \ref{assumption:tracetrimmed},
if $(\uvec{E}_h,\uvec{B}'_h)$ is a solution to \eqref{eq:Maxwell.discrete.semicont}, then
\begin{equation}\label{eq:constraint.discrete}
\langle \uvec{E}_h(t),\ul{\DIFF}^0_{r,h}\ul{v}_h^0\rangle = -\langle \rho_h(t),\ul{v}_h^0\rangle
\qquad\forall t\in [0,T]\,,\forall\ul{v}_h^0\in\uH{0}{h}.
\end{equation}
\end{proposition}

\begin{remark}[On the discrete constraint preservation]
By the Riesz representation theorem, \eqref{eq:def.rhoh} uniquely defines the discrete electric charge density $\rho_h(t)\in\uH{0}{h}$.
This definition only depends on the data of the model (initial value of the electric field and electric current density) and, given the relation \eqref{eq:compatibility}, provides a consistent discrete version of the continuous electric charge density.

Letting $\ul{\delta}_{r,h}^1$ be the adjoint of $\ul{\DIFF}_{r,h}^1$, the relation \eqref{eq:constraint.discrete} can be recast as $\ul{\delta}_{r,h}^1\uvec{E}_h(t)=-\rho_h(t)$, which is a coherent discrete version of the continuous constraint \eqref{eq:Maxwell.Stationary.constraint} (recall that $\epsilon_0=1$ here).
\end{remark}

\begin{proof}[Proof of Proposition \ref{eq:prop.constraint}]
Let $\ul{v}_h^0\in\uH{0}{h}$ and set $\uvec{v}_h^1=\ul{\DIFF}_{r,h}^0\ul{v}_h^0$ in \eqref{eq:Maxwell.discrete.semicont.E}.
Since $\ul{\DIFF}_{r,h}^1\ul{\DIFF}_{r,h}^0=0$ by complex property of the DDR sequence, we obtain
\[
\langle \partial_t \uvec{E}_h(t), \ul{\DIFF}_{r,h}^0\ul{v}_h^ 0\rangle
=-\langle \uI{1}{h} \widetilde{J}(t), \ul{\DIFF}_{r,h}^0\ul{v}_h^0 \rangle
=-\langle \partial_t\rho_h(t), \ul{v}_h^0 \rangle,
\]
where the second equality follows differentiating the definition \eqref{eq:def.rhoh} of $\rho_h$ with respect to $t$.
Since $\ul{v}_h^0$ does not depend on $t$, this shows that the derivative of $\mathcal C(t):= \langle \uvec{E}_h(t),\ul{\DIFF}^0_{r,h}\ul{v}_h^0\rangle +\langle \rho_h(t),\ul{v}_h^0\rangle$ vanishes. Since $\mathcal C(0)=0$ by \eqref{eq:def.rhoh}, we infer that $\mathcal C=0$ on $[0,T]$, which proves \eqref{eq:constraint.discrete}.
\end{proof}

\begin{lemma}[Energy preservation]
In the absence of a current density, the solution $(\uvec{E}_h,\uvec{B}'_h)$ of \eqref{eq:Maxwell.discrete.semicont} satisfies
\begin{equation}\label{eq:Maxwell.energy}
  \partial_t \left(\langle \uvec{E}_h,\uvec{E}_h\rangle + \langle  \uvec{B}'_h,\uvec{B}'_h\rangle\right)
   = 0.
\end{equation}
\end{lemma}
\begin{proof}
Evaluating \eqref{eq:Maxwell.discrete.semicont.B} and \eqref{eq:Maxwell.discrete.semicont.E} 
with $\uvec{v}_h^1 = \uvec{E}_h$ and $\underline{v}_h^2 = \uvec{B}'_h$ yields $\langle \partial_t \uvec{E}_h,\uvec{E}_h\rangle + \langle \partial_t \uvec{B}'_h,\uvec{B}'_h\rangle=0$, which implies the result.
\end{proof}
\begin{remark}[Discrete energy preservation]\label{rem:discrete.energy}
When using a time discretisation scheme preserving quadratic invariants, such as the Crank--Nicolson time stepping, 
a discrete version of \eqref{eq:Maxwell.energy} can be established and shows that the total energy of the system $\langle \uvec{E}_h,\uvec{E}_h\rangle + \langle \uvec{B}'_h,\uvec{B}'_h\rangle$ remains constant in time. 
Dissipative time-stepping, such as the implicit Euler method, lead to a decrease of the total energy. 
\end{remark}
 
\begin{remark}[Well-posedness of the semi-discrete and fully discrete schemes]
The semi-discrete scheme \eqref{eq:Maxwell.discrete.semicont} is a linear system of ordinary differential equations in the finite-dimensional space $X_h$. By the Cauchy--Lipschitz theorem, for each pair of initial conditions it has a unique solution over $[0,T]$.

Regarding fully discrete versions of this scheme, when using implicit Euler or Crank--Nicolson, the scheme reduces at each time step to a square linear system of equations. The energy estimate satisfied by these schemes (see Remark \ref{rem:discrete.energy}) provides a bound on any of their solutions; as a consequence, the system's matrix has a trivial kernel and the system is uniquely solvable, which shows that the fully discrete schemes also have a unique solution for each chosen initial condition.
\end{remark}

\subsection{Building the mesh}\label{sec:build.mesh}

Meshing manifolds while ensuring Assumption \ref{assumption:straighten} is not trivial, and to our knowledge, there is no available software producing the required data. 
We therefore developed our own in-house mesh generator for the two manifolds considered here, the flat torus and the sphere.
The meshing of the flat torus is a Cartesian grid on $[0,1)^2$ using four charts to correctly match the boundary.
The meshing of the sphere is more interesting, and described now.

We restrict the two charts of the atlas to the north and south hemispheres; these charts do not overlap on an open set, as their transition occurs on a closed $1$-dimensional interface (the equator), but that is sufficient to build a mesh and the associated polynomial spaces for the entire sphere. Each chart maps an hemisphere into the unit disk of $\Real^2$, and both charts coincide on the equator (that they map onto the unit circle). We therefore only have to discretise this unit disk to get the mesh on the whole sphere.
In order to correctly discretise to boundary circle, we chose to use one layer of curved cell, and then to map the inside of the disk using arbitrary flat cells.

We devised an automated mesh generation algorithm. Starting from a chosen radius $r_s$,
we cut the outer circle into $\lfloor 2\pi(1-r_s)/r_s\rfloor$ segment, 
each of the same length $\Delta \alpha$.
Then, we used the construction \ref{sec:cone.curved} with 
\begin{equation*}
I_1(t) = R_{\alpha - \frac{\pi - \Delta \alpha}{2}} 
\begin{pmatrix} \frac{2 t r}{\sqrt{1 + (2 t r)^2}} \\ \frac{1}{\sqrt{1 + (2 t r)^2}} \end{pmatrix}, 
  \quad t \in \left[-\frac12,\frac12\right],
\end{equation*}
where $r = \frac{\cos(\frac{\pi - \Delta \alpha}{2})}{\sin(\frac{\pi - \Delta \alpha}{2})}$, 
$\alpha$ is the angle between the edge $E_4$ specified in Figure \ref{fig:cone.curved} and the line $x = 0$ in the plane,
and $R_{\alpha - \frac{\pi - \Delta \alpha}{2}}$ 
the rotation of angle $\alpha - \frac{\pi - \Delta \alpha}{2}$.
The connected cell ensuring the transition between the curved boundary and the flat interior is the cone given by the mapping 
\begin{equation*}
I_f(t,p) = R_{\alpha - \frac{\pi - \Delta \alpha}{2}}
\begin{pmatrix}
  2 t \frac{x_B}{y_B} \frac{p + (1 - p) r}{\sqrt{1 + (2 t \frac{x_B}{y_B})^2}} \\
  \frac{p + (1 - p) r}{\sqrt{1 + (2 t \frac{x_B}{y_B})^2}}
\end{pmatrix}
, \quad t \in \left[-\frac12,\frac12\right], p \in [0,1],
\end{equation*}
where 
\begin{align*}
x_T ={}& \cos(\frac{\pi - \Delta \alpha}{2})\,,\quad
x_B = 0.8 \cos(\frac{\pi - \Delta \alpha}{2})\,,\\
y_T ={}& \sin(\frac{\pi - \Delta \alpha}{2})\,,\quad
y_B = \sin(\frac{\pi - \Delta \alpha}{2})\quad\mbox{and}\quad
r = \sqrt{y_B^2 + (2 t x_B)^2}.
\end{align*}
Then, we divide the interior with concentric regular convex polygons of radius $1-i r_s$, for $1 \leq i \leq \lfloor 1/r_s\rfloor$,
and connect the vertices of these polygons to their nearest neighbour strictly farther away from the center.
The resulting mapping is given in Figure \ref{fig:mesh.disk}.
The coarsest mesh considered in the numerical experiment uses $28$ boundary cells, 
$16$ triangles, $4$ quads, and $12$ pentagons, 
while the finest uses $346$ boundary cells, $34$ triangles, $4472$ quads, and $312$ pentagons.

\begin{figure}
\begin{center}
  \begin{minipage}[c]{0.45\columnwidth}\centering
    \includegraphics[width=0.9\columnwidth]{mesh_orientation}
    \subcaption{Mesh used to discretise one chart with its orientation.}
  \end{minipage}
  \begin{minipage}[c]{0.45\columnwidth}\centering
    \includegraphics[width=0.9\columnwidth]{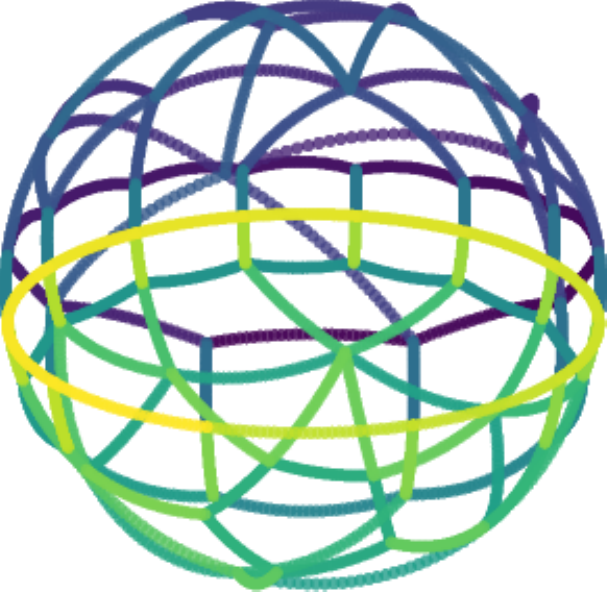}
    \subcaption{Embedding of the $1$-skeleton of the mesh into $\Real^3$.}
  \end{minipage}
  \begin{minipage}[c]{0.45\columnwidth}\centering
    \includegraphics[width=0.9\columnwidth]{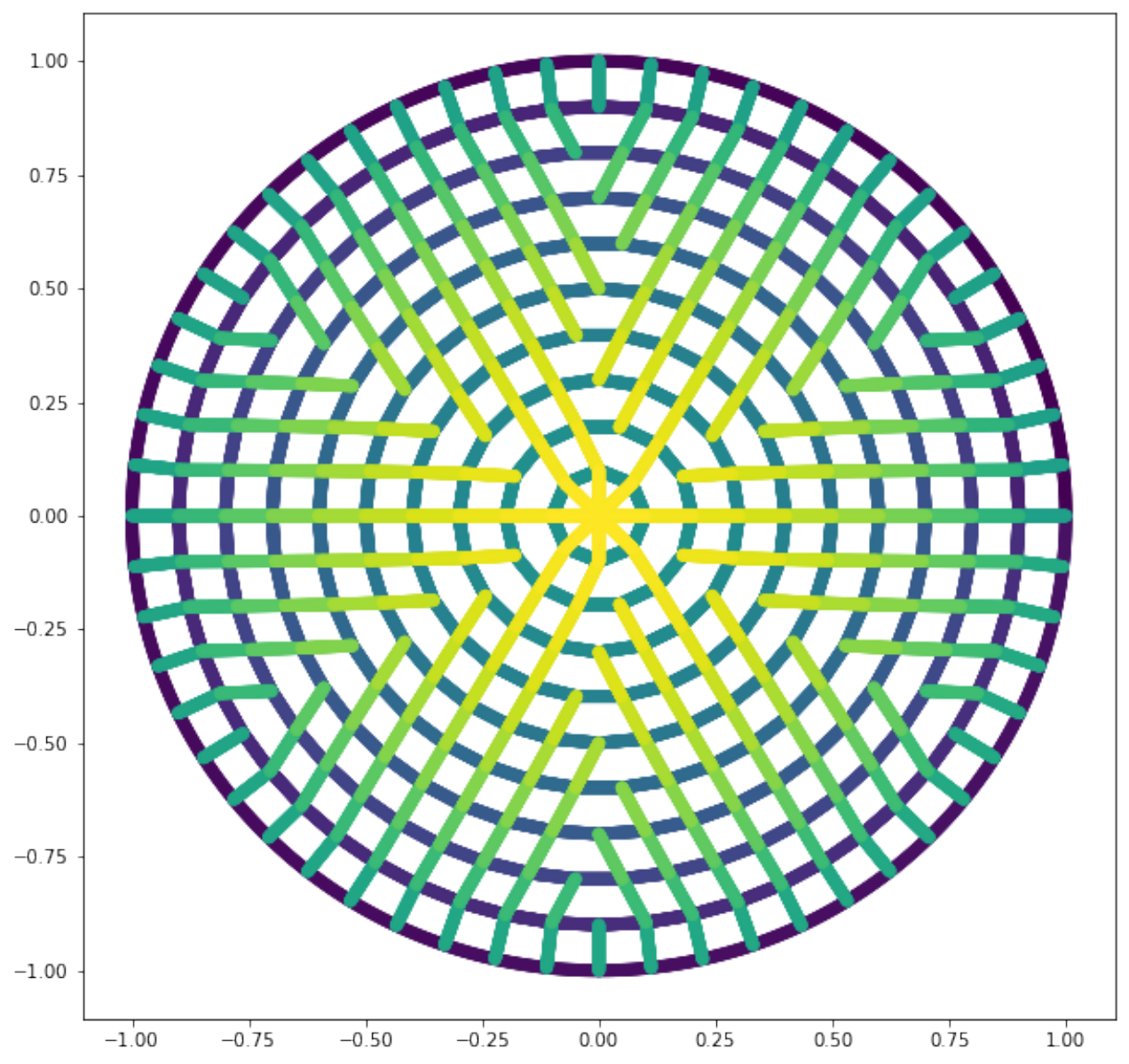}
    \subcaption{Finer meshing of the disk.}
  \end{minipage}
  \begin{minipage}[c]{0.45\columnwidth}\centering
    \includegraphics[width=0.9\columnwidth]{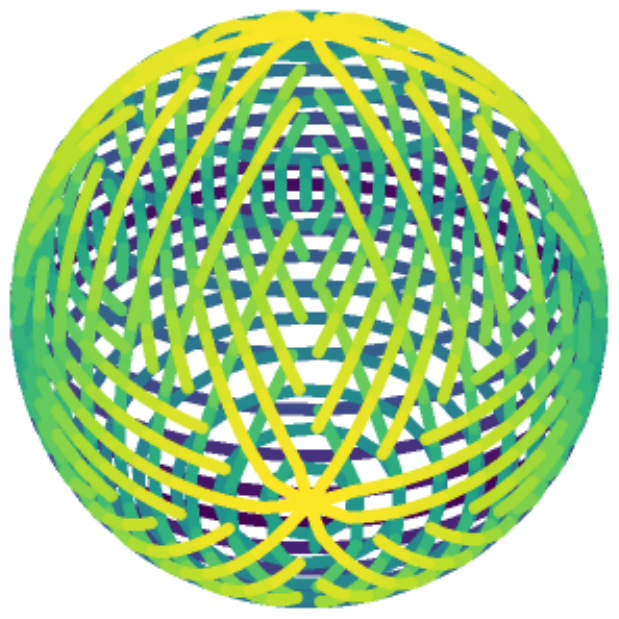}
    \subcaption{Embedding of the finer mesh into $\Real^3$.}
  \end{minipage}
\end{center}
\caption{Visualization of the mesh}
\label{fig:mesh.disk}
\end{figure}

\subsection{Results}

We present here the results of the numerical tests. 
The scheme has been implemented in the Manicore C++ framework 
(see \url{https://mlhanot.github.io/Manicore/}, version 1.0.1), 
using linear algebra facilities from the Eigen3 library (see \url{https://eigen.tuxfamily.org}). 
Although the framework has been reimplemented from scratch, 
its design borrows from the HArDCore C++ framework implementing the DDR complexes on flat spaces 
(see \url{https://github.com/jdroniou/HArDCore}).
From classical results on flat spaces, we expect to see a convergence in space of order $r+1$ for the $L^2$ norm of $E$ and $B$, when the exact solutions are smooth and only spatial error is accounted for.

We used a Crank-Nicolson time stepping.
In order to neglect the error originated from the time discretisation, we first used a time step of $\Delta t = 10^{-5}$. 
These results were compared with those obtained using the larger time step $\Delta t = 10^{-3}$, and showed only a negligible difference (a relative difference of order $10^{-6}$). Hence, at the considered time steps the spatial discretisation fully dominates, and to save computational resources we conducted all the tests below using the larger time step $\Delta t=10^{-3}$.

The simulation are computed on the time span $0 \leq t \leq 2\pi$, and the error is the $L^2$ error over time and space. For $E$, for example, this error is therefore
\[
\sqrt{\int_{t = 0}^{2\pi}  \langle \uI{k}{h} E - \ul{E}_h,\uI{k}{h} E - \ul{E}_h\rangle}.
\]

We consider two test cases: one with $C^0$ (but not smooth) reference solutions on the sphere and the torus, the other with a smooth reference solution on the sphere.

\subsection{$C^0$ test}

On the torus, the reference solution is given in coordinates $(X,Y) \in [0,1]^2$ by
\begin{align*}
B' ={}& (2 + d_2(X,Y,t))\DIFF X\wedge \DIFF Y,\\
E ={}& d_2(X,0.5,t)\DIFF Y,\\
\rho ={}& 0, \\
J ={}& (2Y - 1)\DIFF X,
\end{align*}
where $d_2(X,Y,t) = \Vert (X,Y) - (t,0.5)\Vert^2$ is the Euclidean distance between the point $(X,Y)$ and the point $(t,0.5)$ on the torus.
The solution for the sphere is given in Appendix \ref{sec:exact.solutions.C0} (see \eqref{eq:sol.sp.r1}--\eqref{eq:sol.sp.r2}).

We observed similar results for the two $C^0$ solutions, see Figure \ref{fig:convrate}. We do not expect and have not noticed a convergence of the error on $\DIFF E$, since $E$ is not smooth in this test case.
Moreover, due to this lack of smoothness of the exact solution, we do no necessarily expect an optimal convergence rate of order $r+1$ in $L^2$-norm. 
On the sphere, we however note a rate of convergence of order $1$ when $r=0$ and order $2$ when $r\ge 1$. For the torus, the rates seem to stagnate at $1$. However, in both cases,
the error is reduced by almost an order of magnitude when going from the lowest-order case to the case $r\ge 1$; 
such a phenomenon (improvement of the ratio cost/accuracy when using a slightly higher-order method, even when the solution is not smooth) has already been qualitatively observed in several other numerical schemes for different models \cite{Anderson.Droniou:18,Lemaire.Moatti:23,Beaude.ea:23}.

\begin{figure}
\begin{center}
  \ref{leg.conv.sphere}\medskip
  \\
  \begin{minipage}[c]{0.45\columnwidth}\centering
    \begin{tikzpicture}
      \begin{loglogaxis}[legend columns=5,legend to name=leg.conv.sphere]
        \addplot +[mark=+, style=solid, color=blue] table[x index=0, y index=1] {sphere_0};
        \addplot +[mark=triangle, style=solid, color=red] table[x index=0, y index=1] {sphere_1};
        \addplot +[mark=square, style=solid, color=brown] table[x index=0, y index=1] {sphere_2};
        \addplot +[mark=pentagon, style=solid, color=darkgray] table[x index=0, y index=1] {sphere_3};
        \addplot +[mark=o, style=solid, color=violet] table[x index=0, y index=1] {sphere_4};
        \logLogSlopeTriangle{0.90}{0.4}{0.1}{1}{blue};
        \logLogSlopeTriangle{0.90}{0.4}{0.1}{2}{red};
        \legend{$r=0$,$r=1$,$r=2$,$r=3$,$r=4$};
      \end{loglogaxis}
    \end{tikzpicture}
    \subcaption{Error on $E$ on the sphere.}
  \end{minipage}
  \hfill
  \begin{minipage}[c]{0.5\columnwidth}\centering
    \begin{tikzpicture}
      \begin{loglogaxis}
        \addplot +[mark=+, style=solid, color=blue] table[x index=0, y index=3] {sphere_0};
        \addplot +[mark=triangle, style=solid, color=red] table[x index=0, y index=3] {sphere_1};
        \addplot +[mark=square, style=solid, color=brown] table[x index=0, y index=3] {sphere_2};
        \addplot +[mark=pentagon, style=solid, color=darkgray] table[x index=0, y index=3] {sphere_3};
        \addplot +[mark=o, style=solid, color=violet] table[x index=0, y index=3] {sphere_4};
        \logLogSlopeTriangle{0.90}{0.4}{0.1}{1}{blue};
        \logLogSlopeTriangle{0.90}{0.4}{0.1}{2}{red};
      \end{loglogaxis}
    \end{tikzpicture}
    \subcaption{Error on $B$ on the sphere.}
  \end{minipage}
  \begin{minipage}[c]{0.45\columnwidth}\centering
    \begin{tikzpicture}
      \begin{loglogaxis}
        \addplot +[mark=+, style=solid, color=blue] table[x index=0, y index=1] {torus_0};
        \addplot +[mark=triangle, style=solid, color=red] table[x index=0, y index=1] {torus_1};
        \addplot +[mark=square, style=solid, color=brown] table[x index=0, y index=1] {torus_2};
        \addplot +[mark=pentagon, style=solid, color=darkgray] table[x index=0, y index=1] {torus_3};
        \addplot +[mark=o, style=solid, color=violet] table[x index=0, y index=1] {torus_4};
        \logLogSlopeTriangle{0.90}{0.25}{0.1}{1}{blue};
        \logLogSlopeTriangle{0.90}{0.25}{0.1}{2}{red};
      \end{loglogaxis}
    \end{tikzpicture}
    \subcaption{Error on $E$ on the torus.}
  \end{minipage}
  \hfill
  \begin{minipage}[c]{0.5\columnwidth}\centering
    \begin{tikzpicture}
      \begin{loglogaxis}
        \addplot +[mark=+, style=solid, color=blue] table[x index=0, y index=3] {torus_0};
        \addplot +[mark=triangle, style=solid, color=red] table[x index=0, y index=3] {torus_1};
        \addplot +[mark=square, style=solid, color=brown] table[x index=0, y index=3] {torus_2};
        \addplot +[mark=pentagon, style=solid, color=darkgray] table[x index=0, y index=3] {torus_3};
        \addplot +[mark=o, style=solid, color=violet] table[x index=0, y index=3] {torus_4};
        \logLogSlopeTriangle{0.90}{0.25}{0.1}{1}{blue};
        \logLogSlopeTriangle{0.90}{0.25}{0.1}{2}{red};
      \end{loglogaxis}
    \end{tikzpicture}
    \subcaption{Error on $B$ on the torus.}
  \end{minipage}
\end{center}
\caption{Errors vs.~mesh size for the $C^0$ solutions.}
\label{fig:convrate}
\end{figure}
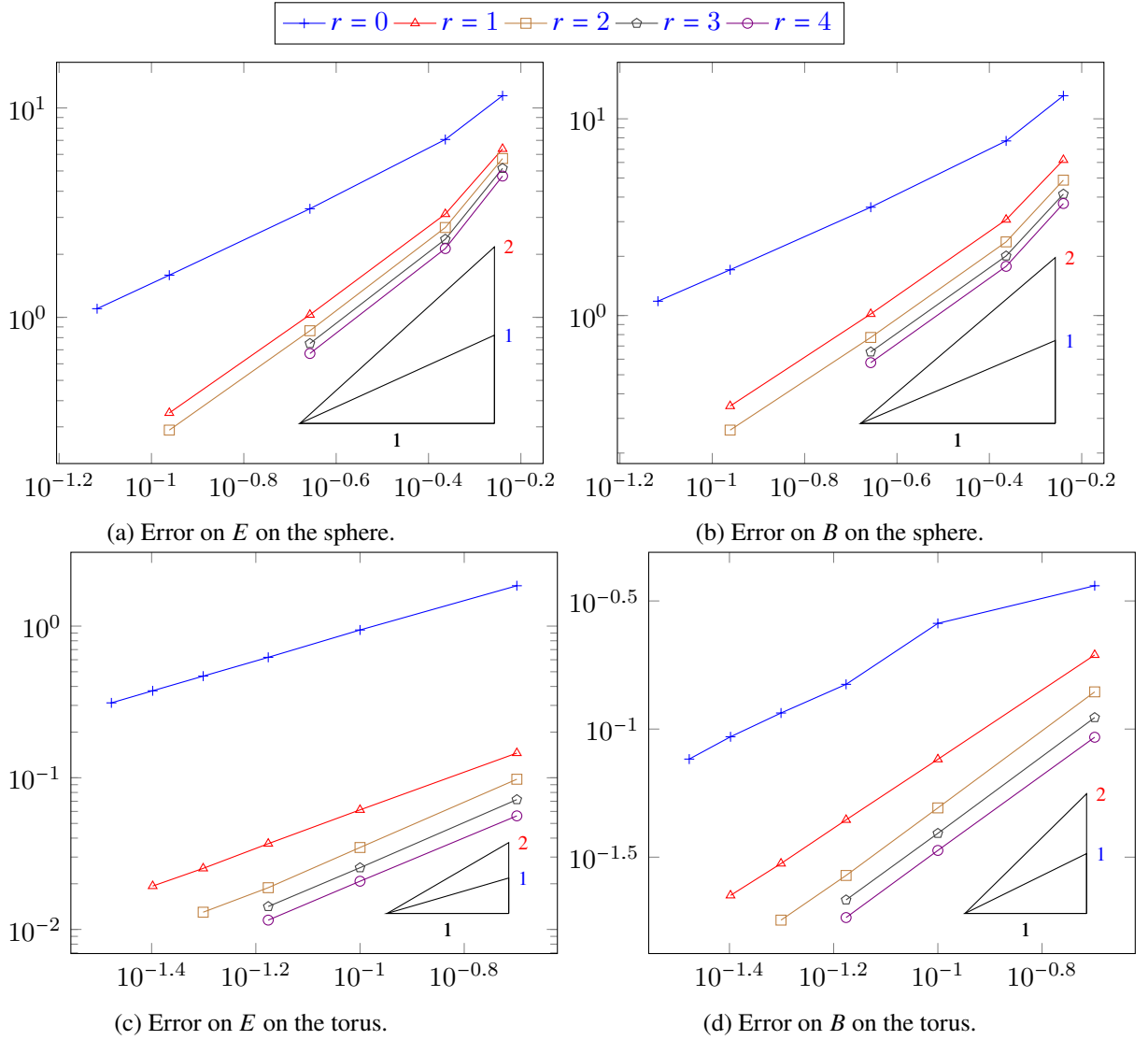

\subsection{Smooth solution}\label{sec:results.smooth}

The smooth reference solution on the sphere is given by \eqref{eq:sol.sp.s} in Appendix \ref{sec:exact.solutions.smooth}.
The errors on $E$, $B$ and $\DIFF E$ for the smooth case on the sphere are given in Figure \ref{fig:convrate.smooth},
and the convergence rates computed as the best fit in the least square sense 
of the value in Figure \ref{fig:convrate.smooth} are given in Table \ref{tab:convrate.smooth}. 
We notice an improvement of the convergence rate when increasing the polynomial degree $r$ that exceeds our expectation for the magnetic field $B$.
The observed orders of convergence for $E$ and $\DIFF E$ are close to $r+ 1.5$ for $r=0,1,3$;
the observed order of convergence for $B$ is overall even higher (around $r+2.5$ for $r=1,2,3$). 
The slight decrease of the order of convergence for $r=4$ may be attributed to the limited accuracy of the quadrature rule available in the code to interpolate the exact function (for this $r$, a very high degree of accuracy of that interpolation is required).
Moreover, the sharp decrease on $B$ for $r=4$ on the finest mesh can be explained by 
the fact that the error from the spatial discretisation becomes 
comparable to the error from the temporal discretisation.
Indeed, this happen when the absolute error approaches $10^{-9}$, which is around the observed contribution 
from the temporal discretisation.

The discrete constraint preservation \eqref{eq:constraint.discrete},
and the energy preservation \eqref{eq:Maxwell.energy} were observed to hold up to machine precision. 
The computed values for a few illustrative test cases for the smooth solution are given in Figure \ref{fig:conservation}.
While we can see the floating point error build up for the discrete constraint, 
the energy remains remarkably stable, with a difference between the minimum and maximum values over time reached 
on any single simulation of less than $10^{-11}$.

One of the goals of high-order methods is to provide highly accurate solutions on a coarser mesh at a lower cost 
than a low-order method on a fine mesh. 
We measured this ratio of cost/accuracy for the smooth case.
The simulations being run on a personal laptop, 
we do not expect the computational time to be a good metric 
as it is affected by the other activities running in parallel.
Instead, we provide a comparison between the accuracy and the number of degrees of freedom 
of each combination of polynomial degree and mesh in Figure \ref{fig:convrate.smoothdof}.
As the pattern of non-zero elements is similar for all the cases with $r>0$, 
we expect the size of the system to be a decent proxy for the computational complexity.
We observe that using higher degree polynomials on coarse meshes is indeed far more efficient
than using low order polynomials on fine meshes.

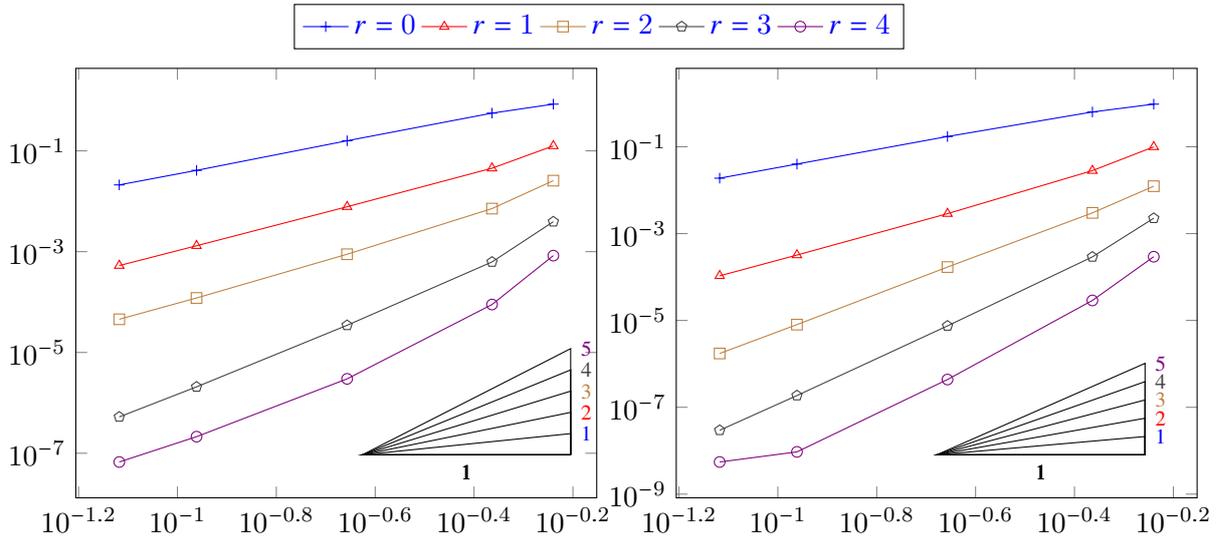
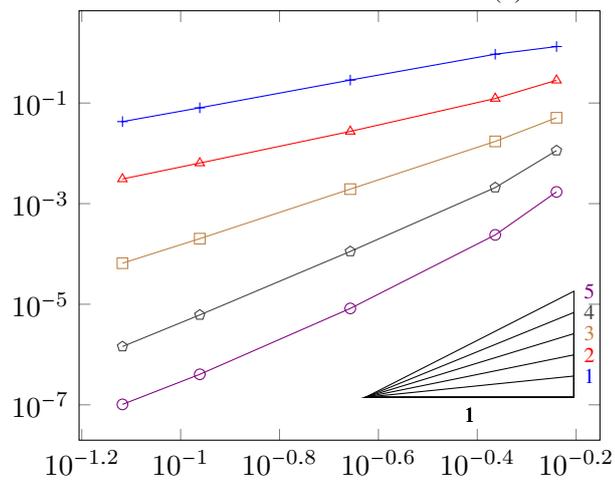
\begin{figure}
\begin{center}
  \ref{leg.conv.sphere}\medskip
  \\
  \begin{minipage}[c]{0.45\columnwidth}\centering
    \begin{tikzpicture}
      \begin{loglogaxis}
        \addplot +[mark=+, style=solid, color=blue] table[x index=0, y index=1] {ssphere_0};
        \addplot +[mark=triangle, style=solid, color=red] table[x index=0, y index=1] {ssphere_1};
        \addplot +[mark=square, style=solid, color=brown] table[x index=0, y index=1] {ssphere_2};
        \addplot +[mark=pentagon, style=solid, color=darkgray] table[x index=0, y index=1] {ssphere_3};
        \addplot +[mark=o, style=solid, color=violet] table[x index=0, y index=1] {ssphere_4};
        \logLogSlopeTriangle{0.95}{0.4}{0.1}{1}{blue};
        \logLogSlopeTriangle{0.95}{0.4}{0.1}{2}{red};
        \logLogSlopeTriangle{0.95}{0.4}{0.1}{3}{brown};
        \logLogSlopeTriangle{0.95}{0.4}{0.1}{4}{darkgray};
        \logLogSlopeTriangle{0.95}{0.4}{0.1}{5}{violet};
      \end{loglogaxis}
    \end{tikzpicture}
    \subcaption{Error on $E$.}
  \end{minipage}
  \hfill
  \begin{minipage}[c]{0.5\columnwidth}\centering
    \begin{tikzpicture}
      \begin{loglogaxis}
        \addplot +[mark=+, style=solid, color=blue] table[x index=0, y index=3] {ssphere_0};
        \addplot +[mark=triangle, style=solid, color=red] table[x index=0, y index=3] {ssphere_1};
        \addplot +[mark=square, style=solid, color=brown] table[x index=0, y index=3] {ssphere_2};
        \addplot +[mark=pentagon, style=solid, color=darkgray] table[x index=0, y index=3] {ssphere_3};
        \addplot +[mark=o, style=solid, color=violet] table[x index=0, y index=3] {ssphere_4};
        \logLogSlopeTriangle{0.90}{0.4}{0.1}{1}{blue};
        \logLogSlopeTriangle{0.90}{0.4}{0.1}{2}{red};
        \logLogSlopeTriangle{0.90}{0.4}{0.1}{3}{brown};
        \logLogSlopeTriangle{0.90}{0.4}{0.1}{4}{darkgray};
        \logLogSlopeTriangle{0.90}{0.4}{0.1}{5}{violet};
      \end{loglogaxis}
    \end{tikzpicture}
    \subcaption{Error on $B$.}
  \end{minipage}\\
  \begin{minipage}[c]{0.45\columnwidth}\centering
    \begin{tikzpicture}
      \begin{loglogaxis}
        \addplot +[mark=+, style=solid, color=blue] table[x index=0, y index=2] {ssphere_0};
        \addplot +[mark=triangle, style=solid, color=red] table[x index=0, y index=2] {ssphere_1};
        \addplot +[mark=square, style=solid, color=brown] table[x index=0, y index=2] {ssphere_2};
        \addplot +[mark=pentagon, style=solid, color=darkgray] table[x index=0, y index=2] {ssphere_3};
        \addplot +[mark=o, style=solid, color=violet] table[x index=0, y index=2] {ssphere_4};
        \logLogSlopeTriangle{0.95}{0.4}{0.1}{1}{blue};
        \logLogSlopeTriangle{0.95}{0.4}{0.1}{2}{red};
        \logLogSlopeTriangle{0.95}{0.4}{0.1}{3}{brown};
        \logLogSlopeTriangle{0.95}{0.4}{0.1}{4}{darkgray};
        \logLogSlopeTriangle{0.95}{0.4}{0.1}{5}{violet};
      \end{loglogaxis}
    \end{tikzpicture}
    \subcaption{Error on $\DIFF E$.}
  \end{minipage}
\end{center}
\caption{Errors vs.~mesh size for the smooth solution.}
\label{fig:convrate.smooth}
\end{figure}

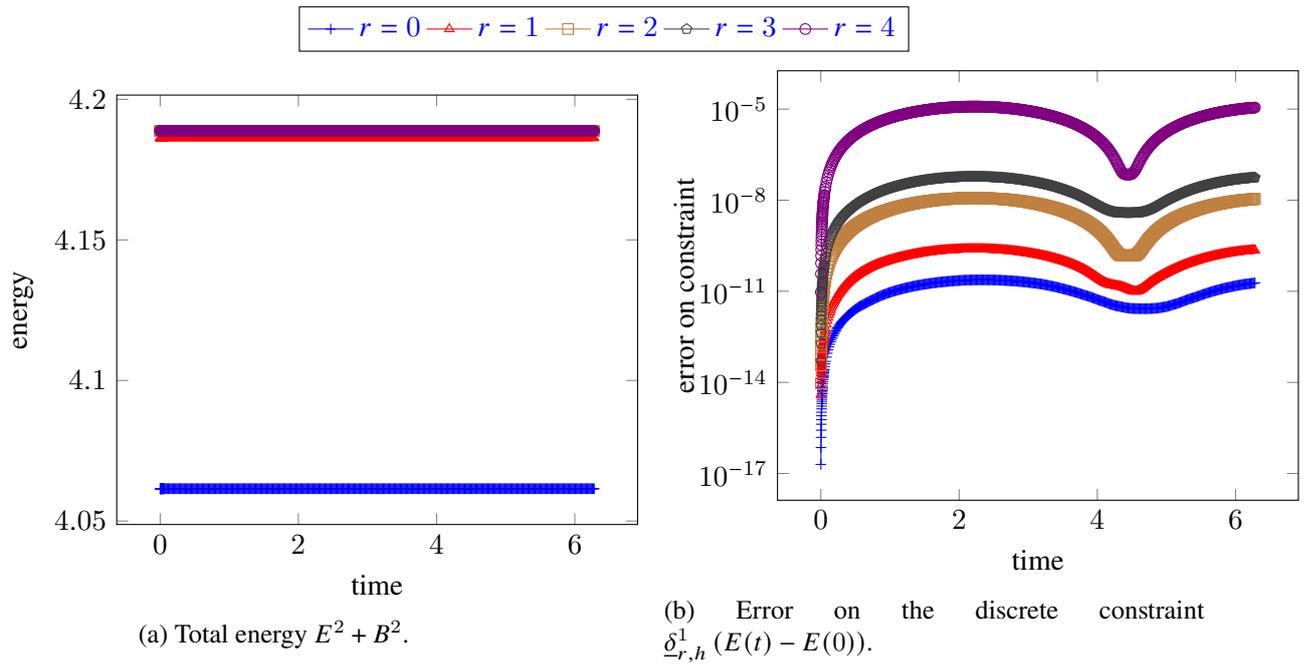
\begin{figure}
\begin{center}
  \ref{leg.conv.sphere}\medskip
  \\
  \begin{minipage}[c]{0.45\columnwidth}\centering
    \begin{tikzpicture}
      \begin{axis}
        \addplot +[mark=+, style=solid, color=blue] table[x index=0, y index=5] {maxwell_d0};
        \addplot +[mark=triangle, style=solid, color=red] table[x index=0, y index=5] {maxwell_d1};
        \addplot +[mark=square, style=solid, color=brown] table[x index=0, y index=5] {maxwell_d2};
        \addplot +[mark=pentagon, style=solid, color=darkgray] table[x index=0, y index=5] {maxwell_d3};
        \addplot +[mark=o, style=solid, color=violet] table[x index=0, y index=5] {maxwell_d4};
      \end{axis}
    \end{tikzpicture}
    \subcaption{Total energy $E^2 + B^2$.}
  \end{minipage}
  \hfill
  \begin{minipage}[c]{0.45\columnwidth}\centering
    \begin{tikzpicture}
      \begin{semilogyaxis}
        \addplot +[mark=+, style=solid, color=blue] table[x index=0, y index=6] {maxwell_d0};
        \addplot +[mark=triangle, style=solid, color=red] table[x index=0, y index=6] {maxwell_d1};
        \addplot +[mark=square, style=solid, color=brown] table[x index=0, y index=6] {maxwell_d2};
        \addplot +[mark=pentagon, style=solid, color=darkgray] table[x index=0, y index=6] {maxwell_d3};
        \addplot +[mark=o, style=solid, color=violet] table[x index=0, y index=6] {maxwell_d4};
      \end{semilogyaxis}
    \end{tikzpicture}
    \subcaption{Error on the discrete constraint $\ul{\delta}^1_{r,h}\left( E(t) - E(0) \right)$.}
  \end{minipage}
\end{center}
\caption{Preservation of the constraint and of the energy.}
\label{fig:conservation}
\end{figure}

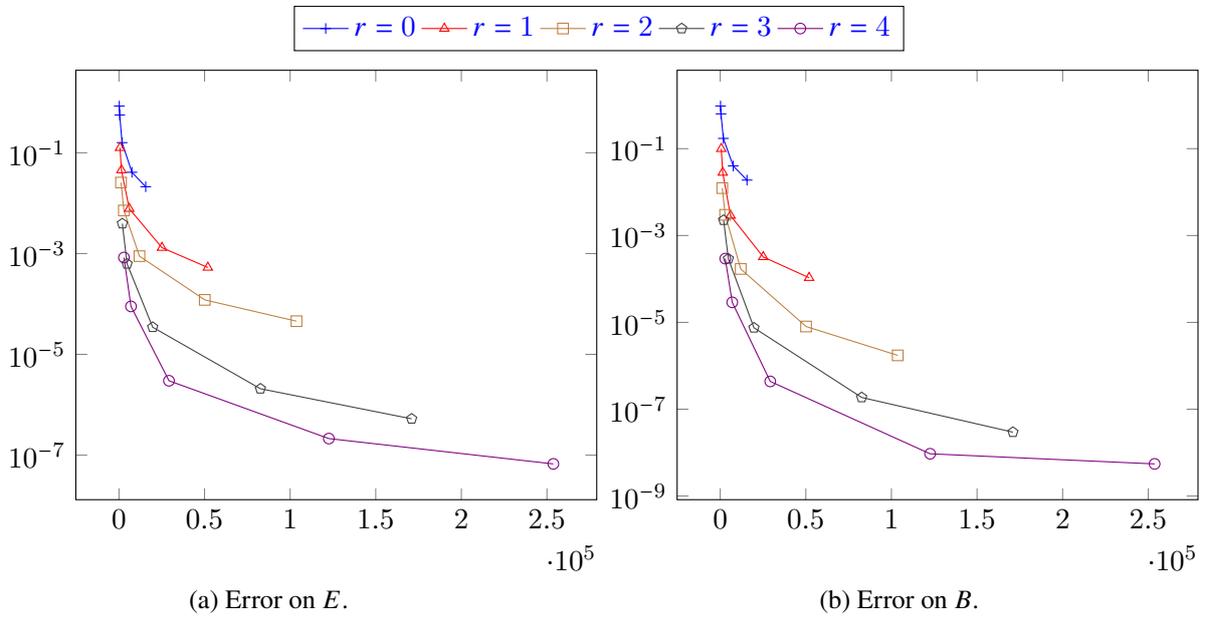
\begin{figure}
\begin{center}
  \ref{leg.conv.sphere}\medskip
  \\
  \begin{minipage}[c]{0.45\columnwidth}\centering
    \begin{tikzpicture}
      \begin{loglogaxis}
        \addplot +[mark=+, style=solid, color=blue] table[x index=4, y index=1] {ssphere_0};
        \addplot +[mark=triangle, style=solid, color=red] table[x index=4, y index=1] {ssphere_1};
        \addplot +[mark=square, style=solid, color=brown] table[x index=4, y index=1] {ssphere_2};
        \addplot +[mark=pentagon, style=solid, color=darkgray] table[x index=4, y index=1] {ssphere_3};
        \addplot +[mark=o, style=solid, color=violet] table[x index=4, y index=1] {ssphere_4};
      \end{loglogaxis}
    \end{tikzpicture}
    \subcaption{Error on $E$.}
  \end{minipage}
  \hfill
  \begin{minipage}[c]{0.5\columnwidth}\centering
    \begin{tikzpicture}
      \begin{loglogaxis}
        \addplot +[mark=+, style=solid, color=blue] table[x index=4, y index=3] {ssphere_0};
        \addplot +[mark=triangle, style=solid, color=red] table[x index=4, y index=3] {ssphere_1};
        \addplot +[mark=square, style=solid, color=brown] table[x index=4, y index=3] {ssphere_2};
        \addplot +[mark=pentagon, style=solid, color=darkgray] table[x index=4, y index=3] {ssphere_3};
        \addplot +[mark=o, style=solid, color=violet] table[x index=4, y index=3] {ssphere_4};
      \end{loglogaxis}
    \end{tikzpicture}
    \subcaption{Error on $B$.}
  \end{minipage}
\end{center}
\caption{Errors vs.~number of degrees of freedom for the smooth solution.}
\label{fig:convrate.smoothdof}
\end{figure}
\begin{table}
\centering
\begin{tabular}{|l|*{5}{c|}}\hline
Degree $r$
&$0$&$1$&$2$
&$3$&$4$\\\hline\hline
Error on $E$ &$1.85$&$2.66$&$3.08$&$4.32$&$4.56$\\\hline
Error on $\DIFF E$&$1.73$&$2.20$&$3.27$&$4.37$&$4.75$\\\hline
Error on $B$ &$1.96$&$3.34$&$4.36$&$5.49$&$5.51$\\\hline
\end{tabular}
\caption{Convergence rates for the smooth solution.}
\label{tab:convrate.smooth}
\end{table}

\section{Conclusion}\label{sec:conclusion}

We developed a discrete de Rham complex on manifolds, that is applicable on curved meshes. We designed a framework of full and trimmed polynomial spaces on such meshes, identifying the required assumptions that enabled us to adapt the setting of \cite{Bonaldi.Di-Pietro.ea:23} to the case of manifolds. The discrete complex, based on the exterior calculus presentation of the de Rham complex, has the same cohomology as the continuous de Rham complex, and has an arbitrary order of accuracy. 

We showed that if the mesh cells can be embedded into flat spaces in a way that ensures that the composition of internal and boundary mappings are linear, then the assumptions of our framework are satisfied. An explicit construction of these embeddings was provided for a range of curved 3- and 4-edge cells in 2D, which typically arise when mapping triangular/quadrangular cells through the charts of the manifold. The ideas behind this construction could in principle be used to cover other kind of cells (or 3D cells), but would require technical efforts and is the topic of interesting further research on mesh generation. 

As an illustration of the usage of this discrete complex, we considered the case of the Maxwell equations written in differential forms on a sphere and on a torus. The complex property enabled us to design a scheme that preserves both the energy and the constraint associated with this model.
We constructed smooth and non-smooth solutions, and numerically assessed the efficiency of the scheme in terms of its preservation properties and accuracy.

\appendix

\section{Construction of suitable charts in dimension $n=2$} \label{sec:example.construction}

In this appendix, we provide a few examples of constructions that satisfy Assumption \ref{assumption:straighten} for various topologies.
They can be used where a curved boundary is required and combined with flat cells elsewhere.
Actually, when a cell is flat in a chart, we can simply restrict this chart to get a suitable $I_f$ verifying this assumption.
Therefore, our method is really a generalization of the polynomials on $\Real^d$, 
and we only need to create more exotic spaces on curved cells.
While there is no fully generic way to construct $I_f$, we give below some possible constructions in several cases.

The idea is to apply a hierarchical construction of the mappings $I_f$ in Assumption \ref{assumption:straighten}. In the 2D case we consider here, this means that we start from the diffeomorphisms $I_f:[0,1]\to \overline{f}$ corresponding the 1-cells (edges) $f$, which are often easy to describe. Then, for each 2-cell $f$, we use the mappings $(I_{f'})_{f'\in\Delta_1(f)}$ of its edges to construct a diffeomorphism $I_f:[0,1]^2\to \overline{f}$ (or sometimes its inverse $J_f=(I_f)^{-1}$, which is equivalent) that satisfies the compatibility condition of Assumption \ref{assumption:straighten}, that is, the linearity of the composition of mappings. The initial approach to find the suitable $I_f$ is to pick two edges of $f$ connected by a vertex and try and foliate this 2-cell using the coordinates on these edges; of course, maintaining the compatibility with the other edges is where the challenge arises, and requires to adjust this first idea to the particular shape of the element by modifying the foliation as we move closer to the other edges. On principle, this approach can also be extended to other shapes and to higher dimensions, but will still require some level of case-by-case adjustment and, probably, extensive calculations in complicated geometries.

In practice, we work in coordinates, and therefore start from a given chart $J^0_f\st \overline{f} \to \Real^2$.
When we mention the injection of an edge $E \in \FM{1}(f)$, 
we are actually speaking about the composition 
$J^0_f \circ \INJ_{f,E} \circ I_E$.
Likewise, we will not directly build maps from $f$ to $\Real^2$, 
but rather some diffeomorphism $J_f$ between two subsets of $\Real^2$.
The actual map that verifies Assumption \ref{assumption:straighten} is $I_f:=(J_f \circ J^0_f)^{-1}$.

\subsection{Triangle with two curved edges} 

Let $f \in \FM{2}(\Mh)$ be a face with $3$ edges.
In most cases, especially since $f$ is supposed to be small, we can find a chart in which one of the edges of $f$ is straight.
In this case, the injections of the edges are given by 
\begin{align*}
  I_1(t) ={}& (h_1(t),g_1(t))\\
  I_2(t) ={}& (h_2(t),g_2(t))\\
  I_3(t) ={}& (0,t),
\end{align*}
where $t \in [0,1]$ and $h_i,g_i$ are $C^2$. See Figure \ref{fig:2curves} for an illustration (in which each $I_r$ corresponds to $E_r$).

\begin{figure}[h]
\centering
  \includegraphics[width=0.5\linewidth]{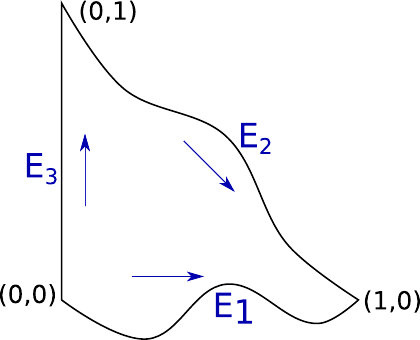}
\caption{$f$ with three edges, one of which being straight in a chosen chart.}
\label{fig:2curves}
\end{figure}

Moreover assume that, for $i \in \lbrace 1,2 \rbrace$, $h_i(0) = 0$, $h_i(1) = 1$, and $h_i$ is strictly increasing, 
and that $g_1(0) = 0$, $g_1(1) = 0$, $g_2(0) = 1$, and $g_2(1) = 0$.
Let $G_i(x) := g_i(h_i^{-1}(x))$, $i \in\lbrace1,2\rbrace$.
\begin{lemma}
  If the function 
  \[
    J_f(x,y) := \frac{1}{G_2(x) - G_1(x)} \begin{pmatrix} 
      (y - G_1(x))h_2^{-1}(x) + (G_2(x) - y)h_1^{-1}(x) \\
      (y - G_1(x))(1 - h_2^{-1}(x))
    \end{pmatrix}
  \]
  is a $C^2$ diffeomorphism, then it
  satisfies Assumption \ref{assumption:straighten}.
\end{lemma}
\begin{proof}
  We need to check that $J_f\circ I_r$ is linear for all $r=1,2,3$.
  By definition, $G_i(h_i(t)) = g_i(t)$ so, for $E_1$, we have
  \begin{align*}
    J_f(I_1(t)) ={}& J_f(h_1(t),g_1(t)) \\
    ={}& \frac{1}{G_2(h_1(t)) - g_1(t)} \begin{pmatrix}
      (g_1(t) - g_1(t))h_2^{-1}(h_1(t)) + (G_2(h_1(t)) - g_1(t))t \\
      (g_1(t) - g_1(t))(1 - h_2^{-1}(h_1(t))
    \end{pmatrix}\\
    ={}& \frac{1}{G_2(h_1(t)) - g_1(t)} \begin{pmatrix}
       (G_2(h_1(t)) - g_1(t))t \\
      0
    \end{pmatrix}\\
    ={}& \begin{pmatrix} t \\ 0 \end{pmatrix}.
  \end{align*}
  For $E_2$, we have
  \begin{align*}
    J_f(I_2(t)) ={}& J_f(h_2(t),g_2(t)) \\
    ={}& \frac{1}{g_2(t) - G_1(h_2(t))} \begin{pmatrix}
      (g_2(t) - G_1(h_2(t)))t + (g_2(t) - g_2(t))h_1^{-1}(h_2(t)) \\
      (g_2(t) - G_1(h_2(t)))(1 - t)
    \end{pmatrix}\\
    ={}& \begin{pmatrix} t \\ 1 - t \end{pmatrix}.
  \end{align*}
  Finally, for $E_3$,
  noticing that $h_i^{-1}(0) = 0$, we have $G_1(0) = 0$ and $G_2(0) = 1$
  and thus
  \begin{align*}
    J_f(I_3(t)) ={}& J_f(0,t) \\
    ={}& \frac{1}{G_2(0) - G_1(0)} \begin{pmatrix}
      (t - G_1(0))h_2^{-1}(0) + (G_2(0) - t)h_1^{-1}(0)\\
      (t - G_1(0))(1 - h_2^{-1}(0))
    \end{pmatrix}\\
    ={}& \begin{pmatrix} 0\\t\end{pmatrix}.
  \end{align*}
\end{proof}

\subsection{Cone section with a curved edge} \label{sec:cone.curved}

Let $f \in \FM{2}(\Mh)$ be a face with $4$ edges such that, when expressed in polar coordinates $(r,\theta)$ on the chart of $f$, 
the injections of the edges are given by
\begin{align*}
  I_1(t) ={}& (g_1(t),h(t)),\\
  I_2(t) ={}& (t + 1, \theta_1),\\
  I_3(t) ={}& (g_3(t), h(t)), \\
  I_4(t) ={}& (t (a - 1) + 1, \theta_2),
\end{align*}
where $t \in [0,1]$, and the functions $h,g_1,g_3$ are $C^2$ and satisfy
\begin{align*}
  h(0) = \theta_1,&\quad h(1) = \theta_2,\\
  g_1(0) = 2,&\quad g_1(1) = a,\\
  g_3(0) = 1,&\quad g_3(1) = 1.
\end{align*}
See Figure \ref{fig:cone.curved} for an illustration. As explained in Section \ref{sec:build.mesh}, such cells naturally appears as transitions between flat cells (fully contained in one of the charts of the selected atlas) and curved boundaries, see Figure \ref{fig:mesh.disk}.

\begin{figure}[h]
\centering
  \includegraphics[width=0.5\linewidth]{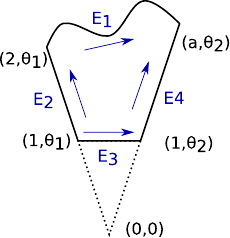}
\caption{$f$ with four edges, corresponding to a section of cone with a curved edge.}
\label{fig:cone.curved}
\end{figure}

Moreover, assume that $h$ is strictly increasing with a derivative bounded from below,
and that $g_1 - g_3$ is strictly positive (meaning that $E_1$ and $E_3$ do not intersect).
\begin{lemma}\label{lem:4-edgecell}
  The following function $I_f\st[0,1]^2\to\overline{f}$ satisfies Assumption \ref{assumption:straighten}:
  \begin{equation*}
    I_f(t,p) := \begin{pmatrix} p g_1(t) + (1 - p) g_3(t) \\ h(t) \end{pmatrix}.
  \end{equation*}
\end{lemma}
\begin{proof}
  Since $h$ is a $C^2$-diffeomorphism and $g_1-g_3>0$, it can easily be verified that $I_f$ is a $C^2$-isomorphism with inverse 
  \[
  \begin{pmatrix}a\\b\end{pmatrix}\mapsto \begin{pmatrix}\frac{a-g_3(h^{-1})(b))}{g_1(h^{-1}(b))-g_3(h^{-1}(b))}\\ h^{-1}(b)\end{pmatrix}.
  \]
  In order to show the compatibility with the trace, we exhibit for all $i=1,2,3,4$ an explicit affine function 
  $T_i$ such that $I_f \circ T_i = I_i$; this shows that $J_f \circ I_i = T_i$ is affine, as required by Assumption \ref{assumption:straighten}.
  For this, we simply readily check that 
  \begin{align*}
    I_f(t,0) ={}& \begin{pmatrix} g_3(t) \\ h(t) \end{pmatrix} = I_3(t),\\
    I_f(t,1) ={}& \begin{pmatrix} g_1(t) \\ h(t) \end{pmatrix} = I_1(t),\\
    I_f(0,p) ={}& \begin{pmatrix} p (g_1(0) - g_3(0)) + g_3(0) \\ h(0) \end{pmatrix} = \begin{pmatrix} p + 1 \\ \theta_1 \end{pmatrix} = I_2(p),\\
    I_f(1,p) ={}& \begin{pmatrix} p (g_1(1) - g_3(1)) + g_3(1) \\ h(1) \end{pmatrix} = \begin{pmatrix} p (a - 1) + 1 \\ \theta_2 \end{pmatrix} = I_4(p).
  \end{align*}
  so that $T_1(t,p)=(t,0)$, $T_2(t,p)=(t,1)$, $T_3(t,p)=(0,p)$ and $T_4(t,p)=(1,p)$ are simply the natural parametrisation of the edges of $[0,1]^2$.
\end{proof}

\subsection{Quadrilateral with four curved edges}

Let $f \in \FM{2}(\Mh)$ be a face with $4$ edges such that, in a given chart of $f$, the injections of the edges are given by 
\begin{align*}
  I_i(t) ={}& (h_i(t),g_i(t)), \quad i \in\lbrace1,2,3,4\rbrace,
\end{align*}
where $t \in [0,1]$.
Assume moreover that the functions $h_i,g_i$ are $C^2$ and that
\begin{align*}
  (h_1(0),g_1(0)) ={}& (0,0),\quad (h_1(1),g_1(1)) = (1,0),\\
  (h_2(0),g_2(0)) ={}& (1,0),\quad (h_2(1),g_2(1)) = (a,b),\\
  (h_3(0),g_3(0)) ={}& (0,1),\quad (h_3(1),g_3(1)) = (a,b),\\
  (h_4(0),g_4(0)) ={}& (0,0),\quad (h_4(1),g_4(1)) = (0,1).
\end{align*}
This situation is illustrated in Figure \ref{fig:quad}.
\begin{figure}[h]
\centering
  \includegraphics[width=0.5\linewidth]{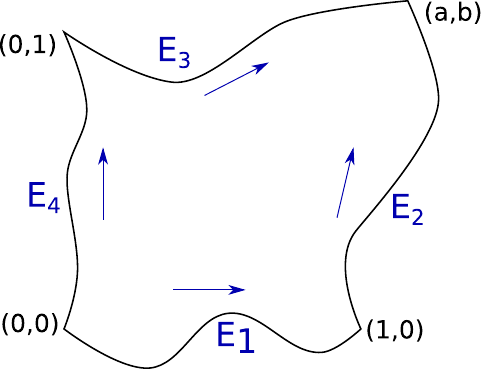}
\caption{$f$ with four curved edges.}
\label{fig:quad}
\end{figure}
Let 
  \[
    I_f(t,p) := \begin{pmatrix}
      \big(p\frac{h_3(t)}{a} + (1 - p)h_1(t)\big)h_2(p) + \big(1 - (p\frac{h_3(t)}{a} + (1 - p)h_1(t))\big)h_4(p)\\
      \big(t\frac{g_2(p)}{b} + (1 - t)g_4(p)\big)g_3(t) + \big(1 - (t\frac{g_2(p)}{b} + (1 - t)g_4(p))\big)g_1(t)
    \end{pmatrix}.
  \]
\begin{lemma}
  If $I_f$ is bijective and $\DET(\JAC I_f) \neq 0$ then
  the function $I_f$
  satisfies Assumption \ref{assumption:straighten}.
\end{lemma}
\begin{proof}
  As in the proof of Lemma \ref{lem:4-edgecell}, we show that the composition of $I_f$ with the natural linear parametrisations of the edges of $[0,1]^2$ give $I_i$ for $i=1,2,3,4$.
  For $E_1$, we have
  \[
    I_f(t,0) 
    =  \begin{pmatrix}
      \big(0 \frac{h_3(t)}{a} + (1 - 0 )h_1(t)\big)\overbrace{h_2(0 )}^{=1} + \big(1 - (0 \frac{h_3(t)}{a} + (1 - 0 )h_1(t))\big)\overbrace{h_4(0 )}^{=0}\\
      \big(t\underbrace{\frac{g_2(0)}{b}}_{=0} + (1 - t)\underbrace{g_4(0 )}_{=0}\big)g_3(t) + \big(1 - (t\underbrace{\frac{g_2(0)}{b}}_{=0} + (1 - t)\underbrace{g_4(0 )}_{=0})\big)g_1(t)
    \end{pmatrix}
    = \begin{pmatrix} h_1(t) \\ g_1(t) \end{pmatrix}.
  \]
  For $E_2$, we have
  \[
    I_f(1,t)
    = \begin{pmatrix}
      \big(t\overbrace{\frac{h_3(1)}{a}}^{=1} + (1 - t)\overbrace{h_1(1)}^{=1}\big)h_2(t) + \big(1 - (t\overbrace{\frac{h_3(1)}{a}}^{=1} + (1 - t)\overbrace{h_1(1)}^{=1}\big)h_4(t)\\
      \big(\frac{g_2(t)}{b} + (1 - 1)g_4(t)\big)\underbrace{g_3(1)}_{=b} + \big(1 - (\frac{g_2(t)}{b} + (1 - 1)g_4(t))\big)\underbrace{g_1(1)}_{=0}
    \end{pmatrix}\\
    = \begin{pmatrix} h_2(t)\\g_2(t) \end{pmatrix}.
   \]
  For $E_3$, we have
  \[
    I_f(t,1)
    = \begin{pmatrix}
      \big(\frac{h_3(t)}{a} + (1 - 1)h_1(t)\big)\overbrace{h_2(1)}^{=a} + \big(1 - (\frac{h_3(t)}{a} + (1 - 1)h_1(t))\big)\overbrace{h_4(1)}^{=0}\\
      \big(t\underbrace{\frac{g_2(1)}{b}}_{=1} + (1 - t)\underbrace{g_4(1)}_{=1}\big)g_3(t) + \big(1 - (t\underbrace{\frac{g_2(1)}{b}}_{=1} + (1 - t)\underbrace{g_4(1)}_{=1})\big)g_1(t)
    \end{pmatrix}
   = \begin{pmatrix} h_3(t) \\ g_3(t) \end{pmatrix}.
   \]
  For $E_4$, we have
  \[
    I_f(0,t)
    = \begin{pmatrix}
      \big(t\overbrace{\frac{h_3(0)}{a}}^{=0} + (1 - t)\overbrace{h_1(0)}^{=0}\big)h_2(t) + \big(1 - (t\overbrace{\frac{h_3(0)}{a}}^{=0} + (1 - t)\overbrace{h_1(0)}^{=0})\big)h_4(t)\\
      \big(0\frac{g_2(t)}{b} + (1 - 0)g_4(t)\big)\underbrace{g_3(0)}_{=1} + \big(1 - (0\frac{g_2(t)}{b} + (1 - 0)g_4(t))\big)\underbrace{g_1(0)}_{=0}
    \end{pmatrix}\\
    = \begin{pmatrix} h_4(t) \\ g_4(t) \end{pmatrix}.
    \qedhere
   \]
\end{proof}

\section{Derivation of the Maxwell equations}\label{sec:tensor.calculus}

In this section, we first recall some generic identities on tensor calculus and then show that, in the absence of shift, \eqref{eq:Maxwell} reduces to \eqref{eq:Maxwell.general}.

Let $g = (g_{\mu\nu})$ denote the metric on the tangent bundle of a generic $n$-dimensional Lorentzian manifold,
and let $(g^{\mu\nu}) := (g_{\mu\nu})^{-1}$.
Assuming $\det g < 0$, the volume form is given in coordinates by 
\begin{equation} \label{eq:volcoord}
  (\vol)_{\mu_1\dots\mu_n} = \sqrt{-\det g}\ \epsilon_{\mu_1\dots\mu_n}
\end{equation}
where $\epsilon_{\mu_1\dots \mu_n}$ denotes the Levi--Civita symbol. 
Given a $1$-form $\alpha := \alpha_\mu \dx{\mu}$ and a vector field $b := b^{\mu} \partial_{\mu}$, 
we define the sharp $\sharp$ and flat $\flat$ operator to raise and lower indices by:
\begin{equation*}
    (\alpha^\sharp)^{\nu} := g^{\nu\mu} \alpha_\mu ,\qquad
    (b^\flat)_{\nu} := g_{\nu\mu} b^{\mu} .
\end{equation*}
We use the Hodge star operator $\star$ to identify $k$-forms with $(n-k)$-forms.
In index notation, it is defined pointwise on the linear basis of $k$-forms by
\begin{equation}\label{eq:Hodge.star.coords}
  \star (\dx{\mu_1} \wedge \dots \wedge \dx{\mu_k}) := \frac{\sqrt{- \det g}}{(n-k)!} \epsilon^{\mu_1\dots \mu_k}{}_{\nu_{k+1}\dots \nu_n} \dx{\nu_{k+1}} \wedge \dots \wedge \dx{\nu_n} 
\end{equation}
where
\[
\epsilon^{\mu_1\dots \mu_k}{}_{\nu_{k+1}\dots \nu_n} := g^{\mu_1 \nu_1} \dots g^{\mu_k \nu_k} \epsilon_{\nu_1\dots \nu_n}.
\]
In general, the product of two $\epsilon$ symbols is given by
\begin{equation} \label{eq:espdet}
  \epsilon^{\mu_1 \dots \mu_n} \epsilon_{\nu_1 \dots \nu_n} = \det(g^{-1}) \delta^{\mu_1 \dots \mu_n}_{\nu_1 \dots \nu_n},
\end{equation}
where $\delta^{\mu_1 \dots \mu_n}_{\nu_1 \dots \nu_n}$ is the generalized Kronecker delta 
(equal to $1$ if $\mu_1 \dots \mu_n$ is an even permutation of $\nu_1 \dots \nu_n$, $-1$ if it is an odd permutation and $0$ otherwise).
The generalized Kronecker delta can also be used to compute wedge products in index notation.
The wedge product of two $1$-forms $\alpha=\alpha_\mu\dx{\mu}$ and $\beta=\beta_\nu\dx{\nu}$ is given by
\begin{equation}\label{eq:wedge2Kronecker}
  \alpha \wedge \beta = \frac{1}{2} \delta^{\mu'\nu'}_{\mu\nu} \alpha_{\mu'} \beta_{\nu'} \dx{\mu}\wedge\dx{\nu} .
\end{equation}

In the rest of this appendix, we specialise to the case of $2+1$ dimensions and adopt the notations \eqref{eq:g.2metric} and \eqref{eq:defn} for the metric and the future pointing unit normal:
\begin{equation*}
  (g_{\mu\nu}) := \begin{pmatrix} -N^2 + \vert\beta\vert_{\gamma}^2 & \beta_j \\ \beta_i & \gamma_{ij} \end{pmatrix}\,,\qquad
   n := (-N\dt)^\sharp.
\end{equation*}
The inverse matrix of $g$ is given by 
\[
g^{-1} = (g^{\mu\nu}) = \begin{pmatrix} -\frac{1}{N^2} & \frac{\beta_j}{N^2} \\ 
\frac{\beta_i}{N^2} & \gamma^{ij} - \frac{\beta_i \beta_j}{N^2} \end{pmatrix},
\]
  and $n = \frac{1}{N} \partial_t - \frac{1}{N}\beta^i \partial_i$.
  The determinant of $g$ and $\gamma$ are related by $\sqrt{-\det g} = N \sqrt{\det \gamma}$.
  \begin{lemma}
    The contraction of the volume form with $n$ is given by
    \begin{equation} \label{eq:invol}
      \contr{n} \vol = \spvol + \dt \wedge \contr{\beta^\spsharp} \spvol
    \end{equation}
    where $\spvol$ is the spatial volume form (associated with $\gamma$).
  \end{lemma}
  \begin{proof}
    An explicit computation gives
    \begin{equation*}
      \begin{aligned}
        \contr{n}\ \vol = \contr{n} \sqrt{-\det{g}}\ \dt\wedge\dx1\wedge\dx2 
        ={}& \frac{\sqrt{-\det g}}{N}\left( \dx1\wedge\dx2 + \dt\wedge\left[ \beta^1\dx2 - \beta^2\dx1 \right] \right)\\
        ={}& \underbrace{\sqrt{\det\gamma}\dx1\wedge\dx2}_{=\spvol} + \dt\wedge\underbrace{\left[\sqrt{\det\gamma}(\beta^1\dx2 - \beta^2\dx1)\right]}_{=\contr{\beta^\spsharp} \spvol}.
        \qedhere
      \end{aligned}
    \end{equation*}
  \end{proof}


We note that, for any $1$-form $V$ and denoting by $\widetilde{V}$ the restriction of $V$ to a time slice, we have
    \begin{equation} \label{eq:ndotp}
      V^\sharp = \widetilde{V}^{\spsharp} - n V(n).
    \end{equation}
  We readily infer from \eqref{eq:ndotp} and the identity $\contr{n}\contr{n} = 0$ that
  \begin{equation} \label{eq:iEin} 
    \contr{V^\sharp} \contr{n} = \contr{\widetilde{V}^\spsharp} \contr{n}.
  \end{equation} 
  
  We now turn to the Maxwell equations \eqref{eq:Maxwell} and recall the link \eqref{eq:defE}--\eqref{eq:defB} between the electromagnetic field (2-form) $F$, the electric field (1-form) $E$ and the scalar magnetic field $B$. We first consider the question of recovering the electromagnetic field tensor and its Hodge dual from the electric and magnetic fields. 
  
\begin{lemma}
  We can recover $F$ from $E$ and $B$ by the relation 
  \begin{equation} \label{eq:FEB}
  F = n^\flat \wedge E + B\ \contr{n} \vol .
  \end{equation}
\end{lemma}
\begin{proof}
  Noting from \eqref{eq:defB} that $B = \frac{\sqrt{-\det g}}{2} \epsilon^{\alpha\beta}{}_\gamma F_{\alpha\beta} n^\gamma$, we can, with the help of \eqref{eq:volcoord}, express $B\ \contr{n} \vol$ as
  \begin{align}
      (B\ \contr{n} \vol)_{\mu\nu} 
      ={}& \frac{-\det g}{2} \epsilon^{\alpha\beta}{}_\gamma F_{\alpha\beta} n^\gamma n^\lambda \epsilon_{\lambda\mu\nu} \nonumber\\
      ={}& - \frac{\det g}{2} \epsilon^{\alpha\beta\gamma} \epsilon_{\lambda\mu\nu} n_\gamma n^\lambda F_{\alpha\beta} \nonumber\\
      \overset{\eqref{eq:espdet}}{=}{}& - \frac12 \delta^{\alpha\beta\gamma}_{\lambda\mu\nu}n_\gamma n^\lambda F_{\alpha\beta} \nonumber\\
      ={}& -\frac12 \left( 
      \delta^\alpha_\lambda \delta^{\beta\gamma}_{\mu\nu}
    - \delta^\beta_\lambda \delta^{\alpha\gamma}_{\mu\nu}
      \right) n_\gamma n^\lambda F_{\alpha\beta}
      - \frac12 \delta^\gamma_\lambda \delta^{\alpha\beta}_{\mu\nu}n_\gamma n^\lambda F_{\alpha\beta}.
     \label{eq:prFEB.1}
  \end{align}
  Next, from the definition \eqref{eq:defE} of $E$, we observe that 
  $E_\alpha = -n^\beta F_{\beta\alpha} = \delta^\beta_\lambda n^\lambda F_{\alpha\beta}$. This allows us to express the following terms which appear in \eqref{eq:prFEB.1} in terms of $E\wedge n^\flat$ via the calculations:
  \begin{equation}\label{eq:prFEB.2}
    \begin{aligned}
      \delta^\alpha_\lambda n^\lambda F_{\alpha\beta}\delta^{\beta\gamma}_{\mu\nu} n_\gamma 
      ={}& -E_\beta n_\gamma \delta^{\beta\gamma}_{\mu\nu} \overset{\eqref{eq:wedge2Kronecker}}= -(E\wedge n^\flat)_{\mu\nu},\\
      \delta^\beta_\lambda n^\lambda F_{\alpha\beta}\delta^{\alpha\gamma}_{\mu\nu} n_\gamma 
      ={}& E_\alpha n_\gamma \delta^{\alpha\gamma}_{\mu\nu} \overset{\eqref{eq:wedge2Kronecker}}= (E\wedge n^\flat)_{\mu\nu}.
    \end{aligned}
  \end{equation}
  Moreover, since $n^2 = -1$ and $F$ antisymmetric, we have
  \begin{equation}\label{eq:prFEB.3}
    \delta^\gamma_\lambda n_\gamma n^\lambda \delta^{\alpha\beta}_{\mu\nu} F_{\alpha\beta}
    = 2 n_\lambda n^\lambda F_{\mu\nu} = -2 F_{\mu\nu}.
  \end{equation}
  Plugging \eqref{eq:prFEB.2} and \eqref{eq:prFEB.3} into \eqref{eq:prFEB.1} gives
  $ B\ \contr{n}\vol = E \wedge n^\flat + F$. 
\end{proof}

\begin{lemma}
  The Hodge star of $F$ is given by
  \begin{equation} \label{eq:starFEB}
    \star F = \contr{\spE^\spsharp} \contr{n} \vol - n^\flat B.
  \end{equation}
\end{lemma}
\begin{proof}
  To establish the stated formula, we apply the Hodge star operator to each term of \eqref{eq:FEB} individually.
  For the term involving $E$, we have: 
  \begin{equation*}
    \begin{aligned}
      (\star(n^\flat \wedge E))_{\mu} 
      \overset{\eqref{eq:Hodge.star.coords},\eqref{eq:wedge2Kronecker}}{=}{}& \frac{\sqrt{-\det g}}{2} \epsilon^{\alpha\beta}{}_\mu \delta^{\alpha'\beta'}_{\alpha\beta} n_{\alpha'} E_{\beta'} \\
      ={}& \frac{\sqrt{-\det g}}{2} \epsilon_{\alpha\beta\mu} (n^\alpha E^\beta - n^\beta E^\alpha) \\
      ={}& \sqrt{-\det g} \epsilon_{\alpha\beta\mu} n^\alpha E^\beta \\
      ={}& \bigl(\contr{E^\sharp}\contr{n} \vol\bigr)_\mu
      \overset{\eqref{eq:iEin}}{=} \bigl(\contr{\spE^\spsharp}\contr{n} \vol\bigr)_\mu,
    \end{aligned}
  \end{equation*}
  while for the term involving $B$, we have:
  \begin{equation*}
    \begin{aligned}
      (\star B\ \contr{n} \vol)_\mu 
      ={}& \frac{\sqrt{-\det g}}{2} \epsilon^{\alpha\beta}{}_\mu B n^\lambda \sqrt{-\det g} \epsilon{}_{\lambda\alpha\beta} \\
      ={}& -\frac{\det g}{2} g_{\mu\mu'} \epsilon^{\alpha\beta\mu'} \epsilon_{\alpha\beta\lambda} n^\lambda B\\
      ={}& - \frac12 g_{\mu\mu'} \delta^{\alpha\beta\mu'}_{\alpha\beta\lambda} n^\lambda B \\
      ={}& - \frac12 2! g_{\mu\mu'} \delta^{\mu'}{}_{\lambda} n^\lambda B \\
      ={}& - n_\mu B .
    \end{aligned}
  \end{equation*}
Adding the above two expressions yields the desired formula.
\end{proof}

To remove some technicality in the following derivations, we now restrict our attention to foliations for which $\beta=0$. In the next lemma, we derive $2+1$ decompositions for $\DIFF F$ and $\DIFF\star F$, which will be used below to perform a $2+1$ decomposition of Maxwell's equations.
\begin{lemma}
  If $\beta = 0$, then 
  \begin{align}
    \DIFF F ={}& \dt \wedge \left[ \spdiff (N \spE) + \Lie_{\partial_t}
     (B\  \spvol) \right] , \label{eq:diffF} \\
    \DIFF \star F ={}& \spdiff\spstar \spE + \dt \wedge \left[ \Lie_{\partial_t} (\spstar \spE) - \spdiff (NB) \right]. \label{eq:diffstF}
  \end{align}
\end{lemma}
\begin{proof}
  We infer from \eqref{eq:invol} with $\beta = 0$ that 
  $\contr{n} \vol = \spvol$.
  Replacing $n^\flat$ by its definition \eqref{eq:defn} in \eqref{eq:FEB} gives
  \begin{equation} \label{eq:diffFcoord}
    F = - N \dt \wedge E + B\ \spvol = - N \dt \wedge \spE + B\ \spvol,
  \end{equation}
  where in deriving the second equality, we used the fact that $E = \spE + E_0\dt$.
  Next, we apply the exterior derivative to each term in \eqref{eq:diffFcoord} to obtain:
  \begin{equation*}
    \DIFF (- \dt \wedge N \spE) = \dt \wedge \DIFF (N \spE) = \dt \wedge \spdiff (N \spE),
  \end{equation*}
  \begin{equation*}
    \DIFF (B \ \spvol) = 
    \partial_t B \ \dt \wedge \spvol + \cancel{\spdiff B \wedge \spvol} + B \ \dt \wedge \partial_t \spvol 
    + \cancel{B\ \spdiff \spvol}.
  \end{equation*}
  Since $\partial_t (B\spvol)=\Lie_{\partial_t}(B\spvol)$, 
  \eqref{eq:diffF} follows from adding the above two expressions.
  
  Noting that
  \begin{equation*}
    \begin{aligned}
      (\contr{\spE^\spsharp} \spvol)_{j}
      ={}& \sqrt{\det \gamma} \spE^i \epsilon_{ij}
      ={} \sqrt{\det \gamma} \epsilon^{i}{}_j \spE_i
      ={} (\spstar \spE)_j,
    \end{aligned}
  \end{equation*}
  where the indices are raised with the spatial metric $\gamma$, we can express \eqref{eq:starFEB} as $\star F = \spstar \spE + BN \dt$. Formula \eqref{eq:diffstF} is then readily obtained by applying the exterior derivative to this expression and employing the identity $\DIFF\spstar\spE =  \spdiff \spstar\spE+\dt\wedge \Lie_{\partial_t}(\spstar\spE)$. 
\end{proof}

The last $2+1$ decomposition we need is for the electric $3$-current $\ul{j}$. We recall that the electric charge density is $\rho := - \ul{j}(n)$
and the electric current density is $J := N(\ul{j} - n^\flat \rho)$. 
\begin{lemma}
  The electric $3$-current appearing in \eqref{eq:Maxwell.2} is decomposed as follows
  \begin{equation} \label{eq:jdec}
    \star \ul{j} = \rho\ \spvol - \dt \wedge \spstar \widetilde{J}.
  \end{equation}
\end{lemma}
\begin{proof}
  From the formula \eqref{eq:Hodge.star.coords} for the Hodge star operator, we have
  \begin{equation*}
      (\star\ul{j})_{\mu\nu} 
      = \sqrt{-\det g}\ \epsilon^\alpha{}_{\mu\nu} \ul{j}_\alpha \\
      = \sqrt{-\det g}\ \left( \ul{j}^0 \delta^{12}_{\mu\nu} - (\delta^0_\mu \delta^{12}_{\alpha\nu} - \delta^0_\nu\delta^{12}_{\alpha\mu}) \ul{j}^\alpha \right),
  \end{equation*}
  while we infer from \eqref{eq:ndotp} and the definition of $\rho$, recalling that we have set $\beta=0$ for simplicity, that 
  \[
    \sqrt{-\det g}\ \ul{j}^0 = \frac{\sqrt{-\det g}}{N} \rho = \sqrt{\det \gamma}\ \rho,
  \]
  and thus $\sqrt{-\det g}\ \ul{j}^0 \delta^{12}_{\mu\nu} = (\rho \spvol)_{\mu\nu}$.
  Moreover, we have 
  \[
    \sqrt{- \det g}\ \delta^{12}_{\alpha\mu} \ul{j}^\alpha 
    = N \ul{j}^\alpha \sqrt{\det \gamma}\ \epsilon_{\alpha\mu}
    = (\contr{N \ul{j}^\sharp} \spvol)_\mu,
  \]
  and hence 
  \[
    \sqrt{- \det g}\ (\delta^0_\mu \delta^{12}_{\alpha\nu} - \delta^0_\nu\delta^{12}_{\alpha\mu}) \ul{j}^\alpha
    = (\dt \wedge \contr{N\ul{j}^\sharp} \spvol)_{\mu\nu}.
  \]
  We also have
  \[
    \dt \wedge \contr{N\ul{j}^\sharp } \spvol 
    ={} \dt\wedge \contr{J^\sharp} \spvol + \dt \wedge \contr{N\rho n}\spvol 
    ={} \dt\wedge \contr{J^\sharp} \spvol,
  \]
  where the last equality holds by \eqref{eq:invol} and the assumption $\beta=0$, which show that $\contr{N\rho n^\flat}\spvol=N\rho \contr{n} \contr{n}\vol=0$.
  We also observe from \eqref{eq:ndotp} with $V = J$, noticing that
  $J(n) = N(\ul{j}(n) - n^2 \rho) = N(\ul{j}(n) - \ul{j}(n)) = 0$, that $J^\sharp = \widetilde{J}^\spsharp$,
  and consequently, that $\contr{J^\sharp} \spvol = \spstar \widetilde{J}$. Putting everything together, we deduce \eqref{eq:jdec}.
\end{proof}

It is then easy to see that the 2+1 formulation of the Maxwell equations \eqref{eq:Maxwell} is \eqref{eq:Maxwell.general}: it suffices to plug
\eqref{eq:diffF}, \eqref{eq:diffstF} and \eqref{eq:jdec} into \eqref{eq:Maxwell}, and to notice that $\spstar^{-1}\spdiff\spstar \spE = - \spdelta \spE$, since $\spE$ is a $1$-form, 
  and that $\spstar^{-1}\spdiff(NB) = \spstar^{-1}\spdiff(\spstar \spvol NB) = \spdelta (NB\spvol)$.

\section{Exact solutions}\label{sec:exact.solutions}

We consider two different solutions for the numerical experiment on the sphere: 
one that is only $C^0$ and piecewise smooth, while the other is smooth. 

\subsection{$C^0$ solution on the sphere}\label{sec:exact.solutions.C0}

For the purpose of describing the continuous solution, we introduce stereographic coordinates by considering a unit sphere with an atlas consisting of the north and south stereographic projections, and endowed with the induced metric from $\Real^3$.
In each map, this metric is given by 
\[
  \gamma_{ij} = \lambda \begin{pmatrix} 1 & 0 \\ 0 & 1 \end{pmatrix}, 
    \quad \lambda := \frac{4}{(1 + X^2 + Y^2)^2},
\]
where $X,Y$ are the coordinates in the map.
Writing $\widetilde{E} = E_X \DIFF X +  E_Y \DIFF Y$, we have 
\[
  \spdiff \widetilde{E} = \frac{1}{\lambda} (\partial_X E_Y - \partial_Y E_X) \spvol, \quad
  -\spdelta \widetilde{E} = \frac{1}{\lambda} (\partial_X E_X + \partial_Y E_Y), \quad
  \spdelta (B\spvol) = \partial_Y B \DIFF X - \partial_X B \DIFF Y .
\]

We infer the following solution featuring a non-zero source.
Although it is continuous, it is not $C^1$.
Its expression in the north map is given by
\begin{equation}
\begin{aligned}
  B' ={}&
  \big[ (X^2 + Y^2 - 1)\cos(t) + X^2 + Y^2 + 1 - 2 X \sin(t)\big] \DIFF X \wedge \DIFF Y, \\
  E ={}&
  \frac{Y}{4} \big[ (2 - X^2 - Y^2)\sin(t) - 2 X \cos(t) \big] \DIFF X\\
  & +\frac{X}{4} \big[ (X^2 + Y^2 - 2)\sin(t) + (3 X^2 + Y^2 - 3) \cos(t) \big] \DIFF Y , \\
  \rho ={}& 
0, \\
  J ={}&
  \Big[Y \left( \frac{3}{2} (X^2 + Y^2)^2 (1 + \cos(t)) + (X^2 + Y^2)\left(3 + \frac{5}{4} \cos(t)\right) + \frac{1}{2} - \cos(t) \right) \\
  &\quad - X Y \left(2 X^2 + 2 Y^2 + \frac{5}{2}\right) \sin(t) \Big] \DIFF X\\
  &+\Big[-X \left( \frac{3}{2} (X^2 + Y^2)^2 (1 + \cos(t)) + (X^2 + Y^2)\left(3 + \frac{5}{4} \cos(t)\right) + \frac12 - \cos(t) \right) \\
  &\quad + \frac14 (10 X^4 + 12 X^2Y^2 + 15 X^2 + 2 Y^4 + 5 Y^2 - 1 ) \sin(t)\Big] \DIFF Y
\end{aligned}
\label{eq:sol.sp.r1}
\end{equation}
and in the south map by
\begin{equation}
\begin{aligned}
  B' ={}&
  \big[ (X^2 + Y^2 - 1)\cos(t) - X^2 - Y^2 - 1 + 2 X \sin(t)\big] \DIFF X \wedge \DIFF Y, \\
  E ={}&
  \frac{Y}{4} \left[ (2 - X^2 - Y^2)\sin(t) + 2 X \cos(t) \right] \DIFF X\\
  & +\frac{X}{4} \left[ (X^2 + Y^2 - 2)\sin(t) - (3 X^2 + Y^2 - 3) \cos(t) \right] \DIFF Y , \\
  \rho ={}& 0,\\
  J ={}&
  \Big[-Y \left( \frac32 (X^2 + Y^2)^2 (1 - \cos(t)) + (X^2 + Y^2) \left(3 - \frac54 \cos(t)\right) + \frac12 + \cos(t) \right) \\
  &\quad + X Y \left(2 X^2 + 2 Y^2 + \frac52\right) \sin(t) \Big] \DIFF X\\
  &+\Big[X \left( \frac32 (X^2 + Y^2)^2 (1 - \cos(t)) + (X^2 + Y^2)\left(3 - \frac54 \cos(t)\right) + \frac12+ \cos(t) \right) \\
  &\quad - \frac14 (10 X^4 + 12 X^2Y^2 + 15 X^2 + 2 Y^4 + 5 Y^2 - 1 ) \sin(t)\Big] \DIFF Y.
\end{aligned}
\label{eq:sol.sp.r2}
\end{equation}

\subsection{Smooth solution}\label{sec:exact.solutions.smooth}

Introducing spherical coordinates $(x^1,x^2)=(\phi,\theta)$ on the unit sphere $S^2$ via
\begin{equation*}
(x,y,z)= \bigl(\cos(\theta)\sin(\phi),\sin(\theta)\sin(\phi),\cos(\phi)\bigr), \quad 0<\theta<2\pi,\; 0<\phi<\pi,
\end{equation*}
the induced metric from $\mathbb{R}^3$ on $S^2$ can be expressed as
\begin{equation}\label{S2-metric}
\gamma = \DIFF \phi \otimes \DIFF \phi + \sin^2(\phi)\,\DIFF \theta \otimes \DIFF \theta,
\end{equation}
or equivalently in matrix form
\begin{equation*}
(\gamma_{ij})=\begin{pmatrix} 1 & 0 \\
0 & \sin^2(\phi)\end{pmatrix} \quad \text{and}\quad (\gamma^{ij})=\begin{pmatrix} 1 & 0 \\
0 & \displaystyle \frac{1}{\sin^2(\phi)}\end{pmatrix}.
\end{equation*}

We now consider the vacuum Maxwell equations, that is, \eqref{eq:Maxwell.Stationary} with $\rho=0$, $\widetilde{J}=0$ and $c=1$:
\begin{subequations}
\begin{align}
  \spdiff \spE &= - \partial_t B', \label{eq:vac.Maxwell.B}\\
  \spdelta \spE &=0,  \label{eq:vac.Maxwell.constraint}\\
  \spdelta B' &= \partial_t \spE. \label{eq:vac.Maxwell.E}
\end{align}
\end{subequations}
Noting from \eqref{S2-metric} that the metric $\gamma$ is invariant under the transformation $\theta \longmapsto \theta + \text{const}$,
we look for solutions of these equations that are also invariant under this transformation by imposing the following ansatz:
\begin{equation}\label{EB-ansatz}
\begin{aligned}
\spE  &= a(t,\phi) \, \DIFF \theta,\\
B' &= b(t,\phi)\spvol=  b(t,\phi)\sin(\phi)\, \DIFF \phi \wedge \DIFF \theta.
\end{aligned}
\end{equation}
To see that this ansatz leads to a consistent reduction of the vacuum Maxwell equation, we first observe that the Gauss constraint \eqref{eq:vac.Maxwell.constraint} is automatically satisfied:
\begin{equation*}
\spdelta \spE = -\spstar \spdiff ( a\spstar \, \DIFF \theta)=\spstar \spdiff\biggl(\frac{a(t,\phi)}{\sin(\phi)}\,\DIFF\phi\biggr)=\spstar \biggl(\partial_\phi\biggl(\frac{a(t,\phi)}{\sin(\phi)}\biggr)\, \DIFF \phi \wedge \DIFF \phi \biggr) = 0. 
\end{equation*}
Next, we calculate:
\begin{gather*}
\spdiff \spE = \partial_\phi a \, \DIFF \phi \wedge \DIFF \theta
\intertext{and}
\spdelta B' = - \spstar \spdiff\bigl( b\spstar\spvol\bigr)=- \spstar \spdiff b
=- \partial_\phi b\,\spstar \DIFF\phi=-\sin(\phi)\partial_\phi b \DIFF\theta.
\end{gather*}
From these calculations, it is clear that the anstaz \eqref{EB-ansatz} will satisfy the vacuum Maxwell equations \eqref{eq:vac.Maxwell.B}--\eqref{eq:vac.Maxwell.E} provided $a$ and $b$ satisfy 
\begin{align}
\partial_t a & = -\sin(\phi)\partial_\phi b, \label{a-ev-A}\\
\partial_t b &= -\frac{1}{\sin(\phi)}\partial_\phi a. \label{b-ev-A} 
\end{align}
Together, these equations imply that $b$ satisfies the wave equation
\begin{equation} \label{b-ev-B}
\partial_t^2 b = \frac{1}{\sin(\phi)}\partial_\phi\bigl(\sin(\phi)\partial_\phi b\bigr).
\end{equation}
Treating this equation as our primary equation to solve, it is then not difficult to verify that, given a solution of \eqref{b-ev-B}, we can recover the function $a$ by integrating \eqref{a-ev-A} in time to get
\begin{equation} \label{a-ev-B}
a(t,\phi) = - \sin(\phi)\partial_\phi \int_0^t b(s,\phi)\, ds + a_0(\phi),
\end{equation}
where $a_0(\phi)$ is for now an arbitrary function of $\phi$. 
On the other hand, integrating the wave equation \eqref{b-ev-B} in time gives
\begin{align*}
\partial_t b(t,\phi) &= \frac{1}{\sin(\phi)}\partial_\phi\biggl(\sin(\phi)\partial_\phi \int_0^t b(s,\phi)\, ds \biggr)+\partial_t b(0,\phi) \\
&= - \frac{1}{\sin(\phi)} \partial_\phi a(t,\phi) + \frac{1}{\sin(\phi)}a_0'(\phi) + \partial_t b(0,\phi). 
\end{align*}
From this, we conclude for any solution $b(t,\phi)$ of the wave equation \eqref{b-ev-B}, the pair $\{a(t,\phi),b(t,\phi)\}$, where $a(t,\phi)$ is determined from $b(t,\phi)$ by the formula \eqref{a-ev-B},
will solve \eqref{a-ev-A}--\eqref{b-ev-A}, and hence determine a solution of the vacuum Maxwell equations via \eqref{EB-ansatz},  provided the free function $a_0(\phi)$ is chosen to satisfy
\begin{equation*} 
a_0'(\phi) =-\sin(\phi)\partial_t b(0,\phi).
\end{equation*}

Now, to find exact solutions to the wave equation \eqref{b-ev-B}, we recall the Legendre polynomials $P_\ell(x)$, $\ell\in \mathbb{N}_0$, which solve the differential equation
\begin{equation*}
\frac{d\;}{dx}\biggl((1-x^2)\frac{d\;}{dx}\biggr)P_\ell(x) = -\ell (\ell+1)P_{\ell}(x), \quad x\in [-1,1].
\end{equation*}
We also recall that these polynomial can be computed directly from Rodrigues' formula:
\begin{equation*}
P_\ell(x) = \frac{1}{2^\ell \ell!}\frac{d^\ell\;}{dx^\ell}(x^2-1)^\ell.
\end{equation*}
It is well known that the functions $P_\ell(\cos(\phi))$, which are the  spherical harmonics $Y^0_\ell(\phi,\theta)$ on $S^2$, satisfy the eigenvalue problem
\begin{equation*}
\frac{1}{\sin(\phi)}\partial_\phi\Bigl(\sin(\phi)\partial_\phi P_\ell(\cos(\phi))\Bigr)=-\ell(\ell+1)P_\ell(\cos(\phi)).
\end{equation*}
With the help of this identity, it is then straightforward to verify, for any $c_1,c_2\in \mathbb{R}$ and $\ell\in \mathbb{N}_0$, that the functions
\begin{equation}
b(t,\phi)=\Bigl(c_1\cos\bigl(\sqrt{\ell(\ell+1)}t\bigr)+c_2\sin\bigl(\sqrt{\ell(\ell+1)}t\bigr)\Bigr)P_\ell(\cos(\phi))
\label{eq:b}
\end{equation}
satisfy the wave equation \eqref{b-ev-B}.

The smooth solution considered in Section \ref{sec:results.smooth} is \eqref{EB-ansatz} for $a$ and $b$ respectively given by \eqref{a-ev-B} and \eqref{eq:b} for $\ell=1$.
Expressed in the stereographic coordinates, 
this smooth solution takes the same expression in both the north and south map:
\begin{equation}
\begin{aligned}
  B' ={}&
  \cos(\sqrt{2}t) \frac{1-X^2-Y^2}{1+X^2+Y^2} \frac{4}{\left( 1+X^2+Y^2 \right)^2} \DIFF X \wedge \DIFF Y, \\
  E ={}&
  -Y \biggl( \frac{\sin(\sqrt{2}t)}{\sqrt{2}} \frac{4}{\left( 1+X^2+Y^2 \right)^2}  \biggr) \DIFF X
   +X \biggl( \frac{\sin(\sqrt{2}t)}{\sqrt{2}} \frac{4}{\left( 1+X^2+Y^2 \right)^2} \biggr) \DIFF Y , \\
  \rho ={}& 
0, \\
  J ={}&
  0.
\end{aligned}
\label{eq:sol.sp.s}
\end{equation}

\section*{Acknowledgements}
Funded by the European Union (ERC, NEMESIS, No. 101115663).
Views and opinions expressed are however those of the author(s) only and do not necessarily reflect those of the European Union or the European Research Council Executive Agency. Neither the European Union nor the granting authority can be held responsible for them.

\printbibliography

\end{document}